\documentclass[final,3p,times,sort&compress]{elsarticle}

\usepackage{mathrsfs}        
\usepackage{subfigure}
\usepackage{graphicx}        
\usepackage{amssymb}         
\usepackage{amsmath}         
\usepackage{enumerate}       
\usepackage[colorlinks,
            linkcolor=red,
            anchorcolor=blue,
            citecolor=green
            ]{hyperref}
\usepackage{caption}

\usepackage{amsthm}          
\newtheorem{definition}{Definition}
\newtheorem{theorem}{Theorem}

\newtheorem{lemma}{Lemma}
\newtheorem{corollary}{Corollary}

\DeclareCaptionLabelSeparator{dotmark}{. }
\begin{document}

\begin{frontmatter}

\phantomsection
\addcontentsline{toc}{chapter}{On some analytic properties of nabla tempered fractional calculus}
\title{On some analytic properties of nabla tempered fractional calculus}


\author[seu]{Yiheng Wei$^*$}
\author[nau]{Linlin Zhao}
\author[seu]{Xuan Zhao}
\author[seu,yu]{Jinde Cao}
\address[seu]{School of Mathematics, Southeast University, Nanjing, 211189, China}
\address[nau]{School of Business, Nanjing Audit University, Nanjing 211815, China}
\address[yu]{Yonsei Frontier Lab, Yonsei University, Seoul 03722, South Korea}

\cortext[cor1]{Corresponding author. E-mail: neudawei@seu.edu.cn.}

\begin{abstract}
Despite many applications regarding fractional calculus have been reported in literature, it is still unknown how to model some practical process. One major challenge in solving such a problem is that, the nonlocal property is needed while the infinite memory is undesired. Under this context, a new kind nabla fractional calculus accompanied by a tempered function is formulated. However, many properties of such fractional calculus needed to be discovered. From this, this paper gives particular emphasis to the topic. Some remarkable properties like the equivalence relation, the nabla Taylor formula, and the nabla Laplace transform for such nabla fractional calculus are developed and analyzed. It is believed that this work greatly enriches the mathematical theory of nabla tempered fractional calculus and provides high value and huge potential for further applications.
\end{abstract}


\end{frontmatter}


\section{Introduction}\label{Section1}

Fractional calculus, as a powerful toolbox, has attracted an increasing attention from scholars in the past decades \cite{Ceresa:2022JMAA,Arora:2022Neu,Suzuki:2022JPNM,Reed:2022ARC}. Due to its nonlocal property, fractional calculus play a key role in diverse application of science and engineering. Also, nonlocal property brings many difficulties and challenges in modeling some special objects. From this, the past decade has witnessed a significant progress on tempered fractional calculus by introducing an extra weighted function. Notably, tempered fractional calculus has many merits, and henceforth many researchers have deployed themselves to explore valuables results. A huge work has been done for this subject \cite{Sabzikar:2015JCP,Almeida:2019JCAM,Fernandez:2020JCAM,Mali:2022MMA}, which makes a positive and profound impact.

Compared with the continuous time case, the research in discrete time case is just in its infancy and only few work has been reported preliminarily. I have to admit that the discrete time case performs better in computing, storage, transport than the continuous time case and it has greater potential in the digital era. The memory effect of delta tempered fractional calculus was investigated and applied to image processing \cite{Abdeljawad:2020Optik}. The tempered fractional derivative on an isolated time scale was defined and a new method was presented based on the time scale theory for numerical discretization in \cite{Fu:2021Fractals}. A general definition for nabla discrete time tempered fractional calculus was presented in \cite{Ferreira:2021NODY}. The tempered function was chosen as the nonzero case instead of the discrete exponential function, which greatly enrich the potential of the tempered fractional calculus. Though the study on discrete time tempered fractional calculus is still in sufficient, a proliferation of results reported on discrete time fractional calculus \cite{Cheng:2011Book,Ostalczyk:2015Book,Goodrich:2015Book,Ferreira:2022Book} could give us a lot of helpful inspiration.

The basic arithmetic and equivalence relations of fractional difference and fractional sum were discussed in \cite{Cheng:2011Book,Ostalczyk:2015Book,Goodrich:2015Book,Wei:2019JCND}. The continuity on the order of fractional difference was studied in \cite{Cheng:2011Book,Goodrich:2015Book,Wei:2019CNSNSa}. However, the Gr\"{u}nwald--Letnikov difference is not equivalent to the classical integer order case for integral order, which means that the existing results need extra conditions. The nabla Taylor series was first investigated systematically in \cite{Wei:2019CNSNSa}, while some of the existing results requires quite strict conditions and the nabla Taylor formula expanded at the current time was not discussed enough. A brief review was made on the nabla Laplace transform \cite{Wei:2019FDTA} and some interesting properties still could be built with the existing results. It is worth noting that some similar properties like the classical case can be checked for the tempered case and some seminal properties for the tempered case need to be investigated. Bearing this in mind, some fundamental properties will be developed carefully. However, it is not an easy task to complete this task, since the introduction of the tempered function brings some unexpected difficulty and damage some accustomed properties.

The main contribution of this work lies in the following aspects. \textsf{(i) The relationship between nabla tempered fractional difference/sum is explored. (ii) The nabla Taylor formula/series representation of nabla tempered fractional calculus is developed. (iii) The nabla Laplace transform for nabla tempered fractional difference/sum is derived.} Hopefully, these distinguished properties could be helpful for other researchers understand and apply such fractional calculus.

The remainder of this paper is summarized as follows. In Section \ref{Section2}, the basic concept and properties of nabla tempered fractional calculus are presented here. In Section \ref{Section3}, many foundational properties are developed for such fractional calculus. In Section \ref{Section4}, some concluding remarks are provided to end this paper.

\section{Preliminaries}\label{Section2}
In this section, the definitions and properties for nabla tempered fractional calculus are introduced.

For $x: \mathbb{N}_{a+1-n} \to \mathbb{R}$, its $n$-th nabla difference is defined by
\begin{equation}\label{Eq2.1}
{\textstyle \nabla^{n} x(k):=\sum_{i=0}^{n}(-1)^{i}\big(\begin{smallmatrix}
n \\
i
\end{smallmatrix}\big) x(k-i),}
\end{equation}
where $n \in \mathbb{Z}_{+}$, $k \in \mathbb{N}_{a+1}:=\{a+1,a+2,\cdots\}$, $a \in \mathbb{R}$, $\left(\begin{smallmatrix}p \\ q\end{smallmatrix}\right):=\frac{\Gamma(p+1)}{\Gamma(q+1) \Gamma(p-q+1)}$ is the generalized binomial coefficient and $\Gamma(\cdot)$ is the Gamma function.

For $x: \mathbb{N}_{a+1} \to \mathbb{R}$, its $\alpha$-th Gr\"{u}nwald--Letnikov difference/sum is defined by \cite{Ostalczyk:2015Book,Atici:2021FDC}
\begin{equation}\label{Eq2.2}
{\textstyle {}_{a}^{\rm G} \nabla_{k}^{\alpha} x(k):=\sum_{i=0}^{k-a-1}(-1)^{i}\big(\begin{smallmatrix}
\alpha \\
i
\end{smallmatrix}\big) x(k-i),}
\end{equation}
where $\alpha \in \mathbb{R}$, $k \in \mathbb{N}_{a+1}$ and $a \in \mathbb{R}$. When $\alpha>0$, ${ }_{a}^{\rm G} \nabla_{k}^{\alpha} x(k)$ represents the difference operation. When $\alpha<0$, ${}_{a}^{\rm G} \nabla_{k}^{\alpha} x(k)$ represents the sum operation including the fractional order case and the integer order case. Specially, ${}_{a}^{\rm G} \nabla_{k}^{0} x(k)=x(k)$. Even though $\alpha=n \in \mathbb{Z}_{+}$, ${ }_{a}^{\rm G} \nabla_{k}^{\alpha} x(k) \not \equiv \nabla^{n} x(k)$ for all $k \in \mathbb{N}_{a+1}$.

Defining the rising function $p\overline {^q} := \frac{\Gamma(p+q)}{\Gamma(p)}$, $p\in\mathbb{Z}_+$, $q\in\mathbb{R}$, (\ref{Eq2.2}) can rewritten as
\begin{equation}\label{Eq2.3}
{\textstyle
\begin{array}{rl}
{ }_{a}^{\mathrm{G}} \nabla_{k}^{\alpha} x(k)=&\hspace{-6pt}\sum_{i=0}^{k-a-1}(-1)^{i} \frac{\Gamma(1+\alpha)}{\Gamma(i+1) \Gamma(1+\alpha-i)} x(k-i) \\
=&\hspace{-6pt}\sum_{i=0}^{k-a-1} \frac{\Gamma(i-\alpha)}{\Gamma(i+1) \Gamma(-\alpha)} x(k-i) \\
=&\hspace{-6pt}\sum_{i=0}^{k-a-1} \frac{(i+1)^{\overline{-\alpha-1}}}{\Gamma(-\alpha)} x(k-i) \\
=&\hspace{-6pt}\sum_{i=a+1}^{k} \frac{(k-i+1)^{\overline{-\alpha-1}}}{\Gamma(-\alpha)} x(i),
\end{array}}
\end{equation}
where $\Gamma(\theta) \Gamma(1-\theta)=\frac{\pi}{\sin (\pi \theta)}$, $\theta \in \mathbb{R}$ is adopted.

From the previous definitions, the $\alpha$-th Riemann--Liouville fractional difference and Caputo fractional difference for $x: \mathbb{N}_{a+1-n} \to \mathbb{R}$, $\alpha \in(n-1, n)$, $n \in \mathbb{Z}_{+}$, $k \in \mathbb{N}_{a+1}$ and $a \in \mathbb{R}$ are defined by \cite{Goodrich:2015Book}
\begin{equation}\label{Eq2.4}
{\textstyle
{}_{a}^{\mathrm{R}} \nabla_{k}^{\alpha} x(k):=\nabla^n{}_{a}^{\rm G} \nabla_{k}^{\alpha-n} x(k), }
\end{equation}
\begin{equation}\label{Eq2.5}
{\textstyle
{ }_{a}^{\mathrm{C}} \nabla_{k}^{\alpha} x(k):={ }_{a}^{\rm G} \nabla_{k}^{\alpha-n} \nabla^{n} x(k) .}
\end{equation}
On this basis, the following properties hold.

\begin{lemma}\label{Lemma2.1}\cite{Goodrich:2015Book,Wei:2019JCND}
For any function $x: \mathbb{N}_{a+1-n} \to \mathbb{R}$, $n \in \mathbb{Z}_{+}$, $a \in \mathbb{R}$, one has
\begin{equation}\label{Eq2.6}
{\textstyle
{ }_{a}^{\rm R} \nabla_{k}^\alpha{}_ a^{ \rm G} \nabla_{k}^{-\alpha} x(k)={ }_{a}^{\rm C} \nabla_{k}^{\alpha}{}_a^{ \rm G} \nabla_{k}^{-\alpha} x(k)=x(k),}
\end{equation}
\begin{equation}\label{Eq2.7}
{\textstyle
{ }_{a}^{\rm G} \nabla_{k}^{-\alpha} {}_a^{\rm R} \nabla_{k}^{\alpha} x(k)=x(k)-\sum_{i=0}^{n-1} \frac{(k-a) \overline{^{\alpha-i-1}}}{\Gamma(\alpha-i)}[{ }_{a}^{\rm R} \nabla_{k}^{\alpha-i-1} x(k)]_{k=a},
}
\end{equation}
\begin{equation}\label{Eq2.8}
{\textstyle
{ }_{a}^{\rm G} \nabla_{k}^{-\alpha}{ }_{a}^{\rm C} \nabla_{k}^{\alpha} x(k)=x(k)-\sum_{i=0}^{n-1} \frac{(k-a) \overline{^i}}{i !}[\nabla^{i} x(k)]_{k=a} , }
\end{equation}
where $k \in \mathbb{N}_{a+1}$, $\alpha \in(n-1, n)$.
\end{lemma}

By introducing a tempered function $w: \mathbb{N}_{a+1} \to \mathbb{R} \backslash\{0\}$, the concept of nabla fractional calculus can be extended further.

For $x: \mathbb{N}_{a+1} \to \mathbb{R}$, its $\alpha$-th Gr\"{u}nwald--Letnikov tempered difference/sum is defined by \cite{Ferreira:2021NODY}
\begin{equation}\label{Eq2.9}
{\textstyle
{ }_{a}^{\rm G} \nabla_{k}^{\alpha, w(k)} x(k):=w^{-1}(k)_{a}^{\rm G} \nabla_{k}^{\alpha}[w(k) x(k)],
}
\end{equation}
where $\alpha \in \mathbb{R}, k \in \mathbb{N}_{a+1}, a \in \mathbb{R}$ and $w: \mathbb{N}_{a+1} \to \mathbb{R} \backslash\{0\}$.

The $n$-th nabla tempered difference, the $\alpha$-th Riemann--Liouville tempered fractional difference and Caputo tempered fractional difference of $x: \mathbb{N}_{a+1-n} \to \mathbb{R}$ can be defined by
\begin{equation}\label{Eq2.10}
{\textstyle
\nabla^{n, w(k)} x(k):=w^{-1}(k) \nabla^{n}[w(k) x(k)],}
\end{equation}
\begin{equation}\label{Eq2.11}
{\textstyle
{ }_{a}^{\mathrm{R}} \nabla_{k}^{\alpha, w(k)} x(k):=w^{-1}(k){}_{a}^{\mathrm{R}} \nabla_{k}^{\alpha}[w(k) x(k)],}
\end{equation}
\begin{equation}\label{Eq2.12}
{\textstyle
{ }_{a}^{\rm C} \nabla_{k}^{\alpha, w(k)} x(k):=w^{-1}(k){}_{a}^{\rm C} \nabla_{k}^{\alpha}[w(k) x(k)],}
\end{equation}
respectively, where $\alpha \in(n-1, n), n \in \mathbb{Z}_{+}, k \in \mathbb{N}_{a+1}, a \in \mathbb{R}$ and $w: \mathbb{N}_{a+1} \to \mathbb{R} \backslash\{0\}$. On this basis, the following relationships hold
\begin{equation}\label{Eq2.13}
{\textstyle
{ }_{a}^{\rm R} \nabla_{k}^{\alpha, w(k)} x(k)=\nabla^{n, w(k)}{}_{a}^{\rm G} \nabla_{k}^{\alpha-n, w(k)} x(k),}
\end{equation}

\begin{equation}\label{Eq2.14}
{\textstyle
{ }_{a}^{\rm C} \nabla_{k}^{\alpha, w(k)} x(k)={ }_{a}^{\rm G} \nabla_{k}^{\alpha-n, w(k)} \nabla^{n, w(k)} x(k) .}
\end{equation}

The equivalent condition of $w: \mathbb{N}_{a+1} \to \mathbb{R} \backslash\{0\}$ is finite nonzero. In this work, when $w(k)=(1-\lambda)^{k-a}$, $\lambda \in \mathbb{R} \backslash\{1\}$, the operations $\nabla^{n, w(k)} x(k),{ }_{a}^{\rm G} \nabla_{k}^{\alpha, w(k)} x(k)$, ${ }_{a}^{\rm R} \nabla_{k}^{\alpha, w(k)} x(k),{ }_{a}^{\rm C} \nabla_{k}^{\alpha, w(k)} x(k)$ could be abbreviate as $\nabla^{n, \lambda} x(k),{ }_{a}^{\rm G} \nabla_{k}^{\alpha, \lambda} x(k),{ }_{a}^{\rm R} \nabla_{k}^{\alpha, \lambda} x(k)$, ${ }_{a}^{\rm C} \nabla_{k}^{\alpha, \lambda} x(k)$, respectively. Notably, this special case is different from the one in \cite{Ferreira:2021NODY}, which facilitates the use and analysis. \textsf{Compared to existing results, the tempered function $w(k)$ is no longer limited to the exponential function, which makes this work more general and practical.}

By using the linearity, the following lemma can be derived immediately, which is simple while useful for understanding such fractional calculus.
\begin{lemma}\label{Lemma2.2}
For any function $x: \mathbb{N}_{a+1-n} \to \mathbb{R}$, $n \in \mathbb{Z}_{+}$, $a \in \mathbb{R}$, finite nonzero $w(k)$, $\alpha\in(n-1,n)$, $k \in \mathbb{N}_{a+1}$, $\lambda \in\mathbb{R}\backslash\{0\}$, one has
\begin{equation}\label{Eq2.15}
{\textstyle
\left\{ \begin{array}{rl}
{\nabla ^{n,w(k)}}x(k) =&\hspace{-6pt} {\nabla ^{n,\rho w(k)}}x(k),\\
{}_a^{\rm{G}}\nabla _k^{\alpha ,w(k)}x(k) =&\hspace{-6pt} {}_a^{\rm{G}}\nabla _k^{\alpha ,\lambda w(k)}x(k),\\
{}_a^{\rm{G}}\nabla _k^{ - \alpha ,w(k)}x(k) =&\hspace{-6pt} {}_a^{\rm{G}}\nabla _k^{ - \alpha ,\lambda w(k)}x(k),\\
{}_a^{\rm{R}}\nabla _k^{\alpha ,w(k)}x(k) =&\hspace{-6pt} {}_a^{\rm{R}}\nabla _k^{\alpha ,\lambda w(k)}x(k),\\
{}_a^{\rm{C}}\nabla _k^{\alpha ,w(k)}x(k) =&\hspace{-6pt} {}_a^{\rm{C}}\nabla _k^{\alpha ,\lambda w(k)}x(k).
\end{array} \right.}
\end{equation}
\end{lemma}
Note that Lemma \ref{Lemma2.2} is indeed the scale invariance. When $\lambda=-1$, the sign of $w(k)$ is just reversed to $\lambda w(k)$. From this, one is ready to claim that if a property on tempered calculus holds for $w(k)>0$, it also holds for $w(k)<0$.

\section{Main Results}\label{Section3}
In this section, a series of nice properties will be nicely developed on the nabla tempered fractional calculus.
\subsection{The basic relationship}
In this part, the relationship between tempered nabla fractional difference/sum will be discussed.

\begin{theorem}\label{Theorem3.1}
For any $x: \mathbb{N}_{a+1-n} \to \mathbb{R}$, $\alpha \in(n-1, n)$, $n \in \mathbb{Z}_{+}$, $k \in \mathbb{N}_{a+1}$, $a \in \mathbb{R}$, $w: \mathbb{N}_{a+1} \to \mathbb{R} \backslash\{0\}$, one has
\begin{equation}\label{Eq3.1}
{\textstyle
{ }_{a}^{\mathrm{R}} \nabla_{k}^{\alpha, w(k)} x(k)={ }_{a}^{\rm G} \nabla_{k}^{\alpha, w(k)} x(k),}
\end{equation}
\begin{equation}\label{Eq3.2}
{\textstyle
{ }_{a}^{\mathrm{R}} \nabla_{k}^{\alpha, w(k)} x(k)={ }_{a}^{\mathrm{C}} \nabla_{k}^{\alpha, w(k)} x(k)+\sum_{i=0}^{n-1} \frac{(k-a)^{\overline{i-\alpha}}}{\Gamma(i-\alpha+1)} \frac{w(a)}{w(k)}[\nabla^{i, w(k)} x(k)]_{k=a},}
\end{equation}
\begin{equation}\label{Eq3.3}
{\textstyle
{ }_{a}^{\rm G} \nabla_{k}^{-\alpha, w(k)} x(k)=\nabla^{n, w(k)}{}_a^{ \mathrm{G}} \nabla_{k}^{-\alpha-n, w(k)} x(k) .
}
\end{equation}
\end{theorem}
\begin{proof}
Let $z(k):=w(k) x(k)$. By using \cite[Theorem 3.57, Theorem 3.41]{Goodrich:2015Book}, i.e., $\nabla \frac{(k-a)^{\overline{\beta}}}{\Gamma(\beta+1)}=\frac{(k-a)^{\overline{\beta-1}}}{\Gamma(\beta)}, \beta \in \mathbb{R} \backslash \mathbb{Z}_{+}, \nabla \sum_{j=a+1}^{k} f(j, k)=\sum_{j=a+1}^{k} \nabla f(j, k)+f(k, k-1)$, one has
\begin{equation}\label{Eq3.4}
{\textstyle \begin{array}{rl}
{}_a^{\rm R}\nabla _{k}^{\alpha,w(k)} x(k)=&\hspace{-6pt}w^{-1}(k){}_a^{\rm R}\nabla _k^\alpha z(k)\\
 =&\hspace{-6pt}w^{-1}(k){\nabla ^n}{}_a^{\rm G}\nabla _k^{\alpha  - n}z(k)\\
 =&\hspace{-6pt} w^{-1}(k){\nabla ^n}\sum\nolimits_{j = a + 1}^k {\frac{{ ( {k - j + 1}  )\overline {^{n - \alpha  - 1}} }}{{\Gamma  ( {n - \alpha }  )}}z(j)}\\
 =&\hspace{-6pt} w^{-1}(k){\nabla ^{n - 1}}[\sum\nolimits_{j = a + 1}^k {\nabla \frac{{ ( {k - j + 1}  )\overline {^{n - \alpha  - 1}} }}{{\Gamma  ( {n - \alpha }  )}}z(j)}  + \frac{{0\overline {^{n - \alpha  - 1}} }}{{\Gamma  ( {n - \alpha }  )}}z(k)]\\
 =&\hspace{-6pt}w^{-1}(k){\nabla ^{n - 1}}\sum\nolimits_{j = a + 1}^k {\frac{{ ( {k - j + 1}  )\overline {^{n - \alpha  - 2}} }}{{\Gamma  ( {n - \alpha  - 1}  )}}z(j)}  + w^{-1}(k)\frac{{0\overline {^{n - \alpha  - 1}} }}{{\Gamma  ( {n - \alpha }  )}}{\nabla ^{n - 1}}z(k)\\
 =&\hspace{-6pt} w^{-1}(k){\nabla ^{n - 2}}\sum\nolimits_{j = a + 1}^k {\nabla \frac{{ ( {k - j + 1}  )\overline {^{n - \alpha  - 2}} }}{{\Gamma  ( {n - \alpha  - 1}  )}}z(j)}  + w^{-1}(k)\frac{{0\overline {^{n - \alpha  - 2}} }}{{\Gamma  ( {n - \alpha  - 1}  )}}{\nabla ^{n - 2}}z(k)\\
 &\hspace{-6pt}+ w^{-1}(k)\frac{{0\overline {^{n - \alpha  - 1}} }}{{\Gamma  ( {n - \alpha }  )}}{\nabla ^{n - 1}}z(k)\\
 =&\hspace{-6pt} w^{-1}(k)\sum\nolimits_{j = a + 1}^k {\frac{{ ( {k - j + 1}  )\overline {^{-\alpha  - 1}} }}{{\Gamma  ( { - \alpha }  )}}z(j)}  + w^{-1}(k)\sum\nolimits_{i = 0}^{n - 1} {\frac{{0\overline {^{ i- \alpha}} }}{{\Gamma  ( { i- \alpha + 1}  )}}{\nabla ^i}z(k)}\\
=&\hspace{-6pt} {}_a^{\rm G}\nabla _k^{\alpha ,w(k) }x(k) + \sum\nolimits_{i = 0}^{n - 1} {\frac{{0\overline {^{ i- \alpha}} }}{{\Gamma  ( { i- \alpha  + 1} )}}{\nabla ^{i,w(k) }}x(k)} ,
\end{array}}
\end{equation}
in which all the first order differences are taken with respect to $k$. By using $\frac{0^{\overline{-\beta}}}{\Gamma(1-\beta)}=0$, $\beta \in \mathbb{R}\backslash \mathbb{Z}_{+} $ and the finite $\nabla^{i, w(k)} x(k)$, $i=0,1, \cdots, n-1$, $k \in \mathbb{N}_{a+1}$, the desired result in (\ref{Eq3.1}) can be further derived from (\ref{Eq3.4}).

By applying the formula of summation by parts in \cite[Theorem 3.39]{Goodrich:2015Book}, i.e., $\sum_{j=a+1}^{k} \nabla f(j) g(j)=\left.f(j) g(j)\right|_{j=a} ^{j=k}-\sum_{j=a+1}^{k} f(j-1) \nabla g(j)$, it follows
\begin{equation}\label{Eq3.5}
\begin{array}{rl}
{}_a^{\rm G}\nabla _k^{\alpha,w(k)} x(k)=&\hspace{-6pt}w^{-1}(k){}_a^{\rm G}\nabla _k^\alpha z(k) \\
=&\hspace{-6pt} w^{-1}(k)\sum\nolimits_{j = a + 1}^k {\frac{{ ( {k - j + 1} )\overline {^{ - \alpha  - 1}} }}{{\Gamma  ( { - \alpha }  )}}z(j)} \\
 =&\hspace{-6pt}w^{-1}(k)  - \sum\nolimits_{j = a + 1}^k {\nabla \frac{{ ( {k - j}  )\overline {^{ - \alpha }} }}{{\Gamma  ( {1 - \alpha }  )}}z(j)} \\
 =&\hspace{-6pt}w^{-1}(k) \frac{{ ( {k - a}  )\overline {^{ - \alpha }} }}{{\Gamma  ( {1 - \alpha }  )}}z ( a  ) - \frac{{0\overline {^{ - \alpha }} }}{{\Gamma  ( {1 - \alpha }  )}}z(k) + w^{-1}(k)\sum\nolimits_{j = a + 1}^k {\frac{{ ( {k - j + 1}  )\overline {^{ - \alpha }} }}{{\Gamma  ( {1 - \alpha }  )}}\nabla z(j)} \\
 =&\hspace{-6pt}w^{-1}(k) \sum\nolimits_{j = a + 1}^k {\frac{{ ( {k - j + 1}  )\overline {^{n - \alpha  - 1}} }}{{\Gamma  ( {n - \alpha  + 1}  )}}{\nabla ^n}z(j)}  + w^{-1}(k)\sum\nolimits_{i = 0}^{n - 1} {\frac{{ ( {k - a}  )\overline {^{i - \alpha }} }}{{\Gamma  ( {i - \alpha  + 1}  )}}} { [ {{\nabla ^i}z(k)} ]_{k = a}}\\
 &\hspace{-6pt}- w^{-1}(k)\sum\nolimits_{i = 0}^{n - 1} {\frac{{0\overline {^{i - \alpha }} }}{{\Gamma  ( {i - \alpha  + 1}  )}}{\nabla ^i}z(k)}\\
 =&\hspace{-6pt}{}_a^{\rm C}\nabla _k^{\alpha,w(k)} x(k)+\sum\nolimits_{i = 0}^{n - 1} {\frac{{ ( {k - a}  )\overline {^{i - \alpha }} }}{{\Gamma  ( {i - \alpha  + 1}  )}}} \frac{w(a)}{w(k)} [ {{\nabla ^{i,w(k)}}z(k)}  ]_{k = a}\\
 &\hspace{-6pt}- \sum\nolimits_{i = 0}^{n - 1} {\frac{{0\overline {^{ i- \alpha}} }}{{\Gamma  ( { i- \alpha  + 1} )}}{\nabla ^{i,w(k) }}x(k)}.
\end{array}
\end{equation}
The result in (\ref{Eq3.2}) can be further developed by using (\ref{Eq3.4}) and (\ref{Eq3.5}).

By utilizing  \cite[Theorem 4]{Wei:2019JCND}, one has
\begin{equation}\label{Eq3.6}
\begin{array}{rl}
{}_a^{\rm G}\nabla _k^{ - \alpha ,w(k)}x(k) =&\hspace{-6pt} {w^{ - 1}}(k){}_a^{\rm G}\nabla _k^{ - \alpha }z(k)\\
 =&\hspace{-6pt} {w^{ - 1}}(k){\nabla ^n}{}_a^{\rm{G}}\nabla _k^{ - \alpha  - n}z(k)\\
 =&\hspace{-6pt} {w^{ - 1}}(k){\nabla ^n}w(k){w^{ - 1}}(k){}_a^{\rm{G}}\nabla _k^{ - \alpha  - n}w(k)x(k)\\
 =&\hspace{-6pt} {\nabla ^{n,w(k)}}{}_a^{\rm{G}}\nabla _k^{ - \alpha-n ,w(k)}x(k),
\end{array}
\end{equation}
which is just (\ref{Eq3.3}). The proof is completed.
\end{proof}

Theorem \ref{Theorem3.1} presents the basic relation between different fractional differences. Note that the finite value assumption on $\nabla^{i, w(k)} x(k)$ is not necessary for (3.2). When $w(k)=1$, the tempered case reduces to the classical case \cite[Theorem 4]{Wei:2019JCND}, \cite[Theorem 3]{Wei:2019CNSNSa}. $w(k) \neq 0$, while it can be positive or negative, for example, $w(k)=\sin \left(k \pi-a \pi+\frac{\pi}{4}\right), k \in \mathbb{N}_{a+1}$.

\begin{theorem}\label{Theorem3.2}
For any $\alpha\in(n-1,n)$, $n\in\mathbb{Z}_+$, $w: \mathbb{N}_{a+1} \to \mathbb{R} \backslash\{0\}$, $k\in\mathbb{N}_{a+1}$, $a\in\mathbb{R}$, one has
\begin{equation}\label{Eq3.7}
{\textstyle {}_a^{\rm R}\nabla _{k }^{\alpha,w(k)} {}_a^{\rm G}\nabla _{k }^{ - \alpha,w(k) }x(k) = {}_a^{\rm C}\nabla _{k }^{\alpha,w(k)} {}_a^{\rm G}\nabla _{k }^{ - \alpha,w(k) }x(k) = x(k),}
\end{equation}
\begin{equation}\label{Eq3.8}
{\textstyle
\begin{array}{l}
{}_a^{\rm G}\nabla _{k }^{ - \alpha,w(k) }{}_a^{\rm R}\nabla _{k }^{\alpha,w(k)} x(k) = x(k) - \sum\nolimits_{i = 0}^{n - 1} {\frac{{ ( {k - a}  )\overline {^{\alpha  - i - 1}} }}{{\Gamma ( {\alpha  - i}  )}}}\frac{w(a)}{w(k)} {[ {{}_a^{\rm R}\nabla _k^{\alpha  - i - 1,w(k)}x(k)} ]_{k = a}},
\end{array}}
\end{equation}
\begin{equation}\label{Eq3.9}
{\textstyle {}_a^{\rm G}\nabla _{k }^{ - \alpha,w(k) }{}_a^{\rm C}\nabla _{k }^{\alpha,w(k)} x(k) =  x(k) - \sum\nolimits_{i = 0}^{n - 1} {\frac{{ ( {k - a}  )\overline {^i} }}{{i!}}} \frac{w(a)}{w(k)} {[ {{\nabla ^{i,w(k)}}x(k)} ]_{k = a}}.}
\end{equation}
\end{theorem}
\begin{proof}
From Lemma \ref{Lemma2.1} and \cite[Theorem 2]{Wei:2019JCND}, one has
\begin{equation}\label{Eq3.10}
{\textstyle \begin{array}{rl}
{}_a^{\rm R}\nabla _k^{\alpha ,w(k) }{}_a^{\rm G}\nabla _k^{ - \alpha ,w(k) }x(k) =&\hspace{-6pt} {\nabla ^{n,w(k) }}{}_a^{\rm G}\nabla _k^{\alpha  - n,w(k) }{}_a^{\rm G}\nabla _k^{ - \alpha ,w(k) }x(k)\\
 =&\hspace{-6pt} {w^{-1}(k)}{\nabla ^{n }}{}_a^{\rm G}\nabla _k^{\alpha  - n }{}_a^{\rm G}\nabla _k^{ - \alpha  }[ w(k)x(k) ]\\
 =&\hspace{-6pt} {w^{-1}(k)}{\nabla ^n}{}_a^{\rm G}\nabla _k^{ - n}[ w(k)x(k) ]\\
 =&\hspace{-6pt} x(k).
\end{array}}
\end{equation}

Similarly, by using Lemma \ref{Lemma2.1} and \cite[Theorem 4]{Wei:2019JCND}, it follows
\begin{equation}\label{Eq3.11}
{\textstyle \begin{array}{rl}
{}_a^{\rm C}\nabla _k^{\alpha ,w(k) }{}_a^{\rm G}\nabla _k^{ - \alpha ,w(k) }x(k) =&\hspace{-6pt} {}_a^{\rm G}\nabla _k^{\alpha  - n,w(k) }{\nabla ^{n,w(k) }}{}_a^{\rm G}\nabla _k^{ - \alpha ,w(k) }x(k)\\
 =&\hspace{-6pt} w^{-1}(k){}_a^{\rm G}\nabla _k^{\alpha  - n }{\nabla ^{n}}{}_a^{\rm G}\nabla _k^{ - \alpha }[w(k)x(k)]\\
 =&\hspace{-6pt} w^{-1}(k){}_a^{\rm G}\nabla _k^{\alpha  - n }{}_a^{\rm G}\nabla _k^{ n- \alpha }[w(k)x(k)]\\
 =&\hspace{-6pt} w^{-1}(k){}_a^{\rm G}\nabla _k^{0}[w(k)x(k)]\\
 =&\hspace{-6pt} x(k).
\end{array}}
\end{equation}
Combining (\ref{Eq3.10}) and (\ref{Eq3.11}) yields (\ref{Eq3.7}).

Assume $z(k) := w(k)x(k)$. It follows
\begin{equation}\label{Eq3.12}
\begin{array}{rl}
{}_a^{\rm G}\nabla _{k }^{ - \alpha,w(k) }{}_a^{\rm R}\nabla _{k }^{\alpha,w(k)} x(k)=&\hspace{-6pt}w^{-1}(k){}_a^{\rm G}\nabla _k^{ - \alpha }{}_a^{\rm R}\nabla _k^\alpha z(k)\\
=&\hspace{-6pt}w^{-1}(k)\sum\nolimits_{j = a + 1}^k {\frac{{ ( {k - j + 1}  )\overline {^{\alpha  - 1}} }}{{\Gamma  ( \alpha   )}}{}_a^{\rm R}\nabla _j^\alpha z(j)} \\
 =&\hspace{-6pt} w^{-1}(k)\nabla \sum\nolimits_{j = a + 1}^k {\frac{{ ( {k - j + 1}  )\overline {^\alpha } }}{{\Gamma  ( {\alpha  + 1}  )}}{}_a^{\rm R}\nabla _j^\alpha z(j)} \\
 =&\hspace{-6pt} w^{-1}(k)\nabla \sum\nolimits_{j = a + 1}^k {\frac{{ ( {k - j + 1}  )\overline {^\alpha } }}{{\Gamma  ( {\alpha  + 1}  )}}{\nabla ^n}_a^{\rm G}\nabla _j^{\alpha  - n}z(j)} \\
 =&\hspace{-6pt} w^{-1}(k)\nabla \sum\nolimits_{j = a + 1}^k {\frac{{ ( {k - j + 1}  )\overline {^{\alpha  - n}} }}{{\Gamma  ( {\alpha  - n + 1}  )}}{}_a^{\rm G}\nabla _j^{\alpha  - n}z(j)} \\
  &\hspace{-6pt}- w^{-1}(k){\sum\nolimits_{i = 0}^{n - 1} {[ {{\nabla ^{n - i - 1}}{}_a^{\rm G}\nabla _k^{\alpha  - n}z(k)}]} _{k = a}\frac{{ ( {k - a}  )\overline {^{\alpha  - i - 1}} }}{{\Gamma  ( {\alpha  - i} )}}}\\
 =&\hspace{-6pt} w^{-1}(k)\nabla {}_a^{\rm G}\nabla _k^{n - \alpha  - 1}{}_a^{\rm G}\nabla _k^{\alpha  - n}z(k)\\
  &\hspace{-6pt}- w^{-1}(k){\sum\nolimits_{i = 0}^{n - 1} {[ {{}_a^{\rm R}\nabla _k^{\alpha  - i - 1}z(k)} ]} _{k = a}\frac{{ ( {k - a}  )\overline {^{\alpha  - i - 1}} }}{{\Gamma  ( {\alpha  - i}  )}}}\\
  =&\hspace{-6pt} w^{-1}(k){}_a^{\rm G}\nabla _k^{n - \alpha }{}_a^{\rm G}\nabla _k^{\alpha  - n}z(k)\\
  &\hspace{-6pt}- {\sum\nolimits_{i = 0}^{n - 1} \frac{w(a)}{w(k)}{[ {{}_a^{\rm R}\nabla _k^{\alpha  - i - 1}z(k)} ]} _{k = a}\frac{{ ( {k - a}  )\overline {^{\alpha  - i - 1}} }}{{\Gamma  ( {\alpha  - i}  )}}}\\
 =&\hspace{-6pt} w^{-1}(k)z(k)- {\sum\nolimits_{i = 0}^{n - 1} \frac{{ ( {k - a}  )\overline {^{\alpha  - i - 1}} }}{{\Gamma  ( {\alpha  - i}  )}}\frac{w(a)}{w(k)}{[ {{}_a^{\rm R}\nabla _k^{\alpha  - i - 1,w(k) }x(k)} ]} _{k = a}},
\end{array}
\end{equation}
which implies (\ref{Eq3.8}).

For the Caputo case, one has
\begin{equation}\label{Eq3.13}
\begin{array}{rl}
{}_a^{\rm G}\nabla _k^{ - \alpha,w(k) }{}_a^{\rm C}\nabla _k^{\alpha,w(k)} x(k)=&\hspace{-6pt}w^{-1}(k){}_a^{\rm G}\nabla _k^{ - \alpha }{}_a^{\rm C}\nabla _k^\alpha z(k)\\
=&\hspace{-6pt}w^{-1}(k){}_a^{\rm G}\nabla _k^{ - \alpha }{}_a^{\rm G}\nabla _k^{\alpha  - n}{\nabla ^n}z(k)\\
=&\hspace{-6pt}w^{-1}(k) {}_a^{\rm G}\nabla _k^{ - n}{\nabla ^n}z(k)\\
=&\hspace{-6pt}w^{-1}(k) z(k) -w^{-1}(k){\sum\nolimits_{i = 0}^{n - 1} {\frac{{ ( {k - a}  )\overline {^i} }}{{i!}}[ {{\nabla ^i}z(k)} ]} _{k = a}}\\
=&\hspace{-6pt}x(k) -{\sum\nolimits_{i = 0}^{n - 1} \frac{w(a)}{w(k)}{\frac{{ ( {k - a} )\overline {^i} }}{{i!}}[ {{\nabla ^{i,w(k) }}x(k)} ]} _{k = a}},
\end{array}
\end{equation}
which leads to (\ref{Eq3.9}). All of these complete the proof.
\end{proof}

Theorem \ref{Theorem3.2} gives the relation between tempered fractional difference and tempered fractional sum, which can be the generalization of \cite[Corollary 5]{Wei:2019CNSNSa}. Along this way, letting $\beta\in(m-1,m)$, $m,n\in\mathbb{Z}_+$, $m\le n$, $k\in\mathbb{N}_{a+1}$, then one has
\[
\left\{\begin{array}{l}
{}_a^{\rm R}\nabla _{k }^{n-\beta,w(k)} {}_a^{\rm C}\nabla _{k }^{ \beta,w(k) }x(k) =\nabla ^{n,w(k) }x(k),\\
{}_a^{\rm R}\nabla _{k }^{\beta,w(k)} {}_a^{\rm C}\nabla _{k }^{ n-\beta,w(k) }x(k) =\nabla ^{n,w(k) }x(k),\\
{}_a^{\rm C}\nabla _{k }^{n-\beta,w(k)} {}_a^{\rm R}\nabla _{k }^{ \beta,w(k) }x(k) ={}_a^{\rm G}\nabla _{k }^{n,w(k) }x(k),\\
{}_a^{\rm C}\nabla _{k }^{\beta,w(k)} {}_a^{\rm R}\nabla _{k }^{ n-\beta,w(k) }x(k) ={}_a^{\rm G}\nabla _{k }^{n,w(k) }x(k).
\end{array}\right.
\]

By rearranging (\ref{Eq3.9}) further, one has
\begin{equation}\label{Eq3.14}
{\textstyle
\begin{array}{rl}
x(k)   =&\hspace{-6pt} \sum\nolimits_{i = 0}^{n - 1} {{\frac{{ ( {k - a}  )\overline {^i} }}{{i!}}} \frac{w(a)}{w(k)}[ {{\nabla ^{i,w(k)}}x(k)} ]_{k = a}}+ \sum\nolimits_{j = a + 1}^k {{}\frac{{ ( {k - j + 1}  )\overline {^{\alpha-1}} }}{{\Gamma(\alpha)}}\frac{w(j)}{w(k)}{{}_a^{\rm C}\nabla_j ^{\alpha,w(j)}}x ( j )}\\
=&\hspace{-6pt} \sum\nolimits_{i = 0}^{n - 1} {{\frac{{ ( {k - a}  )\overline {^i} }}{{i!}}} \frac{w(a)}{w(k)}[ {{\nabla ^{i,w(k)}}x(k)} ]_{k = a}}+ \sum\nolimits_{j = a + 1}^k {{}\frac{{ ( {k - j + 1}  )\overline {^{n-1}} }}{{\Gamma(n)}}\frac{w(j)}{w(k)}{\nabla ^{n,w(j)}}x ( j )},
\end{array}}
\end{equation}
which can be regarded as \textsf{the nabla Taylor expansion} of $x(k)$ at $k=a$ with summation reminder. The summation might be the fractional order case ${}_a^{\rm G}\nabla _k^{ - \alpha,w(k) }{}_a^{\rm C}\nabla _k^{\alpha,w(k)} x(k)$ or the integer order case ${}_a^{\rm G}\nabla _k^{ - n,w(k) }\nabla ^{n,w(k)} x(k)$.

In a similar way, for any $m\in\mathbb{Z}_+$, $m<\alpha$, one has
\begin{equation}\label{Eq3.15}
{\textstyle
\begin{array}{rl}
{\nabla ^{m,w(k)}}x(k)   =&\hspace{-6pt} \sum\nolimits_{i = m}^{n-1} {\frac{ ( {k - a}  )\overline {^{i-m}} }{{(i-m)!}}\frac{w(a)}{w(k)}[{{\nabla ^{i,w(k)}}x(k)}]_{k=a}} + \sum\nolimits_{j = a + 1}^k {{}\frac{{ ( {k - j + 1}  )\overline {^{\alpha-m-1}} }}{{\Gamma(\alpha-m)}}\frac{w(j)}{w(k)}{}_a^{\rm C}\nabla_j ^{\alpha,w(k)}x ( j )}.
\end{array}}
\end{equation}

From the definition, one can derive that
\begin{equation}\label{Eq3.16}
{\textstyle
\begin{array}{rl}
{}_a^{\rm G}\nabla _k^{n,w(k) }x(k) =&\hspace{-6pt} {\nabla ^{n,w(k) }}x(k)+ \left \{ \begin{array}{rl}
\sum\nolimits_{i = k - a}^n {{{ ( { - 1}  )}^i} ( {\begin{smallmatrix}
n\\
i
\end{smallmatrix}}  )\frac{w(k-i)}{w(k)}x ( {k - i}  )} ,&\hspace{-6pt}k < a + n,\\
0 ,&\hspace{-6pt}k \ge a + n,
\end{array}  \right.
\end{array}}
\end{equation}
where $n\in\mathbb{Z}_+$, $k\in \mathbb{N}_{a+1}$, $a\in\mathbb{R}$ and $w(k)\neq0$. In other words, ${}_{a}^{\rm G}{\nabla}_{k}^{n,w(k)} x(k)$ is not always identical to ${\nabla}^{ n,w(k) } x(k)$. To explore more details, the limit of the fractional order case will be discussed.

\begin{theorem}\label{Theorem3.3}
For any $x : \mathbb{N}_{a+1} \to \mathbb{R}$, $\alpha>0$, $k\in\mathbb{N}_{a+1}$, $a\in\mathbb{R}$, $w: \mathbb{N}_{a+1} \to \mathbb{R} \backslash\{0\}$, ${}_a^{\rm G}\nabla _k^{-\alpha,w(k)}x(k)$ is continuous with respect to $\alpha$ and one has
\begin{equation}\label{Eq3.17}
{\textstyle \mathop {\lim }\limits_{\alpha  \to 0} {}_a^{\rm G}\nabla _k^{-\alpha,w(k)}x(k) = x(k)}.
\end{equation}
\end{theorem}
\begin{proof}
By using (\ref{Eq2.6}), (\ref{Eq2.9}) and the formula of summation by parts, it follows
\begin{equation}\label{Eq3.18}
{\textstyle
\begin{array}{rl}
{}_a^{\rm G}\nabla _k^{-\alpha,w(k)}x(k) =&\hspace{-6pt} w^{-1}(k)\sum\nolimits_{j = a + 1}^k {{}\frac{{ ( {k - j + 1}  )\overline {^{\alpha-1}} }}{{\Gamma(\alpha+1)}}w(j)x (j)}\\
=&\hspace{-6pt}w^{-1}(k)\sum\nolimits_{j = a + 1}^k {{}\frac{{ ( {k - j + 1}  )\overline {^{\alpha}} }}{{\Gamma(\alpha+1)}}\nabla[w(j)x (j)]}-w^{-1}(k)\big[\frac{{ ( {k - j}  )\overline {^{\alpha}} }}{{\Gamma(\alpha)}}w(j)x(j)\big]_{j=a}^{j=k}\\
=&\hspace{-6pt}w^{-1}(k)\sum\nolimits_{j = a + 1}^k {\frac{{ ( {k - j + 1}  )\overline {^{\alpha}} }}{{\Gamma(\alpha+1)}}\nabla[w(j)x (j)]}+\frac{{ ( {k - a}  )\overline {^{\alpha}} }}{{\Gamma(\alpha+1)}}\frac{w(a)}{w(k)}x(a).
\end{array}}
\end{equation}
Since $\frac{{ ( {k - j + 1}  )\overline {^{\alpha}} }}{{\Gamma(\alpha+1)}}$ is continuous with respect to $\alpha$ for any $j\in\mathbb{N}_{a+1}^{k}$, ${}_a^{\rm G}\nabla _k^{-\alpha,w(k)}x(k)$ is also continuous. Due to the fact of $\mathop {\lim }\limits_{\alpha  \to 0}{\frac{{ ( {k - j + 1}  )\overline {^{\alpha}} }}{{\Gamma(\alpha+1)}}}=1$, taking limit for both sides of (\ref{Eq3.18}), one has
\begin{equation}\label{Eq3.19}
{\textstyle
\begin{array}{rl}
\mathop {\lim }\limits_{\alpha  \to 0}{{}_a^{\rm G}\nabla _k^{-\alpha,w(k)}x(k)} =&\hspace{-6pt} w^{-1}(k)\sum\nolimits_{j = a + 1}^k {\nabla[w(j)x (j)]}+\frac{w(a)}{w(k)}x(a)= x(k).
\end{array}}
\end{equation}
The proof is thus done.
\end{proof}

\begin{theorem}\label{Theorem3.4}
For any $x : \mathbb{N}_{a+1-n} \to \mathbb{R}$, $a\in\mathbb{R}$, $n\in\mathbb{Z}_+$, $\alpha\in(n-1,n)$, $w: \mathbb{N}_{a+1} \to \mathbb{R} \backslash\{0\}$, ${}_a^{\rm C}\nabla _k^{\alpha ,w(k) }x(k)$, $k\in\mathbb{N}_{a+1}$ and ${}_a^{\rm R}\nabla _k^{\alpha ,w(k) }x(k)$, $k\in\mathbb{N}_{a+n+1}$ are continuous with respect to $\alpha$ and the following limits hold
\begin{equation}\label{Eq3.20}
{\textstyle \mathop {\lim }\limits_{\alpha  \to n } {}_a^{\rm C}\nabla _k^{\alpha ,w(k) }x(k) = {\nabla ^{n,w(k) }}x(k)}, k\in\mathbb{N}_{a+1},
\end{equation}
\begin{equation}\label{Eq3.21}
{\textstyle
\mathop {\lim }\limits_{\alpha  \to n - 1} {}_a^{\rm C}\nabla _k^{\alpha ,w(k) }x(k) + \frac{w(a)}{w(k)}{ [ {{\nabla ^{n - 1,w(k) }}x(k)}  ]_{k = a}} = {\nabla ^{n - 1,w(k) }}x(k),k\in\mathbb{N}_{a+1},}
\end{equation}
\begin{equation}\label{Eq3.22}
{\textstyle \mathop {\lim }\limits_{\alpha  \to n} {}_a^{\rm R}\nabla _k^{\alpha ,w(k) }x(k)  = {\nabla ^{n,w(k) }}x(k)}, k\in\mathbb{N}_{a+n+1},
\end{equation}
\begin{equation}\label{Eq3.23}
{\textstyle
\mathop {\lim }\limits_{\alpha  \to n - 1} {}_a^{\rm R}\nabla _k^{\alpha ,w(k) }x(k) ={\nabla ^{n - 1,w(k) }}x(k), k\in\mathbb{N}_{a+n}.}
\end{equation}
\end{theorem}
\begin{proof}
Assume $z(k) :=w(k)x(k)$. The formula of summation by parts gives
\begin{equation}\label{Eq3.24}
{\textstyle \begin{array}{rl}
{}_a^{\rm C}\nabla _k^{\alpha ,w(k) }x(k) =&\hspace{-6pt}w^{-1}(k){}_a^{\rm C}\nabla _k^\alpha z(k)\\
 =&\hspace{-6pt}w^{-1}(k)  \sum\nolimits_{j = a + 1}^k {\frac{{ ( {k - j + 1} )\overline {^{n - \alpha  - 1}} }}{{\Gamma  ( {n - \alpha }  )}}{\nabla ^n}z(j)} \\
 =&\hspace{-6pt}w^{-1}(k) \sum\nolimits_{j = a + 1}^k {\frac{{ ( {k - j + 1} )\overline {^{n - \alpha }} }}{{\Gamma  ( {n - \alpha  + 1}  )}}{\nabla ^{n + 1}}z(j)}+ w^{-1}(k)\frac{{ ( {k - a}  )\overline {^{n - \alpha }} }}{{\Gamma  ( {n - \alpha  + 1}  )}}{{ [ {{\nabla ^n}z(k)} ]}_{k = a}}.
\end{array}}
\end{equation}

Note that $\frac{{ ( {k - j + 1}  )\overline {^{n-\alpha-1}} }}{{\Gamma(n-\alpha)}}$ is continuous regarding to $\alpha$ for any $j\in\mathbb{N}_{a+1}^{k}$ $k\in\mathbb{N}_{a+1}$, and therefore ${}_a^{\rm C}\nabla _k^{\alpha,w(k)}x(k)$ is also continuous. Due to the continuity of ${\frac{{ ( {k - a}  )\overline {^{i - \alpha }} }}{{\Gamma  ( {i - \alpha  + 1}  )}}}$ with respect to $\alpha$, $i=0,1,\cdots,n-1$, $\forall k\in\mathbb{N}_{a+n+1}$ and ${}_a^{\rm R}\nabla _k^{\alpha,w(k)} x(k) = {}_a^{\rm C}\nabla _k^{\alpha,w(k)} x(k)+ \sum\nolimits_{i = 0}^{n - 1} {\frac{{ ( {k - a}  )\overline {^{i - \alpha }} }}{{\Gamma  ( {i - \alpha  + 1}  )}}}\frac{w(a)}{w(k)} { [ {{\nabla ^{i,w(k)}}x(k)}  ]_{k = a}}$, ${}_a^{\rm C}\nabla _k^{\alpha,w(k)}x(k)$ is continuous $k\in\mathbb{N}_{a+n+1}$.

Taking limit for $\alpha\to n$ and using $\Gamma(-z)=\infty$, $z\in\mathbb{N}$, one has
\begin{equation}\label{Eq3.25}
{\textstyle \begin{array}{rl}
\mathop {\lim }\limits_{\alpha  \to n} {}_a^{\rm C}\nabla _k^{\alpha,w(k)} x(k)=&\hspace{-6pt}\mathop {\lim }\limits_{\alpha  \to n}w^{-1}(k) \sum\nolimits_{j = a + 1}^k {\frac{{ ( {k - j + 1} )\overline {^{n - \alpha }} }}{{\Gamma  ( {n - \alpha  + 1}  )}}{\nabla ^{n + 1}}z(j)}\\
 &\hspace{-6pt}+\mathop {\lim }\limits_{\alpha  \to n}w^{-1}(k)\frac{{ ( {k - a}  )\overline {^{n - \alpha }} }}{{\Gamma  ( {n - \alpha  + 1}  )}}{{ [ {{\nabla ^n}z(k)} ]}_{k = a}}\\
=&\hspace{-6pt}w^{-1}(k) \sum\nolimits_{j = a + 1}^k \mathop {\lim }\limits_{\alpha  \to n}{\frac{{ ( {k - j + 1} )\overline {^{n - \alpha }} }}{{\Gamma  ( {n - \alpha  + 1}  )}}{\nabla ^{n + 1}}z(j)}\\
 &\hspace{-6pt}+w^{-1}(k)\mathop {\lim }\limits_{\alpha  \to n}\frac{{ ( {k - a}  )\overline {^{n - \alpha }} }}{{\Gamma  ( {n - \alpha  + 1}  )}}{{ [ {{\nabla ^n}z(k)} ]}_{k = a}}\\
 =&\hspace{-6pt}w^{-1}(k)\sum\nolimits_{j = a + 1}^k {{\nabla ^{n + 1}}z(j)}+w^{-1}(k){ [ {{\nabla ^n}z(k)}  ]_{k = a}}\\
=&\hspace{-6pt}w^{-1}(k) {\nabla ^n}z(k)-w^{-1}(k) [{\nabla ^n}z(k)]_{k=a}+w^{-1}(k){ [ {{\nabla ^n}z(k)}  ]_{k = a}}\\
 =&\hspace{-6pt} w^{-1}(k) {\nabla ^n}z(k)={\nabla ^{n,w(k)}}x(k).
\end{array}}
\end{equation}

For the case of $\alpha\to n-1$, it follows
\begin{equation}\label{Eq3.26}
{\textstyle \begin{array}{rl}
\mathop {\lim }\limits_{\alpha  \to n - 1} {}_a^{\rm C}\nabla _k^{\alpha,w(k)} x(k) =&\hspace{-6pt} \mathop {\lim }\limits_{\alpha  \to n - 1} \{ {w^{-1}(k)\sum\nolimits_{j = a + 1}^k {\frac{{ ( {k - j + 1} )\overline {^{n - \alpha  - 1}} }}{{\Gamma  ( {n - \alpha }  )}}{\nabla ^n}z(j)} } \}\\
 =&\hspace{-6pt} w^{-1}(k)\sum\nolimits_{j = a + 1}^k {{\nabla ^n}z(j)} \\
 =&\hspace{-6pt} w^{-1}(k){\nabla ^{n - 1}}z(k) - w^{-1}(k){ [ {{\nabla ^{n - 1}}z(k)}  ]_{k = a}}\\
 =&\hspace{-6pt}{\nabla ^{n - 1,w(k)}}x(k) - \frac{w(a)}{w(k)}{ [ {{\nabla ^{n - 1,w(k)}}x(k)}  ]_{k = a}},
\end{array}}
\end{equation}
which leads to (\ref{Eq3.21}) smoothly.

For any $k\in\mathbb{N}_{a+n+1}$, the following limit can be obtained
\begin{equation}\label{Eq3.27}
{\textstyle \begin{array}{l}
\mathop {\lim }\limits_{\alpha  \to n} \sum\nolimits_{i = 0}^{n - 1} {\frac{{ ( {k - a}  )\overline {^{i - \alpha }} }}{{\Gamma  ( {i - \alpha  + 1}  )}}}\frac{w(a)}{w(k)} { [ {{\nabla ^{i,w(k)}}x(k)}  ]_{k = a}}=\sum\nolimits_{i = 0}^{n - 1} \frac{\Gamma( {k - a+i-n}  )}{\Gamma  ( {i - n  + 1}  )\Gamma  ( {k - a  )}}\frac{w(a)}{w(k)} { [ {{\nabla ^{i,w(k)}}x(k)}  ]_{k = a}} =0.
\end{array}}
\end{equation}

Moreover, by using the relationship in (\ref{Eq3.2}), (\ref{Eq3.25}) and (\ref{Eq3.27}), one has the following equation for $k\in\mathbb{N}_{a+n+1}$
\begin{equation}\label{Eq3.28}
{\textstyle \begin{array}{l}
\mathop {\lim }\limits_{\alpha  \to n} {}_a^{\rm R}\nabla _k^{\alpha,w(k)} x(k) =\mathop {\lim }\limits_{\alpha  \to n} {}_a^{\rm C}\nabla _k^{\alpha,w(k)} x(k)+\mathop {\lim }\limits_{\alpha  \to n} \sum\nolimits_{i = 0}^{n - 1} {\frac{{ ( {k - a}  )\overline {^{i - \alpha }} }}{{\Gamma  ( {i - \alpha  + 1}  )}}}\frac{w(a)}{w(k)} { [ {{\nabla ^{i,w(k)}}x(k)}  ]_{k = a}} ={\nabla ^{n,w(k)}}x(k).
\end{array}}
\end{equation}

Similarly, for any $k\in\mathbb{N}_{a+n}$, it follows
\begin{equation}\label{Eq3.29}
{\textstyle \begin{array}{rl}
\mathop {\lim }\limits_{\alpha  \to n - 1} {}_a^{\rm R}\nabla _k^{\alpha,w(k)} x(k) =&\hspace{-6pt} \mathop {\lim }\limits_{\alpha  \to n - 1} w^{-1}(k)\sum\nolimits_{j = a + 1}^k {\frac{{ ( {k - j + 1} )\overline {^{n - \alpha  - 1}} }}{{\Gamma  ( {n - \alpha }  )}}{\nabla ^n}z(j)} \\
  &\hspace{-6pt}+\mathop {\lim }\limits_{\alpha  \to n - 1}w^{-1}(k)\sum\nolimits_{i = 0}^{n - 1} {\frac{{ ( {k - a}  )\overline {^{i - \alpha }} }}{{\Gamma  ( {i - \alpha  + 1}  )}}{{ [ {{\nabla ^i}z(k)}  ]}_{k = a}}} \\
 =&\hspace{-6pt} w^{-1}(k)\sum\nolimits_{j = a + 1}^k {{\nabla ^n}z(j)}  + w^{-1}(k){ [ {{\nabla ^{n - 1}}z(k)}  ]_{k = a}}\\
 =&\hspace{-6pt} w^{-1}(k){\nabla ^{n - 1}}z(k)\\
 =&\hspace{-6pt}{\nabla ^{n - 1,w(k)}}x(k),
\end{array}}
\end{equation}
Till now, (\ref{Eq3.23}) has been proved.
\end{proof}

In Theorem \ref{Theorem3.4}, the limit is the unilateral limit. Similar result for the continuous time case \cite[page 781]{Li:2007AMC} or even the classical discrete time case \cite[Theorem 3.63]{Goodrich:2015Book} have been studied. \textsf{However, the methods in the existing results do not work here. Notably, the range of suitable $k$ is provided, i.e., $k\in\mathbb{N}_{a+1}$, $k\in\mathbb{N}_{a+n+1}$, $k\in\mathbb{N}_{a+n}$, which coincides with (\ref{Eq3.16}) and actually  refines the existing results.} Besides, the range of $\alpha$ for ${ }_{a}^{\mathrm{R}} \nabla_{k}^{\alpha, w(k)} x(k)$ can be extended as $[n-1, n]$ and the range of $\alpha$ for ${ }_{a}^{\rm C} \nabla_{k}^{\alpha, w(k)} x(k)$ should be $(n-1, n]$.
\begin{theorem}\label{Theorem3.5}
For $x,x_i : \mathbb{N}_{a+1}^b \to \mathbb{R}$, $i\in\mathbb{Z}_+$, if $\{x_i\}$ converges uniformly to $x$, then for any $\alpha>0$, $k,b\in\mathbb{N}_{a+1}$, $a\in\mathbb{R}$, $w: \mathbb{N}_{a+1} \to \mathbb{R} \backslash\{0\}$, one has
\begin{equation}\label{Eq3.30}
{\textstyle \mathop {\lim }\limits_{i  \to +\infty} {}_a^{\rm G}\nabla _k^{-\alpha,w(k)}x_i(k) = {}_a^{\rm G}\nabla _k^{-\alpha,w(k)}x(k)}.
\end{equation}
\end{theorem}
\begin{proof}
Due to the given condition on uniform convergence, one obtains that for any $\varepsilon>0$, there exists $N\in\mathbb{Z}_+$, such that $|x_i(k)-x(k)|<\varepsilon$, for any $i>N$, $k\in\mathbb{N}_{a+1}^b$. Assuming $\kappa : = \mathop {\max }\limits_{k} {}_a^{\rm G}\nabla _k^{ - \alpha ,w(k)}1$ and $\kappa:=\varepsilon \kappa $, then for any $i>N$, one has
\begin{equation}\label{Eq3.31}
{\textstyle \begin{array}{rl}
|{}_a^{\rm{G}}\nabla _k^{ - \alpha ,w(k)}{x_i}(k) - {}_a^{\rm{G}}\nabla _k^{ - \alpha ,w(k)}x(k)| =&\hspace{-6pt} |{w^{ - 1}}(k)\sum\nolimits_{j = a + 1}^k {\frac{{(k - j + 1)\overline {^{\alpha  - 1}} }}{{\Gamma \left( \alpha  \right)}}w(j)[{x_i}(j) - x(j)]} |\\
 \le&\hspace{-6pt} {w^{ - 1}}(k)\sum\nolimits_{j = a + 1}^k {\frac{{(k - j + 1)\overline {^{\alpha  - 1}} }}{{\Gamma \left( \alpha  \right)}}w(j)|{x_i}(j) - x(j)|} \\
 \le&\hspace{-6pt} \varepsilon {w^{ - 1}}(k)\sum\nolimits_{j = a + 1}^k {\frac{{(k - j + 1)\overline {^{\alpha  - 1}} }}{{\Gamma \left( \alpha  \right)}}w(j)} \\
 \le&\hspace{-6pt} \varepsilon \kappa  =\epsilon,
\end{array}}
\end{equation}
which implies (\ref{Eq3.30}).
\end{proof}

Theorem \ref{Theorem3.5} is inspired by \cite[Theorem 4]{Almeida:2019RMJM}, which can be used to find the limit of ${ }_{a}^{\rm G} \nabla_{k}^{-\alpha, w(k)} x_{i}(k)$ as $i \to+\infty$ and estimate the value of ${ }_{a}^{\rm G} \nabla_{k}^{-\alpha, w(k)} x(k)$. In the two theorems, $w(k)$ is assumed to be finite positive and it is effective to extend such a condition.

\subsection{Nabla Taylor formula}
In this part, some properties regarding the nabla Taylor formula of nabla tempered fractional calculus will be developed.

\begin{lemma}\label{Lemma3.1}\cite[Theorem 3.48]{Goodrich:2015Book}
For any $x:\mathbb{N}_{a-K-1}\to\mathbb{R}$, $a\in\mathbb{R}$, $K\in\mathbb{N}$, $k\in\mathbb{N}_{a}^b$, one has
\begin{equation}\label{Eq3.35}
{\textstyle
x(k)   = \sum\nolimits_{i = 0}^{K} {\frac{ ( {k - a}  )\overline {^i} }{{i!}}[{{\nabla ^i}x(k)}]_{k=a}}  + \sum\nolimits_{j = a + 1}^k {{}\frac{{ ( {k - j + 1}  )\overline {^{K}} }}{{{K}!}}{\nabla ^{K+1}}x(j)},}
\end{equation}
which can be regarded as the nabla Taylor formula of $x(k)$ expanded at the initial instant $k=a$.
\end{lemma}
\begin{theorem}\label{Theorem3.7}
For any $x:\mathbb{N}_{a-K-1}\to\mathbb{R}$, $a\in\mathbb{R}$, $K\in\mathbb{N}$, $n\in\mathbb{Z}_+$, $K>n$, $w: \mathbb{N}_{a+1} \to \mathbb{R} \backslash\{0\}$, $k\in\mathbb{N}_{a+1}$, one has
\begin{equation}\label{Eq3.36}
\textstyle{ \begin{array}{rl}
 { {\nabla}^{n,w(k)} x(k) } =&\hspace{-6pt} \sum\nolimits_{i = n}^{K} {\frac{{ ( {k - a } )}^{\overline {i-n} }}{{  ( i-n  )! }}\frac{w(a)}{w(k)}{[{{\nabla ^{i,w(k)}}x ( {k }  )}]_{k=a}}} + \sum\nolimits_{j = a+1}^k {\frac{{{{ ( {k - j + 1}  )}^{\overline {K-n} }}}}{ \Gamma{ ( K-n )! } }}\frac{w(j)}{w(k)}{\nabla ^{K+1,w(j)}}x(j) .
  \end{array}}
\end{equation}
For any $x:\mathbb{N}_{a-K-1}\to\mathbb{R}$, $a\in\mathbb{R}$, $K\in\mathbb{N}$, $\alpha\in\mathbb{R}\backslash\mathbb{Z}_+$, $w: \mathbb{N}_{a+1} \to \mathbb{R} \backslash\{0\}$, $k\in\mathbb{N}_{a+1}$, one has
\begin{equation}\label{Eq3.37}
\textstyle{ \begin{array}{rl}
 { {}_{a}^{\rm G}{\nabla}_{k}^{\alpha,w(k)} x(k) } =&\hspace{-6pt} \sum\nolimits_{i = 0}^{K} {\frac{{ ( {k - a } )}^{\overline {i-\alpha} }}{{\Gamma  ( i-\alpha+1  ) }}\frac{w(a)}{w(k)}{[{{\nabla ^{i,w(k)}}x ( {k }  )}]_{k=a}}} + \sum\nolimits_{j = a+1}^k {\frac{{{{ ( {k - j + 1}  )}^{\overline {K-\alpha} }}}}{ \Gamma{ ( K-\alpha+1  ) } }}\frac{w(j)}{w(k)}{\nabla ^{K+1,w(j)}}x(j).
  \end{array}}
\end{equation}
For any $x:\mathbb{N}_{a-K-1}\to\mathbb{R}$, $a\in\mathbb{R}$, $K\in\mathbb{N}$, $\alpha\in(n-1,n)$, $n\in\mathbb{Z}_+$, $w: \mathbb{N}_{a+1} \to \mathbb{R} \backslash\{0\}$, $k\in\mathbb{N}_{a+1}$, one has
\begin{equation}\label{Eq3.38}
\textstyle{ \begin{array}{rl}
 { {}_{a}^{\rm R}{\nabla}_{k}^{\alpha,w(k)} x(k) } =&\hspace{-6pt} \sum\nolimits_{i = 0}^{K} {\frac{{ ( {k - a } )}^{\overline {i-\alpha} }}{{\Gamma  ( i-\alpha+1  ) }}\frac{w(a)}{w(k)}{[{{\nabla ^{i,w(k)}}x ( {k }  )}]_{k=a}}} + \sum\nolimits_{j = a+1}^k {\frac{{{{ ( {k - j + 1}  )}^{\overline {K-\alpha} }}}}{ \Gamma{ ( K-\alpha+1  ) } }}\frac{w(j)}{w(k)}{\nabla ^{K+1,w(j)}}x(j).
  \end{array}}
\end{equation}
For any $x:\mathbb{N}_{a-K-1}\to\mathbb{R}$, $a\in\mathbb{R}$, $K\in\mathbb{N}$, $\alpha\in(n-1,n)$, $n\in\mathbb{Z}_+$, $K>n$, $w: \mathbb{N}_{a+1} \to \mathbb{R} \backslash\{0\}$, $k\in\mathbb{N}_{a+1}$, one has
\begin{equation}\label{Eq3.39}
\textstyle{ \begin{array}{rl}
 { {}_{a}^{\rm C}{\nabla}_{k}^{\alpha,w(k)} x(k) } =&\hspace{-6pt} \sum\nolimits_{i = n}^{K} {\frac{{ ( {k - a } )}^{\overline {i-\alpha} }}{{\Gamma  ( i-\alpha+1  ) }}\frac{w(a)}{w(k)}{[{{\nabla ^{i,w(k)}}x ( {k }  )}]_{k=a}}} + \sum\nolimits_{j = a+1}^k {\frac{{{{ ( {k - j + 1}  )}^{\overline {K-\alpha} }}}}{ \Gamma{ ( K-\alpha+1  ) } }}\frac{w(j)}{w(k)}{\nabla ^{K+1,w(j)}}x(j).
  \end{array}}
\end{equation}
\end{theorem}
\begin{proof}
Letting $z(k):=w(k)x(k)$, then one has
\begin{equation}\label{Eq3.40}
{\textstyle
z(k)=T_K(k)+R_K(k),}
\end{equation}
where $T_K(k):=\sum\nolimits_{i = 0}^{K} {\frac{ ( {k - a}  )\overline {^i} }{{i!}}[{{\nabla ^i}z(k)}]_{k=a}}$ and $R_K(k):=\sum\nolimits_{j = a + 1}^k {{}\frac{{ ( {k - j + 1}  )\overline {^{K}} }}{{{K}!}}{\nabla ^{K+1}}z(j)}$
$= {}_{a}^{\rm G}{\nabla}_{k}^{-(K+1)}{\nabla^{K+1} z(k) }$.

Along this way, it follows
\begin{equation}\label{Eq3.41}
{\textstyle\begin{array}{rl}
{\nabla ^n}{T_K}(k)  =&\hspace{-6pt} \sum\nolimits_{i = 0}^K {{\nabla ^n}\frac{{\left( {k - a} \right)\overline {^i} }}{{i!}}[{\nabla ^i}z\left( k \right)}]_{k = a}\\
 =&\hspace{-6pt} \sum\nolimits_{i = n}^K {\frac{{{{\left( {k - a} \right)}^{\overline {i - n} }}}}{{\left( {i - n} \right)!}}[{\nabla ^i}z\left( k \right)]_{k = a}},
\end{array}}
\end{equation}

\begin{equation}\label{Eq3.42}
{\textstyle\begin{array}{rl}
{\nabla ^n}{R_K}(k)  =&\hspace{-6pt} {\nabla ^{n - 1}}\nabla \sum\nolimits_{j = a + 1}^k {\frac{{\left( {k - j + 1} \right)\overline {^K} }}{{K!}}{\nabla ^{K + 1}}z\left( j \right)} \\
 =&\hspace{-6pt} {\nabla ^{n - 1}}\sum\nolimits_{j = a + 1}^k {\nabla \frac{{\left( {k - j + 1} \right)\overline {^K} }}{{K!}}{\nabla ^{K + 1}}z\left( j \right)} + {\nabla ^{n - 1}}{\big[ {\frac{{\left( {k - j} \right)\overline {^K} }}{{K!}}{\nabla ^{K + 1}}z\left( j \right)} \big]_{j = k}}\\
 =&\hspace{-6pt} {\nabla ^{n - 1}}\sum\nolimits_{j = a + 1}^k {\frac{{\left( {k - j + 1} \right)\overline {^{K - 1}} }}{{(K - 1)!}}{\nabla ^{K + 1}}z\left( j \right)}  +\frac{{0\overline {^K} }}{{K!}} {\nabla ^{K + n}}z\left( k \right)\\
 =&\hspace{-6pt} \sum\nolimits_{j = a + 1}^k {\frac{{\left( {k - j + 1} \right)\overline {^{K - n}} }}{{(K - n)!}}{\nabla ^{K + 1}}z\left( j \right)}  + \sum\nolimits_{j = 1}^n {\frac{{0\overline {^{K+n-j}} }}{{(K+n-j)!}}{\nabla ^{K + j}}z\left( k \right)}\\
 =&\hspace{-6pt} \sum\nolimits_{j = a + 1}^k {\frac{{\left( {k - j + 1} \right)\overline {^{K - n}} }}{{(K - n)!}}{\nabla ^{K + 1}}z\left( j \right)}.
\end{array}}
\end{equation}

By using (\ref{Eq2.10}), (\ref{Eq3.41}) and (\ref{Eq3.42}), one has
\begin{equation}\label{Eq3.43}
{\textstyle\begin{array}{rl}
{\nabla ^{\alpha ,w(k)}}x(k) =&\hspace{-6pt} {w^{ - 1}}(k){\nabla ^n}z(k)\\
 =&\hspace{-6pt} {w^{ - 1}}(k){\nabla ^n}{T_K}(k) + {w^{ - 1}}(k){\nabla ^n}{R_K}(k)\\
 =&\hspace{-6pt} \sum\nolimits_{i = n}^K {\frac{{{{\left( {k - a} \right)}^{\overline {i - n} }}}}{{\left( {i - n} \right)!}}\frac{w(a)}{w(k)}[{\nabla ^{i,w(k)}}x\left( k \right)]_{k = a}} + \sum\nolimits_{j = a + 1}^k {\frac{{\left( {k - j + 1} \right)\overline {^{K - n}} }}{{(K - n)!}}\frac{w(j)}{w(k)}{\nabla ^{K + 1,w(j)}}z\left( j \right)},
\end{array}}
\end{equation}
which confirms (\ref{Eq3.36}) firmly.

In a similar way, one has
\begin{equation}\label{Eq3.44}
{\textstyle\begin{array}{rl}
{}_a^{\rm G}\nabla _k^{\alpha,w(k)}x(k) =&\hspace{-6pt} {w^{ - 1}}(k){}_a^{\rm G}\nabla _k^\alpha{ z(k)}\\
 =&\hspace{-6pt} {w^{ - 1}}(k){}_a^{\rm G}\nabla _k^\alpha{T_K}(k) + {w^{ - 1}}(k){}_a^{\rm G}\nabla _k^\alpha{R_K}(k)\\
 =&\hspace{-6pt} {w^{ - 1}}(k)\sum\nolimits_{i = 0}^K {{}_a^{\rm G}\nabla _k^\alpha\frac{{\left( {k - a} \right)\overline {^i} }}{{i!}}[{\nabla ^i}z\left( k \right)}]_{k = a}+ {w^{ - 1}}(k){}_a^{\rm G}\nabla _k^\alpha{}_{a}^{\rm G}{\nabla}_{k}^{-(K+1)}{\nabla^{K+1} z(k) }\\
 =&\hspace{-6pt} {w^{ - 1}}(k)\sum\nolimits_{i = 0}^K {\frac{{\left( {k - a} \right)\overline {^{i-\alpha}} }}{{\Gamma(i-\alpha+1)}}[{\nabla ^i}z\left( k \right)}]_{k = a}+ {w^{ - 1}}(k){}_{a}^{\rm G}{\nabla}_{k}^{\alpha-(K+1)}{\nabla^{K+1} z(k) }\\
 =&\hspace{-6pt} \sum\nolimits_{i = 0}^{K} {\frac{{ ( {k - a } )}^{\overline {i-\alpha} }}{{\Gamma  ( i-\alpha+1  ) }}\frac{w(a)}{w(k)}{[{{\nabla ^{i,w(k)}}x ( {k }  )}]_{k=a}}} + \sum\nolimits_{j = a+1}^k {\frac{{{{ ( {k - j + 1}  )}^{\overline {K-\alpha} }}}}{ \Gamma{ ( K-\alpha+1  ) } }}\frac{w(j)}{w(k)}{\nabla ^{K+1,w(j)}}x(j),
\end{array}}
\end{equation}
\begin{equation}\label{Eq3.45}
{\textstyle\begin{array}{rl}
{}_a^{\rm R}\nabla _k^{\alpha,w(k)}x(k) =&\hspace{-6pt} {w^{ - 1}}(k){}_a^{\rm R}\nabla _k^\alpha{ z(k)}\\
 =&\hspace{-6pt} {w^{ - 1}}(k)\nabla^n{}_a^{\rm G}\nabla _k^{\alpha-n}{T_K}(k) + {w^{ - 1}}(k)\nabla^n{}_a^{\rm G}\nabla _k^{\alpha-n}{R_K}(k)\\
  =&\hspace{-6pt} {w^{ - 1}}(k)\nabla^n\sum\nolimits_{i = 0}^K {\frac{{{{( {k - a} )}^{\overline {i -\alpha+ n} }}}}{{\Gamma(i-\alpha+n+1)}}[{\nabla ^i}z( k )]_{k = a}}+ {w^{ - 1}}(k)\nabla^n{}_a^{\rm G}\nabla _k^{\alpha-n-K-1}z(k)\\
 =&\hspace{-6pt} {w^{ - 1}}(k)\sum\nolimits_{i = 0}^K {\frac{{( {k - a} )\overline {^{i-\alpha}} }}{{\Gamma(i-\alpha+1)}}[{\nabla ^i}z( k )}]_{k = a}+ {w^{ - 1}}(k){}_{a}^{\rm G}{\nabla}_{k}^{\alpha-(K+1)}{\nabla^{K+1} z(k) }\\
 =&\hspace{-6pt} \sum\nolimits_{i = 0}^{K} {\frac{{ ( {k - a } )}^{\overline {i-\alpha} }}{{\Gamma  ( i-\alpha+1  ) }}\frac{w(a)}{w(k)}{[{{\nabla ^{i,w(k)}}x ( {k }  )}]_{k=a}}} + \sum\nolimits_{j = a+1}^k {\frac{{{{ ( {k - j + 1}  )}^{\overline {K-\alpha} }}}}{ \Gamma{ ( K-\alpha+1  ) } }}\frac{w(j)}{w(k)}{\nabla ^{K+1,w(j)}}x(j),
\end{array}}
\end{equation}
\begin{equation}\label{Eq3.46}
{\textstyle\begin{array}{rl}
{}_a^{\rm C}\nabla _k^{\alpha,w(k)}x(k) =&\hspace{-6pt} {w^{ - 1}}(k){}_a^{\rm C}\nabla _k^\alpha{ z(k)}\\
 =&\hspace{-6pt} {w^{ - 1}}(k){}_a^{\rm G}\nabla _k^{\alpha-n}\nabla^n{T_K}(k) + {w^{ - 1}}(k){}_a^{\rm G}\nabla _k^{\alpha-n}\nabla^n{R_K}(k)\\
  =&\hspace{-6pt} {w^{ - 1}}(k){}_a^{\rm G}\nabla _k^{\alpha-n}\sum\nolimits_{i = n}^K {\frac{{{{\left( {k - a} \right)}^{\overline {i - n} }}}}{{\left( {i - n} \right)!}}[{\nabla ^i}z\left( k \right)]_{k = a}}\\
   &\hspace{-6pt}+ {w^{ - 1}}(k){}_a^{\rm G}\nabla _k^{\alpha-n}\sum\nolimits_{j = a + 1}^k {\frac{{\left( {k - j + 1} \right)\overline {^{K - n}} }}{{(K - n)!}}{\nabla ^{K + 1}}z\left( j \right)}\\
 =&\hspace{-6pt} {w^{ - 1}}(k)\sum\nolimits_{i = n}^K {{}_a^{\rm G}\nabla _k^{\alpha-n}\frac{{\left( {k - a} \right)\overline {^{i-n}} }}{{(i-n)!}}[{\nabla ^i}z\left( k \right)}]_{k = a}\\
 &\hspace{-6pt}+ {w^{ - 1}}(k){}_a^{\rm G}\nabla _k^{\alpha-n}{}_{a}^{\rm G}{\nabla}_{k}^{-(K-n+1)}{\nabla^{K+1} z(k) }\\
 =&\hspace{-6pt} {w^{ - 1}}(k)\sum\nolimits_{i = n}^K {\frac{{\left( {k - a} \right)\overline {^{i-\alpha}} }}{{\Gamma(i-\alpha+1)}}[{\nabla ^i}z\left( k \right)}]_{k = a}\\
 &\hspace{-6pt}+ {w^{ - 1}}(k){}_{a}^{\rm G}{\nabla}_{k}^{\alpha-(K+1)}{\nabla^{K+1} z(k) }\\
 =&\hspace{-6pt} \sum\nolimits_{i = n}^{K} {\frac{{ ( {k - a } )}^{\overline {i-\alpha} }}{{\Gamma  ( i-\alpha+1  ) }}\frac{w(a)}{w(k)}{[{{\nabla ^{i,w(k)}}x ( {k }  )}]_{k=a}}} \\
  &\hspace{-6pt}+ \sum\nolimits_{j = a+1}^k {\frac{{{{ ( {k - j + 1}  )}^{\overline {K-\alpha} }}}}{ \Gamma{ ( K-\alpha+1  ) } }}\frac{w(j)}{w(k)}{\nabla ^{K+1,w(j)}}x(j).
\end{array}}
\end{equation}

All of these complete the proof.
\end{proof}

Typically, Theorem \ref{Theorem3.7} provides the nabla Taylor formula of nabla tempered fractional difference/sum expanded at the initial instant, which can be regarded as the generalization of  \cite[Corollary 1]{Wei:2019CNSNSa}. Theorem \ref{Theorem3.7} can be adopted for analysis and calculation. For example, Theorem \ref{Theorem3.1} - Theorem \ref{Theorem3.4} can be derived accordingly. To discuss the nabla Taylor series expanded at the initial instant, $\mathbb{Z}_{a}:=\{\cdots, a-2, a-1, a, a+1, a+2, \cdots\}$ is introduced first.

\begin{definition} \label{Definition3.1} \cite[Definition 3.49]{Goodrich:2015Book}
For $x:\mathbb{Z}_a\to\mathbb{R}$, if the following equation holds for $k\in\mathbb{N}_a$
\begin{equation}\label{Eq3.47}
{\textstyle
x\left( k \right)   = \sum\nolimits_{i = 0}^{+\infty} {\frac{\left( {k - a} \right)\overline {^i} }{{i!}}[{{\nabla ^i}x\left( k \right)}]_{k=a}},}
\end{equation}
then (\ref{Eq3.47}) is called the nabla Taylor series of $x$ expanded at $k=a$.
\end{definition}
\begin{theorem}\label{Theorem3.8}
If $x:\mathbb{Z}_{a}\to\mathbb{R}$ can be expanded as a nabla Taylor series at $k=a$, then for any $n\in\mathbb{Z}_+$, $w: \mathbb{N}_{a+1} \to \mathbb{R} \backslash\{0\}$, $k\in\mathbb{N}_{a+1}$, $a\in\mathbb{R}$, one has
\begin{equation}\label{Eq3.48}
{\textstyle  { {\nabla}^{n,w(k)} x(k) } =  \sum\nolimits_{i = n}^{+\infty} {\frac{{ ( {k - a }  )}^{\overline {i-n} }}{(i-n)!}\frac{w(a)}{w(k)}{[{{\nabla ^{i,w(k)}}x ( {k }  )}]_{k=a}}} .}
\end{equation}
If $x:\mathbb{Z}_{a}\to\mathbb{R}$ can be expanded as a nabla Taylor series at $k=a$, then for any $\alpha\in\mathbb{R}\backslash\mathbb{Z}_+$, $w: \mathbb{N}_{a+1} \to \mathbb{R} \backslash\{0\}$, $k\in\mathbb{N}_{a+1}$, $a\in\mathbb{R}$, one has
\begin{equation}\label{Eq3.49}
{\textstyle { {}_{a}^{\rm G}{\nabla}_{k}^{\alpha,w(k)} x(k) } = \sum\nolimits_{i = 0}^{+\infty} {\frac{{ ( {k - a }  )}^{\overline {i-\alpha} }}{{\Gamma  ( i-\alpha+1  ) } }\frac{w(a)}{w(k)}{[{{\nabla ^{i,w(k)}}x ( {k }  )}]_{k=a}}} .}
\end{equation}
If $x:\mathbb{Z}_{a}\to\mathbb{R}$ can be expanded as a nabla Taylor series at $k=a$, then for any $\alpha\in(n-1,n)$, $n\in\mathbb{Z}_+$, $w: \mathbb{N}_{a+1} \to \mathbb{R} \backslash\{0\}$, $k\in\mathbb{N}_{a+1}$, $a\in\mathbb{R}$, one has
\begin{equation}\label{Eq3.50}
{\textstyle  { {}_{a}^{\rm R}{\nabla}_{k}^{\alpha,w(k)} x(k) } = \sum\nolimits_{i = 0}^{+\infty} {\frac{{ ( {k - a }  )}^{\overline {i-\alpha} }}{{\Gamma  ( i-\alpha+1  ) } }\frac{w(a)}{w(k)}{[{{\nabla ^{i,w(k)}}x ( {k }  )}]_{k=a}}} .}
\end{equation}
If $x:\mathbb{Z}_{a}\to\mathbb{R}$ can be expanded as a nabla Taylor series at $k=a$, then for any $\alpha\in(n-1,n)$, $n\in\mathbb{Z}_+$, $w: \mathbb{N}_{a+1} \to \mathbb{R} \backslash\{0\}$, $k\in\mathbb{N}_{a+1}$, $a\in\mathbb{R}$, one has
\begin{equation}\label{Eq3.51}
{\textstyle  { {}_{a}^{\rm C}{\nabla}_{k}^{\alpha,w(k)} x(k) } = \sum\nolimits_{i = n}^{+\infty} {\frac{{ ( {k - a }  )}^{\overline {i-\alpha} }}{{\Gamma  ( i-\alpha+1  ) } }\frac{w(a)}{w(k)}{[{{\nabla ^{i,w(k)}}x ( {k }  )}]_{k=a}}} .}
\end{equation}
\end{theorem}
\begin{proof}
By using Definition \ref{Definition3.1} and the similar method with Theorem \ref{Theorem3.7}, the proof can be completed immediately.
\end{proof}

Sometimes, $x$ is called analytic like \cite[Theorem 11, Theorem 12]{Almeida:2019RMJM}. In this condition, $\mathop {\lim }\limits_{K \to  + \infty } {R_K}(k) = 0$ like \cite[Theorem 3.50]{Goodrich:2015Book}. When $\alpha=0$ in (\ref{Eq3.49}), it follows
\begin{equation}\label{Eq3.52}
{\textstyle
x\left( k \right)   = \sum\nolimits_{i = 0}^{+\infty} {\frac{\left( {k - a} \right)\overline {^i} }{{i!}}\frac{w(a)}{w(k)}[{{\nabla ^{i,w(k)}}x\left( k \right)}]_{k=a}}.}
\end{equation}
To discuss the nabla Taylor formula expanded at the future instant, a new set is introduced here $\mathbb{N}_{a}^{b}:=\{a, a+1, a+2, \cdots, b\}$.

\begin{lemma}\label{Lemma3.2}\cite[Theorem 1]{Wei:2019CNSNSa}
For any $x:\mathbb{N}_{a-K-1}\to\mathbb{R}$, $a\in\mathbb{R}$, $K\in\mathbb{N}$, $k\in\mathbb{N}_{a}^b$, $K<b-k+1$, one has
\begin{equation}\label{Eq3.53}
{\textstyle
x\left( k \right)   = \sum\nolimits_{i = 0}^{K} {\frac{\left( {k - b} \right)\overline {^i} }{{i!}}[{{\nabla ^i}x\left( k \right)}]_{k=b}} - \sum\nolimits_{j = k + 1}^b {{}\frac{{\left( {k - j + 1} \right)\overline {^{K}} }}{{{K}!}}{\nabla ^{K+1}}x\left( j \right)},}
\end{equation}
which can be regarded as the nabla Taylor formula of $x$ expanded at the future instant $k=b$.
\end{lemma}

\begin{theorem}\label{Theorem3.9}
For any $x:\mathbb{N}_{a-K-1}\to\mathbb{R}$, $a\in\mathbb{R}$, $K\in\mathbb{N}$, $\alpha\in\mathbb{R}\backslash\mathbb{Z}_+$, $w: \mathbb{N}_{a+1} \to \mathbb{R} \backslash\{0\}$, $k\in\mathbb{N}_{a+1}$, one has
\begin{equation}\label{Eq3.54}
\textstyle{ \begin{array}{rl}
 { {}_{a}^{\rm G}{\nabla}_{k}^{\alpha,w(k)} x(k) } =&\hspace{-6pt} \sum\nolimits_{i = 0}^{ K } { ( {\begin{smallmatrix}
\alpha \\
i
\end{smallmatrix}}  ) \frac{{{{ ( {k - a- i}  )}^{\overline {i - \alpha } }}}}{{\Gamma  ( {i - \alpha  + 1}  )}}{{\nabla ^{i,w(k)}} x ( {k}  )}} \\
  &\hspace{-6pt}- \sum\nolimits_{i = a + 2}^k {\frac{w(i)}{w(k)}{\nabla ^{K + 1,w(i)}}x\left( i \right)\sum\nolimits_{j = a + 2}^i {\frac{{\left( {k - j+2} \right)\overline {^{ - \alpha  - 1}} }}{{\Gamma \left( { - \alpha } \right)}}} } \frac{{\left( {j - i} \right)\overline {^K} }}{{K!}}.
  \end{array}}
\end{equation}

For any $x:\mathbb{N}_{a-K-1}\to\mathbb{R}$, $a\in\mathbb{R}$, $K\in\mathbb{N}$, $\alpha\in(n-1,n)$, $n\in\mathbb{Z}_+$, $w: \mathbb{N}_{a+1} \to \mathbb{R} \backslash\{0\}$, $k\in\mathbb{N}_{a+1}$, one has
\begin{equation}\label{Eq3.55}
\textstyle{ \begin{array}{rl}
 { {}_{a}^{\rm R}{\nabla}_{k}^{\alpha,w(k)} x(k) } =&\hspace{-6pt} \sum\nolimits_{i = 0}^{ K } { ( {\begin{smallmatrix}
\alpha \\
i
\end{smallmatrix}}  ) \frac{{{{ ( {k - a- i}  )}^{\overline {i - \alpha } }}}}{{\Gamma  ( {i - \alpha  + 1}  )}}{{\nabla ^{i,w(k)}} x ( {k}  )}} \\
  &\hspace{-6pt}- \sum\nolimits_{i = a + 2}^k {\frac{w(i)}{w(k)}{\nabla ^{K + 1,w(i)}}x\left( i \right)\sum\nolimits_{j = a + 2}^i {\frac{{\left( {k - j+2} \right)\overline {^{ - \alpha  - 1}} }}{{\Gamma \left( { - \alpha } \right)}}} } \frac{{\left( {j - i} \right)\overline {^K} }}{{K!}}.
  \end{array}}
\end{equation}

For any $x:\mathbb{N}_{a-K-1}\to\mathbb{R}$, $a\in\mathbb{R}$, $K\in\mathbb{N}$, $K\ge n$ $\alpha\in(n-1,n)$, $n\in\mathbb{Z}_+$, $w: \mathbb{N}_{a+1} \to \mathbb{R} \backslash\{0\}$, $k\in\mathbb{N}_{a+1}$, one has
\begin{equation}\label{Eq3.56}
\textstyle{ \begin{array}{rl}
 { {}_{a}^{\rm C}{\nabla}_{k}^{\alpha,w(k)} x(k) } =&\hspace{-6pt}  \sum\nolimits_{i = n}^K {\left( {\begin{smallmatrix}
{\alpha  - n}\\
{i - n}
\end{smallmatrix}} \right)\frac{{\left( {k - a - i + n} \right)\overline {^{i - \alpha }} }}{{\Gamma \left( {i - \alpha  + 1} \right)}}{\nabla ^{i,w(k)}}x\left( k \right)} \\
 &\hspace{-6pt}- \sum\nolimits_{i = a + 2}^k {\frac{w(i)}{w(k)}{\nabla ^{K + 1,w(i)}}x\left( i \right)\sum\nolimits_{j = a + 2}^i {\frac{{\left( {k - j+2} \right)\overline {^{ n- \alpha  - 1}} }}{{\Gamma \left( { n- \alpha } \right)}}} } \frac{{\left( {j - i} \right)\overline {^{K-n}} }}{{(K-n)!}}.
  \end{array}}\hspace{-12pt}
\end{equation}
\end{theorem}
\begin{proof}
Letting $z(k):=w(k)x(k)$, then one has
\begin{equation}\label{Eq3.57}
{\textstyle
z(j)=T_K(j)-R_K(j),}
\end{equation}
where $T_K(k):=\sum\nolimits_{i = 0}^K {\frac{{\left( {j - k} \right)\overline {^i} }}{{i!}}{\nabla ^i}z\left( k \right)}$, $R_K(j):=\sum\nolimits_{i = j + 1}^k {\frac{{( {j - i + 1} )\overline {^K} }}{{K!}}{\nabla ^{K + 1}}z( i )}$, $k,j\in\mathbb{N}_{a+1}$, $j\le k$, $K<k-j+1$.

By using the relationship in (\ref{Eq2.9}), one has
\begin{equation}\label{Eq3.58}
\begin{array}{rl}
{}_a^{\rm{G}}\nabla _k^{\alpha,w(k)} x\left( k \right)= &\hspace{-6pt}w^{-1}(k){}_a^{\rm{G}}\nabla _k^{\alpha} z\left( k \right)\\
 = &\hspace{-6pt} w^{-1}(k)\sum\nolimits_{j = a + 1}^k {\frac{{\left( {k - j + 1} \right)\overline {^{ - \alpha  - 1}} }}{{\Gamma \left( { - \alpha } \right)}}z\left( j \right)} \\
 = &\hspace{-6pt} w^{-1}(k)\sum\nolimits_{j = a + 1}^k {\frac{{\left( {k - j + 1} \right)\overline {^{ - \alpha  - 1}} }}{{\Gamma \left( { - \alpha } \right)}}T_K(j)}- w^{-1}(k)\sum\nolimits_{j = a + 1}^k {\frac{{\left( {k - j + 1} \right)\overline {^{ - \alpha  - 1}} }}{{\Gamma \left( { - \alpha } \right)}}R_K(j) }.
\end{array}
\end{equation}

The first term in the right hand of (\ref{Eq3.58}) can be further expressed by
\begin{equation}\label{Eq3.59}
\begin{array}{rl}
w^{-1}(k)\sum\nolimits_{j = a + 1}^k {\frac{{\left( {k - j + 1} \right)\overline {^{ - \alpha  - 1}} }}{{\Gamma \left( { - \alpha } \right)}}T_K(j)}  =&\hspace{-6pt} w^{-1}(k)\sum\nolimits_{i = 0}^K {\frac{1}{{\Gamma \left( { - \alpha } \right)i!}}{\nabla ^i}z\left( k \right)\sum\nolimits_{j = a + 1}^k {\left( {k - j + 1} \right)\overline {^{ - \alpha  - 1}} \left( {j - k} \right)\overline {^i} } } \\
 =&\hspace{-6pt} w^{-1}(k)\sum\nolimits_{i = 0}^K {\frac{1}{{\Gamma \left( { - \alpha } \right)i!}}{\nabla ^i}z\left( k \right)\sum\nolimits_{j = a + 1}^k {\frac{{\Gamma \left( {k - j - \alpha } \right)\Gamma \left( {j - k + i} \right)}}{{\Gamma \left( {k - j + 1} \right)\Gamma \left( {j - k} \right)}}} } \\
 =&\hspace{-6pt} w^{-1}(k)\sum\nolimits_{i = 0}^K {\frac{{{{( - 1)}^i}}}{{\Gamma \left( { - \alpha } \right)i!}}{\nabla ^i}z\left( k \right)\sum\nolimits_{j = a + 1}^k {\frac{{\Gamma \left( {k - j - \alpha } \right)}}{{\Gamma \left( {k - j + 1 - i} \right)}}} } \\
 =&\hspace{-6pt} w^{-1}(k)\sum\nolimits_{i = 0}^K {\left( {\begin{smallmatrix}
\alpha \\
i
\end{smallmatrix}} \right){\nabla ^i}z\left( k \right)\sum\nolimits_{j = a + 1}^k {\frac{{\left( {k - j + 1 - i} \right)\overline {^{i - \alpha  - 1}} }}{{\Gamma \left( {i - \alpha } \right)}}} } \\
 =&\hspace{-6pt} -w^{-1}(k)\sum\nolimits_{i = 0}^K {\left( {\begin{smallmatrix}
\alpha \\
i
\end{smallmatrix}} \right){\nabla ^i}z\left( k \right)\sum\nolimits_{j = a + 1}^k {\nabla \frac{{\left( {k - j  - i} \right)\overline {^{i - \alpha }} }}{{\Gamma \left( {i - \alpha  + 1} \right)}}} } \\
  =&\hspace{-6pt} w^{-1}(k)\sum\nolimits_{i = 0}^K {\left( {\begin{smallmatrix}
\alpha \\
i
\end{smallmatrix}} \right){\nabla ^i}z\left( k \right)\big[ {\frac{{\left( {k - a - i} \right)\overline {^{i - \alpha }} }}{{\Gamma \left( {i - \alpha  + 1} \right)}} - \frac{{\left( { - i} \right)\overline {^{i - \alpha }} }}{{\Gamma \left( {i - \alpha  + 1} \right)}}} \big]} \\
 =&\hspace{-6pt} w^{-1}(k)\sum\nolimits_{i = 0}^K {\left( {\begin{smallmatrix}
\alpha \\
i
\end{smallmatrix}} \right)\frac{{\left( {k - a - i} \right)\overline {^{i - \alpha }} }}{{\Gamma \left( {i - \alpha  + 1} \right)}}{\nabla ^i}z\left( k \right)}\\
=&\hspace{-6pt} \sum\nolimits_{i = 0}^K {\left( {\begin{smallmatrix}
\alpha \\
i
\end{smallmatrix}} \right)\frac{{\left( {k - a - i} \right)\overline {^{i - \alpha }} }}{{\Gamma \left( {i - \alpha  + 1} \right)}}{\nabla ^{i,w(k)}}x\left( k \right)},
\end{array}
\end{equation}
where $-{\nabla \frac{{\left( {k - j  - i} \right)\overline {^{i - \alpha }} }}{{\Gamma \left( {i - \alpha  + 1} \right)}}}={\frac{{\left( {k - j + 1 - i} \right)\overline {^{i - \alpha  - 1}} }}{{\Gamma \left( {i - \alpha } \right)}}}$ is adopted and the first order difference is taken with respect to $j$.

The second term in the right hand of (\ref{Eq3.58}) can be described as
\begin{equation}\label{Eq3.60}
{\textstyle \begin{array}{rl}
w^{-1}(k)\sum\nolimits_{j = a + 1}^k {\frac{{\left( {k - j + 1} \right)\overline {^{ - \alpha  - 1}} }}{{\Gamma \left( { - \alpha } \right)}}R_K(j)} =&\hspace{-6pt}w^{-1}(k)\sum\nolimits_{j = a + 2}^k {\frac{{\left( {k - j + 2} \right)\overline {^{ - \alpha  - 1}} }}{{\Gamma \left( { - \alpha } \right)}}\sum\nolimits_{i = j}^k {\frac{{( {j - i } )\overline {^K} }}{{K!}}{\nabla ^{K + 1}}z( i )}} \\
 =&\hspace{-6pt} w^{-1}(k)\sum\nolimits_{i = a + 2}^k {{\nabla ^{K + 1}}z\left( i \right)\sum\nolimits_{j = a + 2}^i {\frac{{\left( {k - j+2} \right)\overline {^{ - \alpha  - 1}} }}{{\Gamma \left( { - \alpha } \right)}}} } \frac{{\left( {j - i} \right)\overline {^K} }}{{K!}}\\
 =&\hspace{-6pt} \sum\nolimits_{i = a + 2}^k {\frac{w(i)}{w(k)}{\nabla ^{K + 1,w(i)}}x\left( i \right)\sum\nolimits_{j = a + 2}^i {\frac{{\left( {k - j+2} \right)\overline {^{ - \alpha  - 1}} }}{{\Gamma \left( { - \alpha } \right)}}} } \frac{{\left( {j - i} \right)\overline {^K} }}{{K!}}.
\end{array}}
\end{equation}
By substituting (\ref{Eq3.59}) and (\ref{Eq3.60}) into (\ref{Eq3.58}), the desired result in (\ref{Eq3.54}) follows.

From Theorem \ref{Theorem3.1}, one has ${}_a^{\rm{R}}\nabla _k^{\alpha,w(k)} x\left( k \right) ={}_a^{\rm{G}}\nabla _k^{\alpha,w(k)} x\left( k \right) $. Therefore, the result in (\ref{Eq3.55}) can be derived for any $\alpha\in(n-1,n)$, $n\in\mathbb{Z}_+$.

By using (\ref{Eq3.57}), one has
\begin{equation}\label{Eq3.61}
{\textstyle \begin{array}{rl}
\nabla^n z(j) =&\hspace{-6pt}\sum\nolimits_{i = 0}^{K-n} {\frac{{\left( {j - k} \right)\overline {^i} }}{{i!}}{\nabla ^{i+n}}z\left( k \right)}+\sum\nolimits_{i = j + 1}^k {\frac{{( {j - i + 1} )\overline {^{K-n}} }}{{(K-n)!}}{\nabla ^{K + 1}}z( i )}.
\end{array}}
\end{equation}

By combining (\ref{Eq2.5}), (\ref{Eq2.12}), (\ref{Eq3.54}) and (\ref{Eq3.61}), the following result can be derived
\begin{equation}\label{Eq3.62}
\begin{array}{rl}
{}_a^{\rm{C}}\nabla _k^{\alpha,w(k)} x( k ) =&\hspace{-6pt} w^{-1}(k) {}_a^{\rm{G}}\nabla _k^{\alpha  - n}{\nabla ^n}z( k )\\
 =&\hspace{-6pt}w^{-1}(k)\sum\nolimits_{j = a + 1}^k {\frac{{( {k - j + 1} )\overline {^{ n- \alpha  - 1}} }}{{\Gamma ( { n- \alpha } )}}\nabla^n z( j )} \\
 =&\hspace{-6pt} \sum\nolimits_{i = 0}^{ K-n } { ( {\begin{smallmatrix}
\alpha-n \\
i
\end{smallmatrix}}  ) \frac{{{{ ( {k - a- i}  )}^{\overline {i - \alpha +n} }}}}{{\Gamma  ( {i - \alpha+n  + 1}  )}}{{\nabla ^{i+n,w(k)}} x ( {k}  )}} \\
  &\hspace{-6pt}- \sum\nolimits_{i = a + 2}^k {\frac{w(i)}{w(k)}{\nabla ^{K + 1,w(i)}}x\left( i \right)\sum\nolimits_{j = a + 2}^i {\frac{{\left( {k - j+2} \right)\overline {^{ n- \alpha  - 1}} }}{{\Gamma \left( { n- \alpha } \right)}}} } \frac{{\left( {j - i} \right)\overline {^{K-n}} }}{{(K-n)!}}\\
 =&\hspace{-6pt} \sum\nolimits_{i = n}^K {\left( {\begin{smallmatrix}
{\alpha  - n}\\
{i - n}
\end{smallmatrix}} \right)\frac{{\left( {k - a - i + n} \right)\overline {^{i - \alpha }} }}{{\Gamma \left( {i - \alpha  + 1} \right)}}{\nabla ^{i,w(k)}}x\left( k \right)} \\
 &\hspace{-6pt}- \sum\nolimits_{i = a + 2}^k {\frac{w(i)}{w(k)}{\nabla ^{K + 1,w(i)}}x\left( i \right)\sum\nolimits_{j = a + 2}^i {\frac{{\left( {k - j+2} \right)\overline {^{ n- \alpha  - 1}} }}{{\Gamma \left( { n- \alpha } \right)}}} } \frac{{\left( {j - i} \right)\overline {^{K-n}} }}{{(K-n)!}}.
\end{array}
\end{equation}

Till now, the proof has been completed.

\end{proof}

\textsf{Notably, $z(k)$ is expanded after substituting into the definition of nabla tempered difference/sum in Theorem \ref{Theorem3.9} while $z(k)$ is expanded before substituting into the definition of nabla tempered difference/sum in Theorem \ref{Theorem3.7}. } Similar to Definition \ref{Definition3.1}, the nabla Taylor series expanded at the future instant instead of the initial instant will be introduced.

\begin{definition} \label{Definition3.2}\cite[Definition 5]{Wei:2019CNSNSa}
For $x:\mathbb{Z}_a\to\mathbb{R}$, if the following equation holds for $k\in\mathbb{N}_{a+1}^b$
\begin{equation}\label{Eq3.63}
{\textstyle
x( k )   = \sum\nolimits_{i = 0}^{+\infty} {\frac{( {k - b} )\overline {^i} }{{i!}}[{{\nabla ^i}x( k )}]_{k=b}},}
\end{equation}
then (\ref{Eq3.62}) is called the nabla Taylor series of $x$ at $k=b$.
\end{definition}

When $i>b-k$, one has $\left( {k - b} \right)^{\overline i}=(-1)^i\frac{\Gamma(b-k+1)}{\Gamma(b-k-i+1)}=0$. Consequently, (\ref{Eq3.63}) can be simplified as $x\left( k \right) = \sum\nolimits_{i = 0}^{b-k} {\frac{\left( {k - b} \right)\overline {^i} }{{i!}}[{{\nabla ^i}x\left( k \right)}]_{k=b}}$. For convenience, it is still called the nabla Taylor series.

\begin{lemma}\label{Lemma3.3}
\cite[Lemma 7.5]{Cheng:2011Book}
For $x:\mathbb{Z}_{a}\to\mathbb{R}$, $k,j\in\mathbb{N}_{a+1}$, $k\ge j$, one has
\begin{equation}\label{Eq3.64}
{\textstyle x( j ) = \sum\nolimits_{i = 0}^{k- j} {\frac{( {j - k} )^{\overline i}}{{i!}}}{{\nabla ^i}x( k )} .}
\end{equation}
\end{lemma}
\begin{proof}
Defining the identity operator $I$, i.e., $Ix( k )=x( k )$, then one has $x( {k - 1} )$ $ = x( k ) - {\nabla}x( k ) = ( {I - {\nabla}} )x( k )$. In a similar way, one has
\begin{equation}\label{Eq3.65}
{\textstyle \begin{array}{rl}
x( j ) =&\hspace{-6pt} {( {I - {\nabla}} )^{k - j}}x( k )\\
 =&\hspace{-6pt} \sum\nolimits_{i = 0}^{k - j} {{{( { - 1} )}^i}\big( {\begin{smallmatrix}
{k - j}\\
i
\end{smallmatrix}} \big){\nabla ^i}x( k )} \\
 =&\hspace{-6pt} \sum\nolimits_{i = 0}^{k - j} {{{( { - 1} )}^i}\frac{( {k - j - i + 1} )\overline {^i}}{i!}}{{{\nabla ^i}x( k )}}\\
=&\hspace{-6pt} \sum\nolimits_{i = 0}^{k - j} {{{( { - 1} )}^i}\frac{\Gamma( {k - j + 1} )}{i!\Gamma( {k - j - i + 1} )}}{{{\nabla ^i}x( k )}}\\
=&\hspace{-6pt} \sum\nolimits_{i = 0}^{k - j} {\frac{\Gamma( {j-k+i} )}{i!\Gamma( {j-k} )}}{{{\nabla ^i}x( k )}}\\
=&\hspace{-6pt} \sum\nolimits_{i = 0}^{k - j} {\frac{( {j-k} )\overline {^i}}{i!}}{{{\nabla ^i}x( k )}},
\end{array}}
\end{equation}
which completes the proof.
\end{proof}

Lemma \ref{Lemma3.3} shows that $x:\mathbb{Z}_{a}\to\mathbb{R}$ can be expanded as a nabla Taylor series at $j=k$ like (\ref{Eq3.63}) without strict conditions.

\begin{theorem}\label{Theorem3.10}
For any $x:\mathbb{Z}_{a}\to\mathbb{R}$, $\alpha\in\mathbb{R}\backslash\mathbb{Z}_+$, $w: \mathbb{N}_{a+1} \to \mathbb{R} \backslash\{0\}$, $k\in\mathbb{N}_{a+1}$, $a\in\mathbb{R}$, one has
\begin{equation}\label{Eq3.66}
\textstyle{}  {} _{a}^{\rm G}{\nabla}_{k}^{\alpha,w(k)} x(k)  = \sum\nolimits_{i = 0}^{ k-a-1 } { \big( {\begin{smallmatrix}
\alpha \\
i
\end{smallmatrix}}  \big) \frac{{{{ ( {k - a- i}  )}^{\overline {i - \alpha } }}}}{{\Gamma  ( {i - \alpha  + 1}  )}}{{\nabla ^{i,w(k)}} x ( {k}  )}}.
\end{equation}
For any $x:\mathbb{Z}_{a}\to\mathbb{R}$, $\alpha\in(n-1,n)$, $n\in\mathbb{Z}_+$, $w: \mathbb{N}_{a+1} \to \mathbb{R} \backslash\{0\}$, $k\in\mathbb{N}_{a+1}$, $a\in\mathbb{R}$, one has
\begin{equation}\label{Eq3.67}
{\textstyle
{}_a^{\rm R}\nabla _k^{\alpha,w(k)} x(k) = \sum\nolimits_{i = 0}^{k - a - 1} { \big( {\begin{smallmatrix}
\alpha \\
i
\end{smallmatrix}}  \big)\frac{ ( {k- a - i}  )^{\overline {i - \alpha } }}{{\Gamma  ( {i - \alpha  + 1}  )}}{{{\nabla ^{i,w(k)}}x(k)}}}.
}
\end{equation}
For any $x:\mathbb{Z}_{a}\to\mathbb{R}$, $\alpha\in(n-1,n)$, $n\in\mathbb{Z}_+$, $w: \mathbb{N}_{a+1} \to \mathbb{R} \backslash\{0\}$, $k\in\mathbb{N}_{a+1}$, $a\in\mathbb{R}$, one has
\begin{equation}\label{Eq3.68}
{\textstyle
{}_a^{\rm C}\nabla _{k}^{\alpha,w(k)} x (k  ) = \sum\nolimits_{i = n}^{k - a - 1+n} { \big( {\begin{smallmatrix}
{\alpha  - n}\\
{i - n}
\end{smallmatrix}}  \big)\frac{ ( {k - a - i + n}  )^{\overline {i - \alpha } }}{{\Gamma  ( {i - \alpha  + 1}  )}}{{{\nabla ^{i,w(k)}}x(k)}}}.}
\end{equation}
\end{theorem}
\begin{proof}
Letting $z(k):=w(k)x(k)$, Lemma \ref{Lemma3.3} gives
\begin{equation}\label{Eq3.69}
{\textstyle
z(j)=\sum\nolimits_{i = 0}^{k - j} {\frac{{{{\left( {j - k} \right)}^{\overline i}}}}{{i!}}{\nabla ^i}z\left( k \right)}.}
\end{equation}

By using $\sum\nolimits_{j = a + 1}^k {\sum\nolimits_{i = 0}^{k - j} {} }  = \sum\nolimits_{i = 0}^{k - a - 1} {\sum\nolimits_{j = a + 1}^{k - i} {} } $, $\Gamma (\theta )\Gamma (1 - \theta ) = \frac{\pi }{{\sin (\pi \theta )}}$ and the basic definition, it follows
\begin{equation}\label{Eq3.70}
\begin{array}{rl}
{}_a^{\rm{G}}\nabla _k^{\alpha,w(k)} x\left( k \right) =&\hspace{-6pt}w^{-1}(k){}_a^{\rm{G}}\nabla _k^\alpha z\left( k \right)\\
 =&\hspace{-6pt}w^{-1}(k) \sum\nolimits_{j = a + 1}^k {\frac{{{{\left( {k - j + 1} \right)}^{\overline { - \alpha  - 1} }}}}{{\Gamma \left( { - \alpha } \right)}}\sum\nolimits_{i = 0}^{k - j} {\frac{{{{\left( {j - k} \right)}^{\overline i}}}}{{i!}}{\nabla ^i}z\left( k \right)} } \\
 =&\hspace{-6pt}w^{-1}(k) \sum\nolimits_{i = 0}^{k - a - 1} {\frac{1}{{i!\Gamma \left( { - \alpha } \right)}}{\nabla ^i}z\left( k \right)\sum\nolimits_{j = a + 1}^{k - i} {{{\left( {j - k} \right)}^{\overline i}}{{\left( {k - j + 1} \right)}^{\overline { - \alpha  - 1} }}} } \\
 =&\hspace{-6pt}\sum\nolimits_{i = 0}^{k - a - 1} {\frac{1}{{i!\Gamma \left( { - \alpha } \right)}}{\nabla ^{i,w(k)}}x\left( k \right)\sum\nolimits_{j = a + 1}^{k - i} {\frac{{\Gamma \left( {j - k + i} \right)}}{{\Gamma \left( {j - k} \right)}}\frac{{\Gamma \left( {k - j - \alpha } \right)}}{{\Gamma \left( {k - j + 1} \right)}}} } \\
 =&\hspace{-6pt}\sum\nolimits_{i = 0}^{k - a - 1} {\frac{{{{( - 1)}^i}}}{{i!\Gamma \left( { - \alpha } \right)}}{\nabla ^{i,w(k)}}x\left( k \right)\sum\nolimits_{j = a + 1}^{k - i} {\frac{{\Gamma \left( {k - j - \alpha } \right)}}{{\Gamma \left( {1 - j + k - i} \right)}}} } \\
 =&\hspace{-6pt}\sum\nolimits_{i = 0}^{k - a - 1} {\big( {\begin{smallmatrix}
\alpha \\
i
\end{smallmatrix}} \big){\nabla ^{i,w(k)}}z\left( k \right)\sum\nolimits_{j = a + 1}^{k - i} {\frac{{\left( {k - j + 1 - i} \right)\overline {^{i - \alpha  - 1}} }}{{\Gamma \left( {i - \alpha } \right)}}} } \\
 =&\hspace{-6pt}-\sum\nolimits_{i = 0}^{k - a - 1} {\big( {\begin{smallmatrix}
\alpha \\
i
\end{smallmatrix}} \big){\nabla ^{i,w(k)}}x\left( k \right)\sum\nolimits_{j = a + 1}^{k - i} {\nabla \frac{{\left( {k - j - i} \right)\overline {^{i - \alpha }} }}{{\Gamma \left( {i - \alpha  + 1} \right)}}} } \\
 =&\hspace{-6pt}\sum\nolimits_{i = 0}^{k - a - 1} {\big( {\begin{smallmatrix}
\alpha \\
i
\end{smallmatrix}} \big){\nabla ^{i,w(k)}}x\left( k \right)\big[\frac{{\left( {k - a - i} \right)\overline {^{i - \alpha }} }}{{\Gamma \left( {i - \alpha  + 1} \right)}} - \frac{{0\overline {^{i - \alpha }} }}{{\Gamma \left( {i - \alpha  + 1} \right)}}\big]} \\
 =&\hspace{-6pt}\sum\nolimits_{i = 0}^{k - a - 1} {\big( {\begin{smallmatrix}
\alpha \\
i
\end{smallmatrix}} \big)\frac{{\left( {k - a - i} \right)\overline {^{i - \alpha }} }}{{\Gamma \left( {i - \alpha  + 1} \right)}}{\nabla ^{i,w(k)}}x\left( k \right)}.
\end{array}
\end{equation}
With the help of ${}_a^{\rm{R}}\nabla _k^{\alpha,w(k)} x\left( k \right) ={}_a^{\rm{G}}\nabla _k^{\alpha,w(k)} x\left( k \right) $ in Theorem \ref{Theorem3.1}, the result in (\ref{Eq3.64}) can be derived for any $\alpha\in(n-1,n)$, $n\in\mathbb{Z}_+$.

From the definition of Caputo tempered fractional difference and the proved result in (\ref{Eq3.67}), it follows
\begin{equation}\label{Eq3.71}
\begin{array}{rl}
{}_a^{\rm{C}}\nabla _k^{\alpha ,w(k)}x\left( k \right) =&\hspace{-6pt}{}_a^{\rm{G}}\nabla _k^{\alpha  - n,w(k)}{\nabla ^{n,w(k)}}x\left( k \right)\\
 =&\hspace{-6pt} \sum\nolimits_{i = 0}^{k-a-1} {\big( {\begin{smallmatrix}
{\alpha  - n}\\
i
\end{smallmatrix}} \big)\frac{{\left( {k - a - i} \right)\overline {^{i - \alpha  + n}} }}{{\Gamma \left( {i - \alpha  + n + 1} \right)}}{\nabla ^{i + n,w(k)}}x\left( k \right)} \\
 =&\hspace{-6pt} \sum\nolimits_{i = n}^{k-a-1+n} {\big( {\begin{smallmatrix}
{\alpha  - n}\\
{i - n}
\end{smallmatrix}} \big)\frac{{\left( {k - a - i + n} \right)\overline {^{i - \alpha }} }}{{\Gamma \left( {i - \alpha  + 1} \right)}}{\nabla ^{i,w(k)}}x\left( k \right)} .
\end{array}
\end{equation}

The proof completes here.
\end{proof}

Similar to (\ref{Eq3.69}), one has $z(k - j) = \sum\nolimits_{i = 0}^j {\frac{{{{( { - j} )}^{\overline i}}}}{{i!}}{\nabla ^i}z( k )} $. Along this way, one has
\begin{equation}\label{Eq3.72}
{\textstyle
\begin{array}{rl}
{\nabla ^{n,w(k)}}x(k) =&\hspace{-6pt} {w^{ - 1}}(k){\nabla ^n}z(k)\\
 =&\hspace{-6pt} {w^{ - 1}}(k)\sum\nolimits_{j = 0}^n {{{( - 1)}^j}} \big(\begin{smallmatrix}
n\\
j
\end{smallmatrix}\big)z(k - j)\\
 =&\hspace{-6pt} {w^{ - 1}}(k)\sum\nolimits_{j = 0}^n {{{( - 1)}^j}} \big(\begin{smallmatrix}
n\\
j
\end{smallmatrix}\big)\sum\nolimits_{i = 0}^j {\frac{{{{\left( { - j} \right)}^{\overline i}}}}{{i!}}{\nabla ^i}z(k)} \\
 =&\hspace{-6pt} \sum\nolimits_{i = 0}^n {{{{\nabla ^{i,w(k)}}x(k)}}} \sum\nolimits_{j = i}^n {( { - 1} )^{j}}\big(\begin{smallmatrix}
n\\
j
\end{smallmatrix}\big)\frac{{{{( { - j} )}^{\overline i}}}}{{i!}}\\
 =&\hspace{-6pt} \sum\nolimits_{i = 0}^n {\big(\begin{smallmatrix}
n\\
i
\end{smallmatrix}\big){\nabla ^{i,w(k)}}x(k)} \sum\nolimits_{j = i}^n {\frac{{(1 - i + j)\overline {^{i- n - 1}} }}{{\Gamma (i - n)}}} \\
 =&\hspace{-6pt} \sum\nolimits_{i = 0}^n {\big(\begin{smallmatrix}
n\\
i
\end{smallmatrix}\big){\nabla ^{i,w(k)}}x(k)} \sum\nolimits_{j = i}^n {\nabla \frac{{(1 - i + j)\overline {^{i - n}} }}{{\Gamma (i - n + 1)}}} \\
 =&\hspace{-6pt} \sum\nolimits_{i = 0}^n {\big(\begin{smallmatrix}
n\\
i
\end{smallmatrix}\big){\nabla ^{i,w(k)}}x(k)} \big[ {\frac{{(1 - i + n)\overline {^{i - n}} }}{{\Gamma (i - n + 1)}} - \frac{{0\overline {^{i - n}} }}{{\Gamma (i - n + 1)}}} \big]\\
 =&\hspace{-6pt} {\nabla ^{n,w(k)}}x(k),
\end{array}}
\end{equation}
which means that similar representation in (\ref{Eq3.48}) does not hold. The result like (\ref{Eq3.36}) can also be discussed in a similar way. It is the main reason that the expansion of $\nabla^{n,w(k)}x(k)$ is considered in Theorem \ref{Theorem3.7} and Theorem \ref{Theorem3.8} while not discussed in Theorem \ref{Theorem3.9} and Theorem \ref{Theorem3.10}.

\begin{theorem}\label{Theorem3.11}
For any $f,g:\mathbb{Z}_a\to\mathbb{R}$,  $n \in\mathbb{Z}_+$, $w: \mathbb{N}_{a+1} \to \mathbb{R} \backslash\{0\}$, $k\in\mathbb{N}_{a+1}$, one has
\begin{equation}\label{Eq3.73}
{\textstyle
\nabla^{n,w(k)}  \{ {f(k)g(k)}  \} = \sum\nolimits_{i= 0}^{n} {\big( {\begin{smallmatrix}
n \\
i
\end{smallmatrix}} \big){\nabla ^{i,w(k)}}f(k)\nabla^{n  - i}g ( {k - i}  )} .}
\end{equation}
For any $\alpha \in\mathbb{R}\backslash\mathbb{Z}_+$, $w: \mathbb{N}_{a+1} \to \mathbb{R} \backslash\{0\}$, $k\in\mathbb{N}_{a+1}$, $a\in\mathbb{R}$, one has
\begin{equation}\label{Eq3.74}
{\textstyle
{}_a^{\rm G}\nabla_k^{\alpha,w(k)}  \{ {f(k)g(k)}  \} = \sum\nolimits_{i = 0}^{k - a - 1} {\big( {\begin{smallmatrix}
\alpha \\
i
\end{smallmatrix}} \big){\nabla ^{i,w(k)}}f(k){}_a^{\rm G}\nabla_{k-i}^{\alpha  - i}g ( {k - i}  )} .}
\end{equation}
For any $\alpha \in(n-1,n)$, $n\in\mathbb{Z}_+$, $w: \mathbb{N}_{a+1} \to \mathbb{R} \backslash\{0\}$, $k\in\mathbb{N}_{a+1}$, $a\in\mathbb{R}$, one has
\begin{equation}\label{Eq3.75}
{\textstyle
{}_a^{\rm R}\nabla_k^{\alpha,w(k)}  \{ {f(k)g(k)}  \} = \sum\nolimits_{i = 0}^{k - a - 1} {\big( {\begin{smallmatrix}
\alpha \\
i
\end{smallmatrix}} \big){\nabla ^{i,w(k)}}f(k){}_a^{\rm G}\nabla_{k-i}^{\alpha  - i}g ( {k - i}  )}.}
\end{equation}
For any $\alpha \in(n-1,n)$, $n\in\mathbb{Z}_+$, $w: \mathbb{N}_{a+1} \to \mathbb{R} \backslash\{0\}$, $k\in\mathbb{N}_{a+1}$, $a\in\mathbb{R}$, one has
\begin{equation}\label{Eq3.76}
{\textstyle
\begin{array}{l}
{}_a^{\rm C}\nabla_k^{\alpha,w(k)}  \{ {f(k)g(k)}  \} =\sum\nolimits_{i = 0}^{k - a - 1} {\big( {\begin{smallmatrix}
\alpha \\
i
\end{smallmatrix}} \big){\nabla ^{i,w(k)}}f(k){}_a^{\rm G}\nabla_{k-i}^{\alpha  - i}g ( {k - i}  )}-R,
\end{array}}
\end{equation}
where $R=\sum\nolimits_{j = 0}^{n - 1} {\sum\nolimits_{i = j}^{n-1}{\big( {\begin{smallmatrix}
i\\
j
\end{smallmatrix}} \big)} \frac{( {k - a} )\overline {^{i - \alpha }}}{{\Gamma ( {i - \alpha  + 1} )}} \frac{w(a)}{w(k)}{{[{\nabla ^{j,w(k)}}f( k ){\nabla ^{i - j}}g( {k-j} )]_{k=a}}}}$.
\end{theorem}
\begin{proof}
By using the nabla Leibniz rule in \cite[Theorem 7.1]{Cheng:2011Book}, one has
\begin{equation}\label{Eq3.77}
{\textstyle
\begin{array}{rl}
{\nabla ^{n,w(k)}}\{ f(k)g(k)\}  =&\hspace{-6pt} {w^{ - 1}}(k){\nabla ^n}\{ w(k)f(k)g(k)\} \\
 =&\hspace{-6pt} {w^{ - 1}}(k)\sum\nolimits_{i = 0}^n {\big(\begin{smallmatrix}
n\\
i
\end{smallmatrix}\big){\nabla ^i}\{ w(k)f(k)\} {\nabla ^{n - i}}g(k - i)} \\
 =&\hspace{-6pt} \sum\nolimits_{i = 0}^n {\big(\begin{smallmatrix}
n\\
i
\end{smallmatrix}\big){\nabla ^{i,w(k)}}f(k){\nabla ^{n - i}}g(k - i)},
\end{array}}
\end{equation}
which is just (\ref{Eq3.73}).

By using (\ref{Eq3.66}) and (\ref{Eq3.73}), one has
\begin{equation}\label{Eq3.78}
{\textstyle
\begin{array}{rl}
{}_a^{\rm G}\nabla _k^{\alpha,w(k) }\{ {f( k )g( k )} \}=&\hspace{-6pt} \sum\nolimits_{j = 0}^{k - a - 1} {\big( {\begin{smallmatrix}
\alpha \\
j
\end{smallmatrix}} \big)\frac{{( {k - a - j} )\overline {^{j- \alpha }} }}{{\Gamma ( {j - \alpha  + 1} )}}{\nabla ^{j,w(k)}}f(k)g(k)}\\
 =&\hspace{-6pt} \sum\nolimits_{j = 0}^{k - a - 1} {\big( {\begin{smallmatrix}
{ \alpha }\\
j
\end{smallmatrix}} \big)\frac{{( {k - a - j} )\overline {^{j - \alpha }} }}{{\Gamma ( {j - \alpha  + 1} )}}\sum\nolimits_{i = 0}^j {\big( {\begin{smallmatrix}
j\\
i
\end{smallmatrix}} \big){\nabla ^{i,w(k)}}f( k ){\nabla ^{j - i}}g( {k - i} )} } .
\end{array}}
\end{equation}

By applying $\sum\nolimits_{j = 0}^{k - a - 1} {\sum\nolimits_{i = 0}^j {} }  = \sum\nolimits_{i = 0}^{k - a - 1} {\sum\nolimits_{j = i}^{k - a - 1} {} } $ and $\big( {\begin{smallmatrix}
{ \alpha }\\
{j + i}
\end{smallmatrix}} \big)\big( {\begin{smallmatrix}
{j + i}\\
i
\end{smallmatrix}} \big) = \big( {\begin{smallmatrix}
{ \alpha }\\
i
\end{smallmatrix}} \big)\big( {\begin{smallmatrix}
{ \alpha  - i}\\
j
\end{smallmatrix}} \big)$, (\ref{Eq3.78}) can be expressed as
\begin{equation}\label{Eq3.79}
{\textstyle
\begin{array}{rl}
{}_a^{\rm G}\nabla _k^{\alpha,w(k) }\left\{ {f( k )g( k )} \right\} =&\hspace{-6pt} \sum\nolimits_{i = 0}^{k - a - 1} {\sum\nolimits_{j = i}^{k - a - 1} {\big( {\begin{smallmatrix}
{ \alpha }\\
j
\end{smallmatrix}} \big)\big( {\begin{smallmatrix}
j\\
i
\end{smallmatrix}} \big){\nabla ^{i,w(k)}}f\left( k \right){\nabla ^{j - i}}g( {k - i} )\frac{{( {k - a - j} )\overline {^{j - \alpha }} }}{{\Gamma ( {j -\alpha  + 1} )}}} } \\
 =&\hspace{-6pt} \sum\nolimits_{i = 0}^{k - a - 1} {\sum\nolimits_{j = 0}^{k - a - 1 - i} {\big( {\begin{smallmatrix}
{ \alpha }\\
{j + i}
\end{smallmatrix}} \big)\big( {\begin{smallmatrix}
{j + i}\\
i
\end{smallmatrix}} \big){\nabla ^{i,w(k)}}f( k ){\nabla ^j}g( {k - i} )\frac{{( {k - a - j - i} )\overline {^{j + i - \alpha }} }}{{\Gamma ( {j + i - \alpha  + 1} )}}} }\\
 =&\hspace{-6pt} \sum\nolimits_{i = 0}^{k - a - 1} {\big( {\begin{smallmatrix}
{ \alpha }\\
i
\end{smallmatrix}} \big){\nabla ^{i,w(k)}}f( k )\sum\nolimits_{j = 0}^{k - i - a - 1} {\big( {\begin{smallmatrix}
{\alpha  - i}\\
j
\end{smallmatrix}} \big){\nabla ^j}g( {k - i} )\frac{{( {k - a - j - i} )\overline {^{j + i- \alpha }} }}{{\Gamma ( {j + i - \alpha  + 1} )}}} } \\
 =&\hspace{-6pt} \sum\nolimits_{i = 0}^{k - a - 1} {\big( {\begin{smallmatrix}
{\alpha }\\
i
\end{smallmatrix}} \big){\nabla ^{i,w(k)}}f( k ){}_a^{\rm G}\nabla _{k-i}^{ \alpha  - i}g( {k - i} )}.
\end{array}}
\end{equation}

Since (\ref{Eq3.66}) and (\ref{Eq3.67}) are similar for $\alpha\in(n-1,n)$, it is not difficult to derive (\ref{Eq3.75}) like (\ref{Eq3.74}).

By using (\ref{Eq3.2}), one has
\begin{equation}\label{Eq3.80}
{\textstyle
\begin{array}{rl}
{}_a^{\rm C}\nabla_k^{\alpha,w(k)} \{ {f( k )g( k )} \} =&\hspace{-6pt}{}_a^{\rm R}\nabla_k^{\alpha,w(k)} \{ {f( k )g( k )} \} -\sum\nolimits_{i = 0}^{n - 1} {\frac{( {k - a } )\overline {^{i - \alpha }}}{{\Gamma ( {i - \alpha  + 1} )}} \frac{w(a)}{w(k)}{ {[ {\nabla^{i,w(k)}\{ {f( k )g( k )} \}} ]}_{k = a}}} \\
 =&\hspace{-6pt} \sum\nolimits_{i = 0}^{k - a - 1} {\big( {\begin{smallmatrix}
{\alpha }\\
i
\end{smallmatrix}} \big){\nabla ^{i,w(k)}}f( k ){}_a^{\rm G}\nabla _{k-i}^{ \alpha  - i}g( {k - i} )}\\
&\hspace{-6pt} - \sum\nolimits_{i = 0}^{n - 1} {\sum\nolimits_{j = 0}^i {\big( {\begin{smallmatrix}
i\\
j
\end{smallmatrix}} \big)} \frac{( {k - a} )\overline {^{i - \alpha }}}{{\Gamma ( {i - \alpha  + 1} )}} \frac{w(a)}{w(k)}{{[{\nabla ^{j,w(k)}}f( k ){\nabla ^{i - j}}g( {k-j} )]_{k=a}}}} \\
 = &\hspace{-6pt}\sum\nolimits_{i = 0}^{k - a - 1} {\big( {\begin{smallmatrix}
{\alpha }\\
i
\end{smallmatrix}} \big){\nabla ^{i,w(k)}}f( k ){}_a^{\rm G}\nabla _{k-i}^{ \alpha  - i}g( {k - i} )}\\
&\hspace{-6pt} - \sum\nolimits_{j = 0}^{n - 1} {\sum\nolimits_{i = j}^{n-1}{\big( {\begin{smallmatrix}
i\\
j
\end{smallmatrix}} \big)} \frac{( {k - a} )\overline {^{i - \alpha }}}{{\Gamma ( {i - \alpha  + 1} )}} \frac{w(a)}{w(k)}{{[{\nabla ^{j,w(k)}}f( k ){\nabla ^{i - j}}g( {k-j} )]_{k=a}}}},
\end{array}}
\end{equation}
which confirms the correctness of (\ref{Eq3.76}).
\end{proof}

\begin{theorem}\label{Theorem3.12}
For any $x: \mathbb{N}_{a+1-n} \to \mathbb{R}$, $\alpha\in(n - 1,n)$, $n\in\mathbb{Z}_+$, $k\in\mathbb{N}_{a+1}$, $a\in\mathbb{R}$, $w: \mathbb{N}_{a+1} \to \mathbb{R} \backslash\{0\}$, one has
\begin{equation}\label{Eq3.81}
{\textstyle
\mathop {\lim }\limits_{k \to  + \infty } {[{}_a^{\rm R}\nabla _k^{\alpha ,w(k)}x(k) - {}_a^{\rm{C}}\nabla _k^{\alpha ,w(k)}x(k)]}=0,}
\end{equation}
\begin{equation}\label{Eq3.82}
{\textstyle
\mathop {\lim }\limits_{a \to  - \infty } {[{}_a^{\rm R}\nabla _k^{\alpha ,w(k)}x(k) - {}_a^{\rm{C}}\nabla _k^{\alpha ,w(k)}x(k)]}=0.}
\end{equation}
\end{theorem}
\begin{proof}
By using the relationship ${}_a^{\rm R}\nabla _{k}^{\alpha,w(k)} x(k) = {}_a^{\rm C}\nabla _{k }^{\alpha,w(k)} x(k)$ $+\sum\nolimits_{i = 0}^{n - 1} {\frac{{ ( {k - a}  )\overline {^{i - \alpha }} }}{{\Gamma  ( {i - \alpha  + 1}  )}}}\frac{w(a)}{w(k)} { [ {{\nabla ^{i,w(k)}}x(k)}  ]_{k = a}}$,  $\mathop {\lim }\limits_{p \to  + \infty } \frac{{p\overline {^q} }}{{{p^q}}} = 1$ and the boundedness of $w(k)$, ${ [ {{\nabla ^{i,w(k)}}x(k)}  ]_{k = a}}$, it follows
 \begin{equation}\label{Eq3.83}
{\textstyle
\mathop {\lim }\limits_{k \to  + \infty } \sum\nolimits_{i = 0}^{n - 1} {\frac{{ ( {k - a}  )\overline {^{i - \alpha }} }}{{\Gamma  ( {i - \alpha  + 1}  )}}}\frac{w(a)}{w(k)} { [ {{\nabla ^{i,w(k)}}x(k)}  ]_{k = a}} =0,}
\end{equation}
 \begin{equation}\label{Eq3.84}
{\textstyle
\mathop {\lim }\limits_{a \to  - \infty }\sum\nolimits_{i = 0}^{n - 1} {\frac{{ ( {k - a}  )\overline {^{i - \alpha }} }}{{\Gamma  ( {i - \alpha  + 1}  )}}}\frac{w(a)}{w(k)} { [ {{\nabla ^{i,w(k)}}x(k)}  ]_{k = a}}=0.}
\end{equation}
On this basis, the desired results in (\ref{Eq3.81}) and (\ref{Eq3.82}) can be developed.
\end{proof}

Theorem \ref{Theorem3.11} and Theorem \ref{Theorem3.12} are the applications of the developed nabla Taylor series. Along this way, different properties can be build. To keep it simple, the extension will not be discussed further.

\subsection{Nabla Laplace transform}
In this part, some properties on the nabla Laplace transform of nabla tempered fractional calculus will be developed.

\begin{definition} \label{Definition3.3}
\cite[Theorem 3.65]{Goodrich:2015Book}
For $x:\mathbb{N}_{a+1} \to \mathbb{R}$, its nabla Laplace transform is defined by
\begin{equation}\label{Eq3.85}
{\textstyle {{\mathscr N}_a}\{ {x( k )} \} := \sum\nolimits_{j = 1}^{ + \infty } {{{( {1 - s} )}^{j - 1}}x( {j + a} )},}
\end{equation}
where $s\in\mathbb{C}$.
\end{definition}

From the existing results in \cite{Goodrich:2015Book,Wei:2019FDTA}, it can be observed that the region of convergence for the infinite series in (\ref{Eq3.85}) is not empty.

\begin{theorem}\label{Theorem3.14}
If the nabla Laplace transform of $x:\mathbb{N}_{a+1} \to \mathbb{R}$ converges for $|s-1| < r$, $r > 0$, then for any $\alpha \in\mathbb{R}\backslash \mathbb{Z}_+$, $\lambda\neq1$, one has
\begin{equation}\label{Eq3.86}
{\textstyle {{\mathscr N}_a}\big\{ {{}_a^{\rm G}\nabla _{k }^{\alpha,\lambda} x(k)} \big\} = {\big( {\frac{{s - \lambda }}{{1 - \lambda }}} \big)^\alpha }{{\mathscr N}_a} \{ {x(k)}  \},}
\end{equation}
where $|s-1|<\min\{r, |1-\lambda|\}$.
\end{theorem}
\begin{proof}
Letting $X(s):={{\mathscr N}_a} \{ {x(k)}  \}$, the nabla Laplace transform of ${{{ ( {1 - \lambda }  )}^{k - a}}x(k)}$ can be calculated as
\begin{equation}\label{Eq3.87}
{\textstyle \begin{array}{rl}
{{\mathscr N}_a}\{ {{{ ( {1 - \lambda }  )}^{k - a}}x(k)} \} =&\hspace{-6pt} \sum\nolimits_{k = 1}^{ + \infty } {{{ ( {1 - \lambda }  )}^k}x ( {k + a}  ){{ ( {1 - s}  )}^{k - 1}}} \\
 =&\hspace{-6pt}  ( {1 - \lambda }  )\sum\nolimits_{k = 1}^{ + \infty } {x ( {k + a}  ){{ [ {1 - ( {s + \lambda  - \lambda s}  )}  ]}^{k - 1}}} \\
 =&\hspace{-6pt}  ( {1 - \lambda }  )X ( {s + \lambda  - \lambda s}  ),
\end{array}}
\end{equation}
where the region of convergence satisfies $ | {s - 1}  | < { | {1 - \lambda }  |^{ - 1}}r$.

By using \cite[Lemma 4, Lemma 6, Theorem 14]{Wei:2019FDTA}, one has
\begin{equation}\label{Eq3.88}
{\textstyle \begin{array}{rl}
{{\mathscr N}_a}\big\{ {{}_a^{\rm G}\nabla _k^{\alpha} [{{ ( {1 - \lambda }  )}^{k - a}}x(k)]} \big\} =&\hspace{-6pt} {{\mathscr N}_a}\big\{ {\frac{{ ( {k - a}  )\overline {^{ - \alpha  - 1}} }}{{\Gamma ( { - \alpha }  )}} \ast [{{{ ( {1 - \lambda }  )}^{k - a}}x(k)}]} \big\}\\
 =&\hspace{-6pt} {{\mathscr N}_a}\big\{ {\frac{{ ( {k - a}  )\overline {^{ - \alpha  - 1}} }}{{\Gamma ( { - \alpha }  )}}} \big\}{{\mathscr N}_a}\{ {{{{ ( {1 - \lambda }  )}^{k - a}}x(k)}} \}\\
 =&\hspace{-6pt}  ( {1 - \lambda }  ){s^\alpha }X ( {s + \lambda  - \lambda s}  ),
\end{array}}
\end{equation}
where $|s-1|<\min\{1,{ | {1 - \lambda }  |^{ - 1}}r\}$ and the convolution operation $f(k) \ast g(k) := \sum\nolimits_{j = a + 1}^k {f ( {k + a + 1 - j}  )g(j)} $.

With the help of \cite[Theorem 5]{Wei:2019FDTA}, one has
\begin{equation}\label{Eq3.89}
{\textstyle \begin{array}{rl}
{{\mathscr N}_a}\big\{ {_a^{\rm G}\nabla _k^{\alpha ,\lambda }x(k)} \big\} =&\hspace{-6pt} {\frac{1}{1 - \lambda } }{{\mathscr N}_a}{\big\{ {{}_a^{\rm G}\nabla _k^\alpha [{{ ( {1 - \lambda }  )}^{k - a}}x(k)]} \big\}_{s = \frac{{s - \lambda }}{{1 - \lambda }}}}\\
 =&\hspace{-6pt} {\big( {\frac{{s - \lambda }}{{1 - \lambda }}} \big)^\alpha }X\big( {\frac{{s - \lambda }}{{1 - \lambda }} + \lambda  - \lambda \frac{{s - \lambda }}{{1 - \lambda }}} \big)\\
 =&\hspace{-6pt} {\big( {\frac{{s - \lambda }}{{1 - \lambda }}} \big)^\alpha }X(s),
\end{array}}
\end{equation}
where $|s-1|<\min\{r, |1-\lambda|\}$.
\end{proof}

\begin{theorem}\label{Theorem3.15}
If the nabla Laplace transform of $x:\mathbb{N}_{a-n+1} \to \mathbb{R}$ converges for $|s-1| < r$, $r > 0$, then for any $\alpha \in(n-1,n)$, $n\in\mathbb{Z}_+$, $\lambda\neq1$, one has
\begin{equation}\label{Eq3.90}
{\textstyle {{\mathscr N}_a}\big\{ {\nabla ^{n,\lambda} x(k)} \big\} = {\big( {\frac{{s - \lambda }}{{1 - \lambda }}} \big)^n }{{\mathscr N}_a} \{ {x(k)}  \} - \frac{1}{{1 - \lambda }}\sum\nolimits_{i = 0}^{n - 1} {{{\big( {\frac{{s - \lambda }}{{1 - \lambda }}} \big)}^{i}}{{[ {{\nabla ^{n  - i - 1,\lambda}}x(k)} ]}_{k = a}}} ,}
\end{equation}
\begin{equation}\label{Eq3.91}
{\textstyle {{\mathscr N}_a}\big\{ {\nabla ^{n,\lambda} x(k)} \big\} = {\big( {\frac{{s - \lambda }}{{1 - \lambda }}} \big)^n }{{\mathscr N}_a} \{ {x(k)}  \} - \frac{1}{{1 - \lambda }}\sum\nolimits_{i = 0}^{n - 1} {{{\big( {\frac{{s - \lambda }}{{1 - \lambda }}} \big)}^{n  - i - 1}}{{[ {{\nabla ^{i,\lambda}}x(k)} ]}_{k = a}}} ,}
\end{equation}
\begin{equation}\label{Eq3.92}
{\textstyle {{\mathscr N}_a}\big\{ {{}_a^{\rm R}\nabla _{k }^{\alpha,\lambda} x(k)} \big\} = {\big( {\frac{{s - \lambda }}{{1 - \lambda }}} \big)^\alpha }{{\mathscr N}_a} \{ {x(k)}  \} - \frac{1}{{1 - \lambda }}\sum\nolimits_{i = 0}^{n - 1} {{{\big( {\frac{{s - \lambda }}{{1 - \lambda }}} \big)}^i}{{[ {{}_a^{\rm R}\nabla _k^{\alpha  - i - 1,\lambda}x(k)} ]}_{k = a}}} ,}\hspace{-6pt}
\end{equation}
\begin{equation}\label{Eq3.93}
{\textstyle {{\mathscr N}_a}\big\{ {{}_a^{\rm C}\nabla _{k }^{\alpha,\lambda} x(k)} \big\} = {\big( {\frac{{s - \lambda }}{{1 - \lambda }}} \big)^\alpha }{{\mathscr N}_a} \{ {x(k)}  \} - \frac{1}{{1 - \lambda }}\sum\nolimits_{i = 0}^{n - 1} {{{\big( {\frac{{s - \lambda }}{{1 - \lambda }}} \big)}^{\alpha  - i - 1}}{{[ {{\nabla ^{i,\lambda}}x(k)} ]}_{k = a}}} ,}
\end{equation}
where $|s-1|<\min\{r, |1-\lambda|\}$.
\end{theorem}
\begin{proof}
Letting $X(s):={{\mathscr N}_a} \{ {x(k)}  \}$, by using \cite[Lemma 10]{Wei:2019FDTA}, one has
\begin{equation}\label{Eq3.94}
\begin{array}{rl}
{{\mathscr N}_a}\big\{ {{\nabla ^n}[ {{{ ( {1 - \lambda }  )}^{k - a}}x(k)} ]} \big\}  =&\hspace{-6pt} {s^n}{{\mathscr N}_a}\big\{ {{{ ( {1 - \lambda }  )}^{k - a}}x(k)} \big\} - {\sum\nolimits_{i = 0}^{n - 1} {{s^i}[ {{\nabla ^{n - 1 - i}}{{ ( {1 - \lambda }  )}^{k - a}}x(k)} ]} _{k = a}}\\
 =&\hspace{-6pt}{s^n} ( {1 - \lambda }  )X ( {\lambda  + s - \lambda s}  ) - {\sum\nolimits_{i = 0}^{n - 1} {{s^i}[ {{\nabla ^{n - 1 - i,\lambda }}x(k)} ]} _{k = a}}\\
 =&\hspace{-6pt} {s^n} ( {1 - \lambda }  )X ( {\lambda  + s - \lambda s}  ) - {\sum\nolimits_{i = 0}^{n - 1} {{s^{n - i - 1}}[ {{\nabla ^{i,\lambda }}x(k)} ]} _{k = a}}.
\end{array}
\end{equation}
Theorem 5 of \cite{Wei:2019FDTA} further leads to
\begin{equation}\label{Eq3.95}
{\textstyle {{\mathscr N}_a} \{ {{\nabla ^{n,\lambda }}x(k)}  \} = \frac{1}{{1 - \lambda }}{{\mathscr N}_a}{\big\{ {{\nabla ^n}[ {{{ ( {1 - \lambda }  )}^{k - a}}x(k)} ]} \big\}_{s = \frac{{s - \lambda }}{{1 - \lambda }}}}.}
\end{equation}
Combining (\ref{Eq3.94}) with (\ref{Eq3.95}), one has the results in (\ref{Eq3.90}) and (\ref{Eq3.91}).

By using \cite[Lemma 12, Lemma 13]{Wei:2019FDTA}, one has
\begin{equation}\label{Eq3.96}
{\textstyle \begin{array}{rl}
{{\mathscr N}_a}\{ {}_a^{\rm{R}}\nabla _k^\alpha [{(1 - \lambda )^{k - a}}x(k)]\}  =&\hspace{-6pt} {s^\alpha }{{\mathscr N}_a}\{ {(1 - \lambda )^{k - a}}x(k)\}  - \sum\nolimits_{i = 0}^{n - 1} {s^i}[ {}_a^{\rm{R}}\nabla _k^{\alpha  - i - 1}{(1 - \lambda )^{k - a}}x(k)]_{k = a}\\
 =&\hspace{-6pt} {s^\alpha }(1 - \lambda )X(\lambda  + s - \lambda s) - \sum\nolimits_{i = 0}^{n - 1} {s^i}[ {}_a^{\rm{R}}\nabla _k^{\alpha  - i - 1,\lambda }x(k)]_{k = a},
\end{array}}
\end{equation}
\begin{equation}\label{Eq3.97}
{\textstyle \begin{array}{rl}
{{\mathscr N}_a}\{ {}_a^{\rm{C}}\nabla _k^\alpha [{(1 - \lambda )^{k - a}}x(k)]\}  =&\hspace{-6pt} {s^\alpha }{{\mathscr N}_a}\{ {(1 - \lambda )^{k - a}}x(k)\}  - \sum\nolimits_{i = 0}^{n - 1} {s^{\alpha  - i - 1}}[ {\nabla ^i}{(1 - \lambda )^{k - a}}x(k)]_{k = a}\\
 =&\hspace{-6pt} {s^\alpha }(1 - \lambda )X(\lambda  + s - \lambda s) - \sum\nolimits_{i = 0}^{n - 1} {{s^{\alpha  - i - 1}}[} {\nabla ^{i,\lambda }}x(k)]_{k = a}.
\end{array}}
\end{equation}

From \cite[Theorem 5]{Wei:2019FDTA}, one has
\begin{equation}\label{Eq3.98}
{\textstyle {{\mathscr N}_a}\{ {}_a^{\rm R}\nabla _k^{\alpha ,\lambda }x(k)\}  = \frac{1}{{1 - \lambda }}{{\mathscr N}_a}{\{ {}_a^{\rm R}\nabla _k^\alpha [{(1 - \lambda )^{k - a}}x(k)]\} _{s = \frac{{s - \lambda }}{{1 - \lambda }}}},}
\end{equation}
\begin{equation}\label{Eq3.99}
{\textstyle {{\mathscr N}_a}\{ {}_a^{\rm C}\nabla _k^{\alpha ,\lambda }x(k)\}  = \frac{1}{{1 - \lambda }}{{\mathscr N}_a}{\{ {}_a^{\rm C}\nabla _k^\alpha [{(1 - \lambda )^{k - a}}x(k)]\} _{s = \frac{{s - \lambda }}{{1 - \lambda }}}}.}
\end{equation}
Along this way, (\ref{Eq3.92}) and (\ref{Eq3.93}) can be obtained.
\end{proof}

By using variable substitution $z(k):=w(k)x(k)$ and Theorem \ref{Theorem3.15}, the following corollary can be developed.
\begin{corollary}\label{Corollary3.1}
Given ${}_a^{\rm C}\nabla _k^{\alpha ,w(k) }x(k) = \mu x(k)$ with $ \alpha \in(0,1)$, finite nonzero $w(k)$, $\mu\in\mathbb{R}\backslash\{1\}$, $x(a)\in\mathbb{R}$, then one has $x(k) = w^{-1}(k){{\mathcal F}_{\alpha ,1}} ( {\mu ,k,a}  )x ( a  )$, $k\in\mathbb{N}_{a+1}$, where ${{\mathcal F}_{\alpha ,\beta}} ( {\mu ,k,a}  ):={{\mathscr N}_a^{-1}}\{\frac{s^{\alpha-\beta}}{s^\alpha-\mu}\}$.
\end{corollary}

For analytic $x:\mathbb{Z}_{a}\to\mathbb{R}$, by using Theorem \ref{Theorem3.8} and \cite[Theorem 5]{Wei:2019FDTA}, one has
\begin{equation}\label{Eq3.100}
\begin{array}{rl}
{{\mathscr N}_a}\{ {{\nabla ^{n,\lambda}}x( k )} \} =&\hspace{-6pt} {{\mathscr N}_a}\big\{ {\sum\nolimits_{i = n}^{ + \infty } {\frac{{(k - a)\overline {^{i - n}} }}{{(i - n)!}}(1-\lambda)^{a-k}{{[ {{\nabla ^{i,\lambda}}x( k )} ]}_{k = a}}} } \big\}\\
 =&\hspace{-6pt} \sum\nolimits_{i = n}^{ + \infty } {{{\mathscr N}_a}\big\{ {\frac{{(k - a)\overline {^{i - n}} }}{{(i - n)!}}} (1-\lambda)^{a-k}\big\}{{[ {{\nabla ^{i,\lambda}}x( k )} ]}_{k = a}}} \\
 =&\hspace{-6pt} \sum\nolimits_{i = n}^{ + \infty }\frac{1}{1-\lambda} {\frac{1}{{{(\frac{s-\lambda}{1-\lambda})^{i - n + 1}}}}{{[ {{\nabla ^{i,\lambda}}x( k )} ]}_{k = a}}} \\
 =&\hspace{-6pt}{\big(\frac{s-\lambda}{1-\lambda}\big)^n} \sum\nolimits_{i = 0}^{ + \infty } {\frac{1}{1-\lambda}\frac{1}{{{(\frac{s-\lambda}{1-\lambda})^{i + 1}}}}{{[ {{\nabla ^{i,\lambda}}x( k )} ]}_{k = a}}}\\
   &\hspace{-6pt}- {\big(\frac{s-\lambda}{1-\lambda}\big)^n}\sum\nolimits_{i = 0}^{n - 1} {\frac{1}{1-\lambda}\frac{1}{{{(\frac{s-\lambda}{1-\lambda})^{i + 1}}}}{{[ {{\nabla ^{i,\lambda}}x( k )} ]}_{k = a}}} \\
 =&\hspace{-6pt} {\big(\frac{s-\lambda}{1-\lambda}\big)^n}{{\mathscr N}_a}\{ {x( k )} \} - \sum\nolimits_{i = 0}^{n - 1} {{\big(\frac{s-\lambda}{1-\lambda}\big)^{n - i - 1}}{{[ {{\nabla ^i}x( k )} ]}_{k = a}}},
\end{array}
\end{equation}
which confirms the result in (\ref{Eq3.91}).

For analytic $x:\mathbb{Z}_{a}\to\mathbb{R}$, $\alpha\in\mathbb{R}\backslash\mathbb{Z}_+$, it follows
\begin{equation}\label{Eq3.101}
\begin{array}{rl}
{{\mathscr N}_a}\{ {{}_a^{\rm G}\nabla _k^{\alpha,\lambda} x( k )} \} =&\hspace{-6pt} {{\mathscr N}_a}\big\{ {\sum\nolimits_{i = 0}^{ + \infty } {\frac{{(k - a)\overline {^{i - \alpha }} }}{{\Gamma ( {i - \alpha  + 1} )}}(1-\lambda)^{a-k}{{[ {{\nabla ^{i,\lambda}}x( k )} ]}_{k = a}}} } \big\}\\
 =&\hspace{-6pt} \sum\nolimits_{i = 0}^{ + \infty } {{{\mathscr N}_a}\big\{ {\frac{{(k - a)\overline {^{i - \alpha }} }}{{\Gamma ( {i - \alpha  + 1} )}}} (1-\lambda)^{a-k}\big\}{{[ {{\nabla ^{i,\lambda}}x( k )} ]}_{k = a}}} \\
 =&\hspace{-6pt}\sum\nolimits_{i = 0}^{ + \infty } {\frac{1}{1-\lambda}\frac{1}{{{(\frac{s-\lambda}{1-\lambda})^{i - \alpha  + 1}}}}{{[ {{\nabla ^{i,\lambda}}x( k )} ]}_{k = a}}} \\
 =&\hspace{-6pt} {\big(\frac{s-\lambda}{1-\lambda}\big)^\alpha }\sum\nolimits_{i = 0}^{ + \infty } {\frac{1}{1-\lambda}\frac{1}{{{(\frac{s-\lambda}{1-\lambda})^{i + 1}}}}{{[ {{\nabla ^{i,\lambda}}x( k )} ]}_{k = a}}} \\
 =&\hspace{-6pt}{\big(\frac{s-\lambda}{1-\lambda}\big)^\alpha }{{\mathscr N}_a}\{ {x( k )} \},
\end{array}
\end{equation}
which coincides with (\ref{Eq3.92}).

Due to the equivalent series representation of ${}_a^{\rm G}\nabla _k^{\alpha,\lambda} x( k )$ and ${}_a^{\rm R}\nabla _k^{\alpha,\lambda} x( k )$, one has ${{\mathscr N}_a}\{ {{}_a^{\rm R}\nabla _k^{\alpha,\lambda} x( k )} \}={\big(\frac{s-\lambda}{1-\lambda}\big)^\alpha }{{\mathscr N}_a}\{ {x( k )} \}$ which implies $[{{}_a^{\rm R}\nabla_k^{\alpha - i - 1,\lambda}}x\left( k \right)]_{k=a}=0$.

For analytic $x:\mathbb{Z}_{a}\to\mathbb{R}$, $\alpha\in(n-1,n)$, one has
\begin{equation}\label{Eq3.102}
\begin{array}{rl}
{{\mathscr N}_a}\{ {{}_a^{\rm{C}}\nabla _k^{\alpha,\lambda} x( k )} \} =&\hspace{-6pt} {{\mathscr N}_a}\big\{ {\sum\nolimits_{i = n}^{ + \infty } {\frac{{(k - a)\overline {^{i - \alpha }} }}{{\Gamma ( {i - \alpha  + 1} )}}(1-\lambda)^{a-k}{{[ {{\nabla ^{i,\lambda}}x( k )} ]}_{k = a}}} } \big\}\\
 =&\hspace{-6pt} \sum\nolimits_{i = n}^{ + \infty } {{{\mathscr N}_a}\big\{ {\frac{{(k - a)\overline {^{i - \alpha }} }}{{\Gamma ( {i - \alpha  + 1} )}}} (1-\lambda)^{a-k}\big\}{{[ {{\nabla ^{i,\lambda}}x( k )} ]}_{k = a}}} \\
 =&\hspace{-6pt} \sum\nolimits_{i = n}^{ + \infty } {\frac{1}{1-\lambda}\frac{1}{{{(\frac{s-\lambda}{1-\lambda})^{i - \alpha  + 1}}}}{{[ {{\nabla ^{i,\lambda}}x( k )} ]}_{k = a}}} \\
 =&\hspace{-6pt}{\big(\frac{s-\lambda}{1-\lambda}\big)^\alpha }\sum\nolimits_{i = 0}^{ + \infty } {\frac{1}{1-\lambda}\frac{1}{{{(\frac{s-\lambda}{1-\lambda})^{i + 1}}}}{{[ {{\nabla ^{i,\lambda}}x( k )} ]}_{k = a}}} \\
  &\hspace{-6pt}- {\big(\frac{s-\lambda}{1-\lambda}\big)^\alpha }\sum\nolimits_{i = 0}^{n - 1} {\frac{1}{1-\lambda}\frac{1}{{{(\frac{s-\lambda}{1-\lambda})^{i + 1}}}}{{[ {{\nabla ^{i,\lambda}}x( k )} ]}_{k = a}}} \\
 =&\hspace{-6pt} {\big(\frac{s-\lambda}{1-\lambda}\big)^\alpha }{{\mathscr N}_a}\{ {x( k )} \} - \sum\nolimits_{i = 0}^{n - 1} {{\big(\frac{s-\lambda}{1-\lambda}\big)^{\alpha  - i - 1}}{{[ {{\nabla ^{i,\lambda}}x( k )} ]}_{k = a}}},
\end{array}
\end{equation}
which coincides with (\ref{Eq3.94}).

\begin{theorem}\label{Theorem3.16}
For any $\alpha\in\mathbb{R}\backslash\mathbb{Z}_+$, $\lambda\neq1$, $x,y:\mathbb{N}_{a+1}\to\mathbb{R}$, $a\in\mathbb{R}$, one has
\begin{equation}\label{Eq3.103}
{\textstyle \sum\nolimits_{j = a + 1}^k {x(k + a + 1 - j){}_a^{\rm G}\nabla _{j }^{  \alpha,\lambda }y(j)}  = \sum\nolimits_{j = a + 1}^k {{}_a^{\rm G}\nabla _{j }^{ \alpha,\lambda }x(j)y ( {k + a + 1-j}  )}.}
\end{equation}
\end{theorem}
\begin{proof}
By using \cite[Lemma 6]{Wei:2019FDTA}, one has
\begin{equation}\label{Eq3.104}
{\textstyle \begin{array}{rl}
{{\mathscr N}_a}\big\{ {\sum\nolimits_{j = a + 1}^k {x(k + a + 1 - j){}_a^{\rm G}\nabla _{j }^{\alpha,\lambda }y(j)} } \big\} =&\hspace{-6pt} {{\mathscr N}_a}\big\{ {x(k) \ast {}_a^{\rm G}\nabla _{k }^{\alpha,\lambda }y(k)} \big\}\\
 =&\hspace{-6pt} {{\mathscr N}_a} \{ {x(k)}  \}{{\mathscr N}_a}\big\{ {{}_a^{\rm G}\nabla _{k }^{\alpha,\lambda }y(k)} \big\}\\
 =&\hspace{-6pt} {\big( {\frac{{s - \lambda }}{{1 - \lambda }}} \big)^{\alpha }}{{\mathscr N}_a} \{ {x(k)} \}{{\mathscr N}_a} \{ {y(k)}  \},
\end{array}}
\end{equation}
\begin{equation}\label{Eq3.105}
{\textstyle \begin{array}{rl}
{{\mathscr N}_a}\big\{ {\sum\nolimits_{j = a + 1}^k {{}_a^{\rm G}\nabla _{j }^{\alpha,\lambda }x(j)y ( {k + a + 1 - j}  )} } \big\}
 =&\hspace{-6pt} {{\mathscr N}_a}\big\{ {{}_a^{\rm G}\nabla _{k }^{\alpha,\lambda }x(k) \ast y(k)} \big\}\\
 =&\hspace{-6pt} {{\mathscr N}_a}\big\{ {{}_a^{\rm G}\nabla _{k }^{\alpha,\lambda }x(k)} \big\}{{\mathscr N}_a} \{ {y(k)} \}\\
 =&\hspace{-6pt} {\big( {\frac{{s - \lambda }}{{1 - \lambda }}} \big)^{\alpha }}{{\mathscr N}_a} \{ {x(k)} \}{{\mathscr N}_a} \{ {y(k)}  \}.
\end{array}}
\end{equation}

With the help of the uniqueness of nabla Laplace transform (see \cite[Lemma 2]{Wei:2019FDTA}), the result in (\ref{Eq3.103}) can be obtained smoothly.
\end{proof}

Actually, (\ref{Eq3.103}) can be rewritten as
\begin{equation}\label{Eq3.106}
{\textstyle x(k) \ast {}_a^{\rm G}\nabla _{k }^{\alpha,\lambda }y(k) = {}_a^{\rm G}\nabla _{k }^{\alpha,\lambda }x(k) \ast y(k).}
\end{equation}
Similarly, when $\alpha=0$, (\ref{Eq3.106}) holds naturally. When $\alpha>0$, (\ref{Eq3.106}) is revelent to the Gr\"{u}nwald--Letnikov tempered difference. When $\alpha<0$, (\ref{Eq3.106}) is revelent to the Gr\"{u}nwald--Letnikov tempered sum.

\begin{theorem}\label{Theorem3.17}
For any $\alpha\in(n-1,n)$, $n\in\mathbb{Z}_+$, $\lambda\neq1$, $x,y:\mathbb{N}_{a-n+1}\to\mathbb{R}$, $a\in\mathbb{R}$, one has
\begin{equation}\label{Eq3.107}
\hspace{-6pt}
{\textstyle
\begin{array}{l}
\sum\nolimits_{j = a + 1}^k {x(k + a + 1 - j)\nabla^{n,\lambda} y(j)}\\  = \sum\nolimits_{j = a + 1}^k {\nabla ^{n,\lambda} x(j)y ( {k + a + 1 - j}  )}
 + \sum\nolimits_{i = 0}^{n - 1} {[ {\nabla^{i,\lambda }x ( {k + a - j}  )\nabla ^{n  - i - 1,\lambda}y(j)} ]_{j = a}^{j = k}},
\end{array}}
\end{equation}
\begin{equation}\label{Eq3.108}
\hspace{-18pt}
{\textstyle
\begin{array}{l}
\sum\nolimits_{j = a + 1}^k {x(k + a + 1 - j){}_a^{\rm R}\nabla _{j }^{\alpha,\lambda} y(j)}\\  = \sum\nolimits_{j = a + 1}^k {{}_a^{\rm C}\nabla _{j }^{\alpha,\lambda} x(j)y ( {k + a + 1 - j}  )}
 + \sum\nolimits_{i = 0}^{n - 1} {[ {\nabla ^{i,\lambda }x ( {k + a - j}  ){}_a^{\rm R}\nabla _{j }^{\alpha  - i - 1,\lambda}y(j)} ]_{j = a}^{j = k}}.
\end{array}}
\end{equation}
\end{theorem}
\begin{proof}
By using Theorem \ref{Theorem3.15}, one has
\begin{equation}\label{Eq3.109}
{\textstyle
\begin{array}{l}
{{\mathscr N}_a}\big\{ {\sum\nolimits_{j = a + 1}^k {x(k + a + 1 - j){\nabla ^{n,\lambda }}y(j)} } \big\}\\
={{\mathscr N}_a} \{ {x(k)}  \}{{\mathscr N}_a} \{ {{\nabla ^{n,\lambda }}y(k)}  \}\\
 = {{\mathscr N}_a} \{ {x(k)}  \}\big\{ {{{\big( {\frac{{s - \lambda }}{{1 - \lambda }}} \big)}^n}{{\mathscr N}_a} \{ {y(k)}  \} - \sum\nolimits_{i = 0}^{n - 1} {{{\big( {\frac{{s - \lambda }}{{1 - \lambda }}} \big)}^{n - i - 1}}{{ [ {{\nabla ^{i,\lambda }}y(k)}  ]}_{k = a}}} } \big\},
\end{array}}
\end{equation}
\begin{equation}\label{Eq3.110}
{\textstyle
\begin{array}{l}
{{\mathscr N}_a}\big\{ {\sum\nolimits_{j = a + 1}^k {{\nabla ^{n,\lambda }}x(k + a + 1 - j)y(j)} } \big\}\\
 = {{\mathscr N}_a}\big\{ {{\nabla ^{n,\lambda }}x(k)} \big\}{{\mathscr N}_a} \{ {y(k)}  \}\\
 = \big\{ {{{\big( {\frac{{s - \lambda }}{{1 - \lambda }}} \big)}^n}{{\mathscr N}_a} \{ {x(k)}  \} - \sum\nolimits_{i = 0}^{n - 1} {{{\big( {\frac{{s - \lambda }}{{1 - \lambda }}} \big)}^{n - i - 1}}{{ [ {{\nabla ^{i,\lambda }}x(k)}  ]}_{k = a}}} } \big\}{{\mathscr N}_a} \{ {y(k)}  \},
\end{array}}
\end{equation}
\begin{equation}\label{Eq3.111}
{\textstyle
\begin{array}{l}
{{\mathscr N}_a}\big\{ {\sum\nolimits_{i = 0}^{n - 1} { [ {{\nabla ^{i,\lambda }}x ( {k + a - j} ){\nabla ^{n - i - 1,\lambda }}y(j)}  ]_{j = a}^{j = k}} } \big\}\\
 = {{\mathscr N}_a}\big\{ {\sum\nolimits_{i = 0}^{n - 1} {{{ [ {{\nabla ^{i,\lambda }}x(k)}  ]}_{k = a}}{\nabla ^{n - i - 1,\lambda }}y(k)} } \big\}\\
 \hspace{12pt} - {{\mathscr N}_a}\big\{ {{{\sum\nolimits_{i = 0}^{n - 1} {{\nabla ^{i,\lambda }}x(k) [ {{\nabla ^{n - i - 1,\lambda }}y(k)}  ]} }_{k = a}}} \big\}\\
 = {{\mathscr N}_a}\big\{ {\sum\nolimits_{i = 0}^{n - 1} {{{ [ {{\nabla ^{i,\lambda }}x(k)}  ]}_{k = a}}{\nabla ^{n - i - 1,\lambda }}y(k)} } \big\}\\
 \hspace{12pt}  - {{\mathscr N}_a}\big\{ {\sum\nolimits_{i = 0}^{n - 1} {{\nabla ^{n - i - 1,\lambda }}x(k){{ [ {{\nabla ^{i,\lambda }}y(k)}  ]}_{k = a}}} } \big\}\\
 = \sum\nolimits_{i = 0}^{n - 1} {{{ [ {{\nabla ^{i,\lambda }}x(k)}  ]}_{k = a}}{{\mathscr N}_a} \{ {{\nabla ^{n - i - 1,\lambda }}y(k)}  \}} \\
 \hspace{12pt}  - \sum\nolimits_{i = 0}^{n - 1} {{{\mathscr N}_a} \{ {{\nabla ^{n - i - 1,\lambda }}x(k)} \}{{ [ {{\nabla ^{i,\lambda }}y(k)}  ]}_{k = a}}} \\
 = {{\mathscr N}_a} \{ {y(k)}  \}\sum\nolimits_{i = 0}^{n - 1} {{{\big( {\frac{{s - \lambda }}{{1 - \lambda }}} \big)}^{n - i - 1}}{{ [ {{\nabla ^{i,\lambda }}x(k)}  ]}_{k = a}}} \\
 \hspace{12pt} - \frac{1}{{1 - \lambda }}\sum\nolimits_{i = 0}^{n - 1} {{{\sum\nolimits_{j = 0}^{n - i - 2} {{\big( {\frac{{s - \lambda }}{{1 - \lambda }}} \big)^{n - i - 2 - j}} [ {{\nabla ^{j,\lambda }}y(k)} ]} }_{k = a}}{{ [ {{\nabla ^{i,\lambda }}x(k)}  ]}_{k = a}}} \\
 \hspace{12pt} - {{\mathscr N}_a} \{ {x(k)}  \}\sum\nolimits_{i = 0}^{n - 1} {{{\big( {\frac{{s - \lambda }}{{1 - \lambda }}} \big)}^{n - i - 1}}{{ [ {{\nabla ^{i,\lambda }}y(k)}  ]}_{k = a}}} \\
 \hspace{12pt} + \frac{1}{{1 - \lambda }}\sum\nolimits_{i = 0}^{n - 1} {{{\sum\nolimits_{j = 0}^{n - i - 2} {{\big( {\frac{{s - \lambda }}{{1 - \lambda }}} \big)^{n - i - 2 - j}} [ {{\nabla ^{j,\lambda }}x(k)} ]} }_{k = a}}{{ [ {{\nabla ^{i,\lambda }}y(k)}  ]}_{k = a}}} \\
 = {{\mathscr N}_a} \{ {y(k)}  \}\sum\nolimits_{i = 0}^{n - 1} {{{\big( {\frac{{s - \lambda }}{{1 - \lambda }}} \big)}^{n - i - 1}}{{ [ {{\nabla ^{i,\lambda }}x(k)}  ]}_{k = a}}} \\
 \hspace{12pt}- {{\mathscr N}_a} \{ {x(k)}  \}\sum\nolimits_{i = 0}^{n - 1} {{{\big( {\frac{{s - \lambda }}{{1 - \lambda }}} \big)}^{n - i - 1}}{{ [ {{\nabla ^{i,\lambda }}y(k)}  ]}_{k = a}}},
\end{array}}
\end{equation}
where $\sum\nolimits_{i = 0}^{n - 1} {\sum\nolimits_{j = 0}^{n - i - 2} {} }  = \sum\nolimits_{i = 0}^{n - 2} {\sum\nolimits_{j = 0}^{n - i - 2} {} } $ is adopted. Combining (\ref{Eq3.109}) - (\ref{Eq3.111}) with the uniqueness of nabla Laplace transform, the result in (\ref{Eq3.107}) follows.

Owing to the convolution operation, (\ref{Eq3.107}) can be rewritten as
\begin{equation}\label{Eq3.112}
{\textstyle
\begin{array}{rl}
x(k) \ast {\nabla ^{n,\lambda }}y(k) =&\hspace{-6pt} {\nabla ^{n,\lambda }}x(k) \ast y(k)  + \sum\nolimits_{i = 0}^{n - 1} {[ {{\nabla ^{i,\lambda }}x ( {k + a - j}  ){\nabla ^{n - i - 1,\lambda }}y(j)} ]_{j = a}^{j = k}}.
\end{array}}
\end{equation}

Likewise, \cite[Lemma 6]{Wei:2019FDTA}, (\ref{Eq3.92}) and (\ref{Eq3.93}) lead to
\begin{eqnarray}\label{Eq3.113}
\begin{array}{rl}
\sum\nolimits_{j = a + 1}^k {x(k + a + 1 - j){}_a^{\rm R}\nabla _{j }^{\alpha,\lambda} y(j)} =&\hspace{-6pt}x(k) \ast {}_a^{\rm R}\nabla _k^{\alpha ,\lambda }y(k)\\
 =&\hspace{-6pt} x(k) \ast {\nabla ^{n,\lambda }}{}_a^{\rm G}\nabla _k^{\alpha  - n,\lambda }y(k)\\
 =&\hspace{-6pt} {\nabla ^{n,\lambda }}x(k) \ast {}_a^{\rm G}\nabla _k^{\alpha  - n,\lambda }y(k)\\
 &\hspace{-6pt}  + \sum\nolimits_{i = 0}^{n - 1} {[ {{\nabla ^{i,\lambda }}x ( {k + a - j}  ){\nabla ^{n - i - 1,\lambda }}{}_a^{\rm G}\nabla _k^{\alpha  - n,\lambda }y(j)} ]_{j = a}^{j = k}} \\
 =&\hspace{-6pt} {}_a^{\rm G}\nabla _k^{\alpha  - n,\lambda }{\nabla ^{n,\lambda }}x(k) \ast y(k)\\
 &\hspace{-6pt}  + \sum\nolimits_{i = 0}^{n - 1} {[ {{\nabla ^{i,\lambda }}x ( {k + a - j}  ){\nabla ^{n - i - 1,\lambda }}{}_a^{\rm G}\nabla _k^{\alpha  - n,\lambda }y(j)} ]_{j = a}^{j = k}} \\
 =&\hspace{-6pt} {}_a^{\rm C}\nabla _k^{\alpha ,\lambda }x(k)\ast y(k) + \sum\nolimits_{i = 0}^{n - 1} {[ {{\nabla ^{i,\lambda }}x ( {k + a - j}  ){}_a^{\rm R}\nabla _j^{\alpha  - i - 1,\lambda }y(j)} ]_{j = a}^{j = k}}\\
 =&\hspace{-6pt} \sum\nolimits_{j = a + 1}^k {{}_a^{\rm C}\nabla _{j }^{\alpha,\lambda} x(j)y ( {k + a + 1 - j}  )} \\
 &\hspace{-6pt} + \sum\nolimits_{i = 0}^{n - 1} {[ {\nabla ^{i,\lambda }x ( {k + a - j}  ){}_a^{\rm R}\nabla _{j }^{\alpha  - i - 1,\lambda}y(j)} ]_{j = a}^{j = k}}.
\end{array}
\end{eqnarray}
All of these complete the proof.
\end{proof}

Theorem \ref{Theorem3.16} and Theorem \ref{Theorem3.17} are inspired by \cite[Theorem 1]{Wei:2021Auto}, which provides an avenue to calculate tempered fractional difference/sum.
\section{Simulation Study}\label{Section4}
Letting $x(k)=\sin (10 k)$, $a=0$, $k\in\mathbb{N}_{a+1}^{100}$, then the Gr\"{u}nwald-Letnikov difference/sum ${ }_{a}^{\rm G} \nabla_{k}^{\alpha, w(k)} x(k)$ can be calculated with different $w(k)$. Setting $w(k)=0.5^{k-a}$ or $w(k)=0.5^{k-a}+0.01$, the simulated results are shown are Fig. \ref{Fig1}. It can be observed from Fig. \ref{Fig1a} that ${ }_{a}^{\rm G} \nabla_{k}^{\alpha, w(k)} x(k)$ diverges as $k\to+\infty$. More specially, when $\alpha>0$, the result tends to $+\infty$. When $\alpha<0$, the result tends to $-\infty$. In this case, ${}_a^{\rm{G}}\nabla _k^{\alpha ,w(k)}x(k) = \sum\nolimits_{i = 0}^{k - a - 1} {\frac{{(i + 1)\overline {^{ - \alpha  - 1}} }}{{\Gamma ( - \alpha )}}{2^i}\sin (10k - 10i)} $. Notably, ${2^{k - a - 1}} \to  + \infty $ as $k \to  + \infty $, $\sin (10k - 10i) = \sin (10) < 0$ when $i = k - a - 1$. If $\alpha\in(0,1)$, $\frac{{(i + 1)\overline {^{ - \alpha  - 1}} }}{{\Gamma ( - \alpha )}} < 0$ for $i\in\mathbb{N}_0^{k-a-1}$. If $\alpha\in(-1,0)$, $\frac{{(i + 1)\overline {^{ - \alpha  - 1}} }}{{\Gamma ( - \alpha )}} > 0$ for $i\in\mathbb{N}_0^{k-a-1}$. From this, the trend of ${ }_{a}^{\rm G} \nabla_{k}^{\alpha, w(k)} x(k)$ as $k\to+\infty$ can be derived. The typical feature of $w(k)=0.5^{k-a}$ is that $w(k)$ declines rapidly and converges to $0$ while the tempered function $w(k)$ is assumed to be nonzero. Consequently, a nonzero number is introduced, namely, $w(k)=0.5^{k-a}+0.01$. Then, bounded ${ }_{a}^{\rm G} \nabla_{k}^{\alpha, w(k)} x(k)$ follows in Fig. \ref{Fig1b}. Actually, with the increase of $k$, $0.01$ will play a greater role than $0.5^{k-a}$. ${ }_{a}^{\rm G} \nabla_{k}^{\alpha, w(k)} x(k)$ performs like ${ }_{a}^{\rm G} \nabla_{k}^{\alpha} x(k)$ in the steady stage and the tempered case degenerates into the classical case.
\begin{figure}[!htbp]
\centering
\setlength{\abovecaptionskip}{-2pt}
	\vspace{-10pt}
	\subfigtopskip=-2pt
	\subfigbottomskip=2pt
	\subfigcapskip=-10pt
\subfigure[$w(k)=0.5^{k-a}$]{
\begin{minipage}[t]{0.48\linewidth}
\includegraphics[width=1.0\hsize]{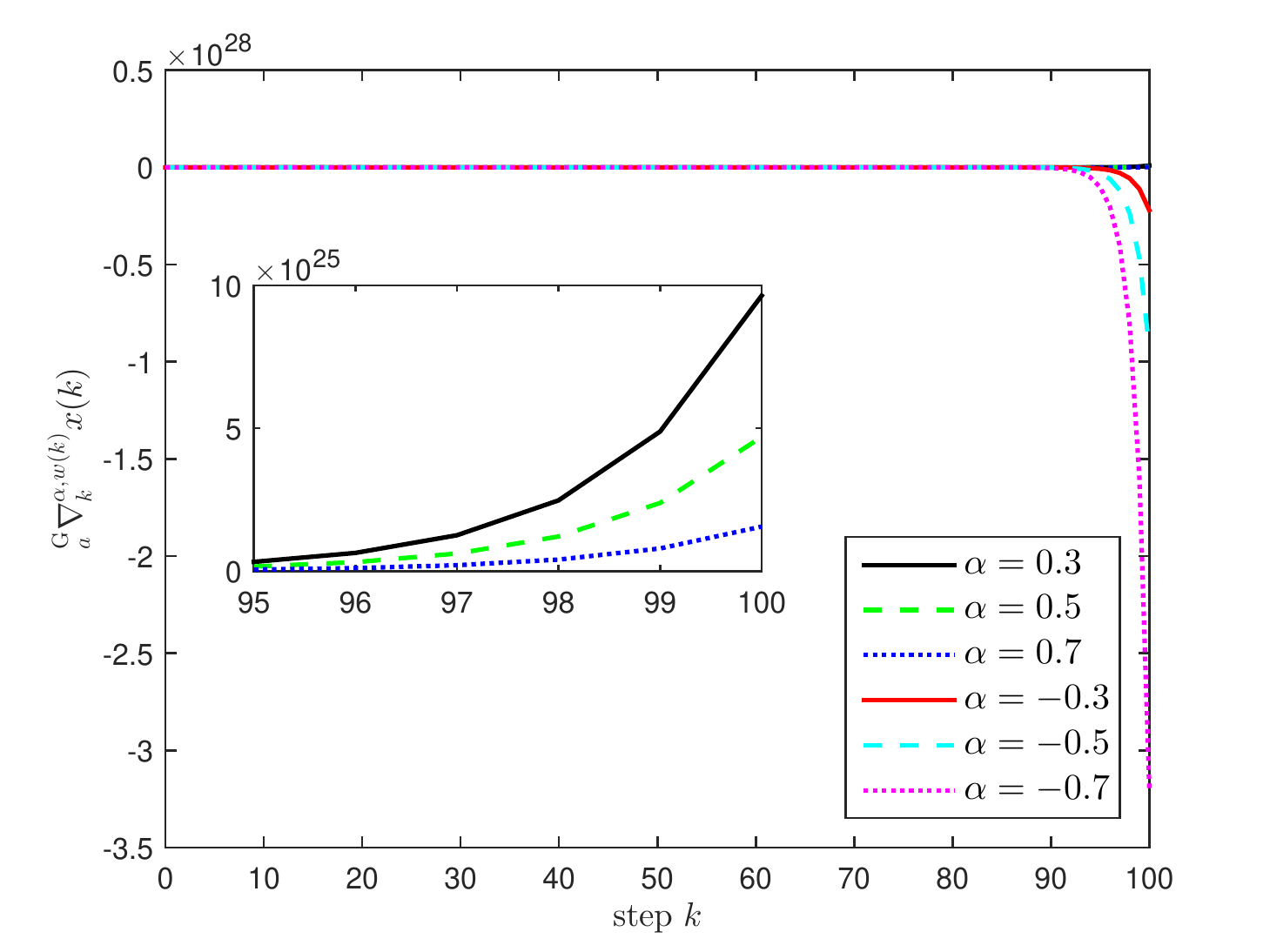}
\label{Fig1a}
\end{minipage}%
}
\subfigure[$w(k)=0.5^{k-a}+0.01$]{
\begin{minipage}[t]{0.48\linewidth}
\includegraphics[width=1.0\hsize]{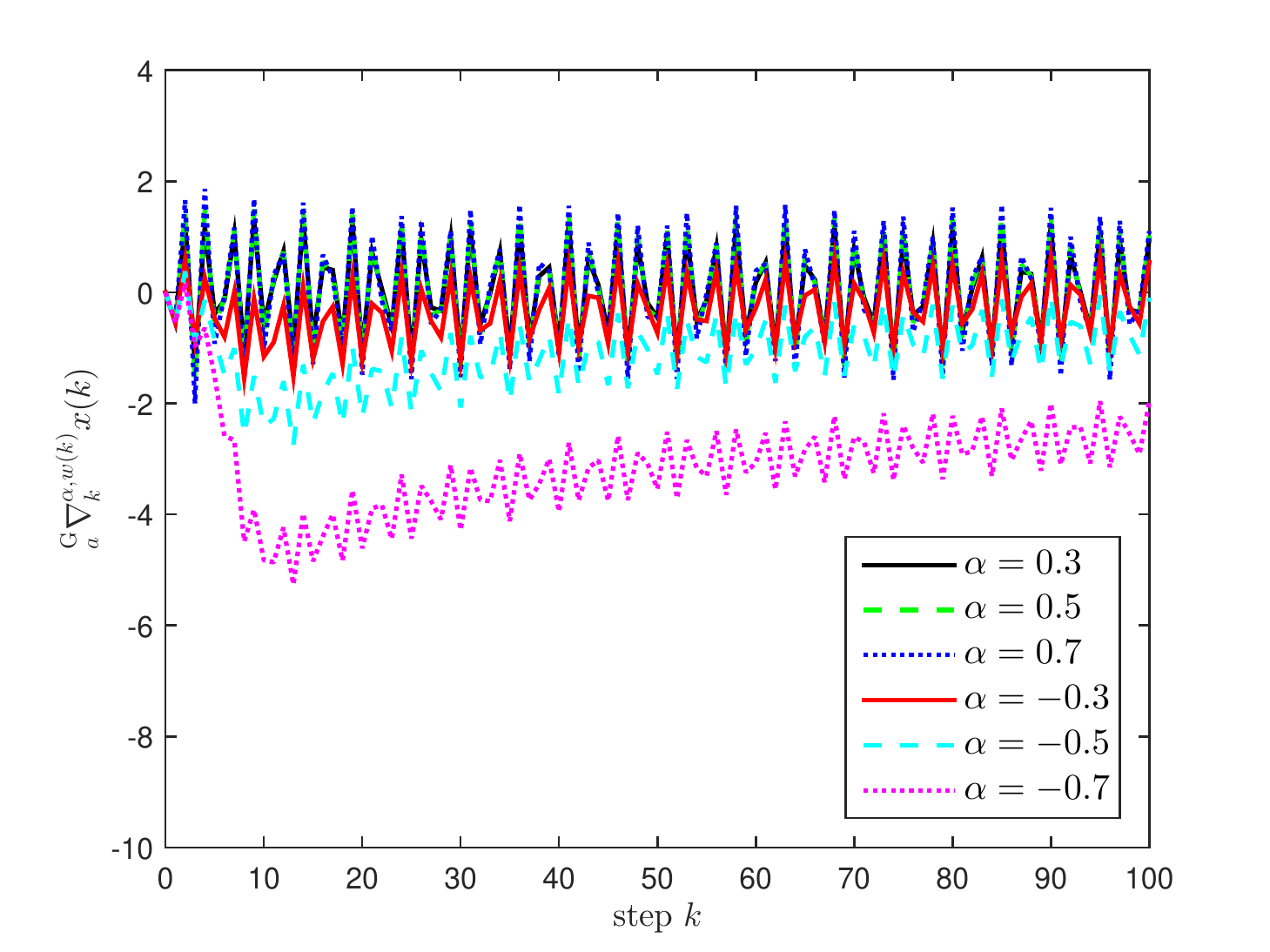}
\label{Fig1b}
\end{minipage}%
}
\caption{Time evolution of ${}_a^{\rm G}\nabla_k^{\alpha,w(k)} x(k)$.}
\label{Fig1}
\end{figure}

To show more details, the following four cases are considered.
\[\left\{ \begin{array}{l}
{\rm{case}}\;1:\;w(k) = {\sqrt 2 ^{k - a}};\\
{\rm{case}}\;2:\;w(k) =  - {\pi ^{k - a}};\\
{\rm{case}}\;3:\;w(k) = {( - 1)^{k - a}};\\
{\rm{case}}\;4:\;w(k) = \sin (\frac{{k\pi }}{2} - \frac{{a\pi }}{2} + \frac{\pi }{4}).
\end{array} \right.\]
The obtained results are shown as Fig. \ref{Fig2}. In case 1, $w(k)$ is positive and monotonically increasing. In case 2, $w(k)$ is negative and monotonically decreasing. In case 3, $w(k)$ is in oscillation. Its positive value and negative value appear alternatively. In case 4, $w(k)$ is also in oscillation. Two positive value will appear after two negative value. It can be found that no matter for the fractional difference or the fractional sum, different $w(k)$ correspond to different ${ }_{a}^{\rm G} \nabla_{k}^{\alpha, w(k)} x(k)$ in Fig. \ref{Fig2}. To clear show the difference, setting case 1 as the standard, the error of ${ }_{a}^{\rm G} \nabla_{k}^{\alpha, w(k)} x(k)$ between each case with the standard is calculated and displayed in Fig. \ref{Fig3}. It is shown that the error always exists and cannot be ignored. Along this way, different tempered functions can be constructed for different scenarios and satisfied different demands, which could provide great potential to apply nabla tempered fractional calculus.

\begin{figure}[!htbp]
\centering
\setlength{\abovecaptionskip}{-2pt}
	\vspace{-10pt}
	\subfigtopskip=-2pt
	\subfigbottomskip=2pt
	\subfigcapskip=-10pt
\subfigure[$\alpha=0.5$]{
\begin{minipage}[t]{0.48\linewidth}
\includegraphics[width=1.0\hsize]{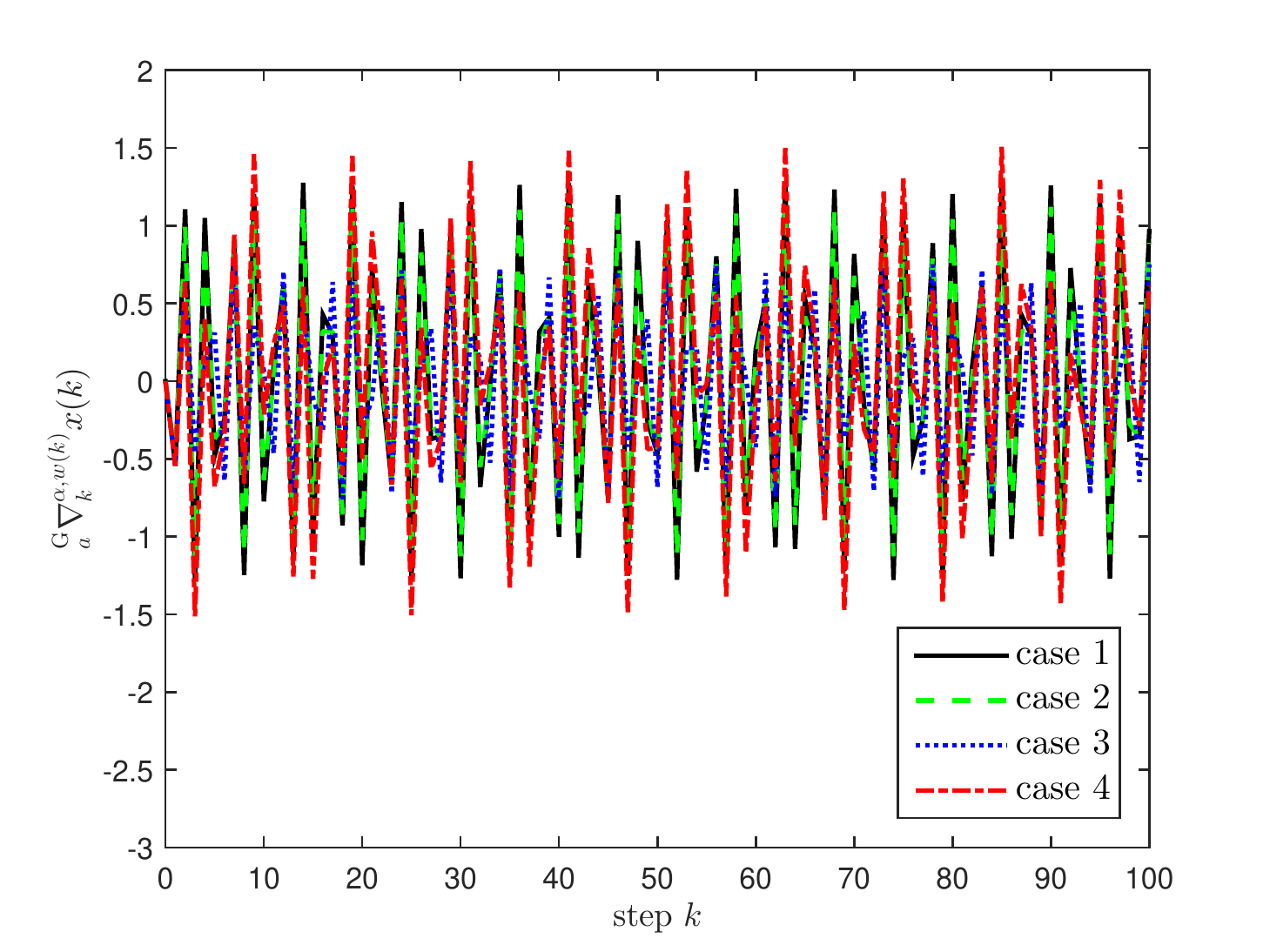}
\label{Fig2a}
\end{minipage}%
}
\subfigure[$\alpha=-0.5$]{
\begin{minipage}[t]{0.48\linewidth}
\includegraphics[width=1.0\hsize]{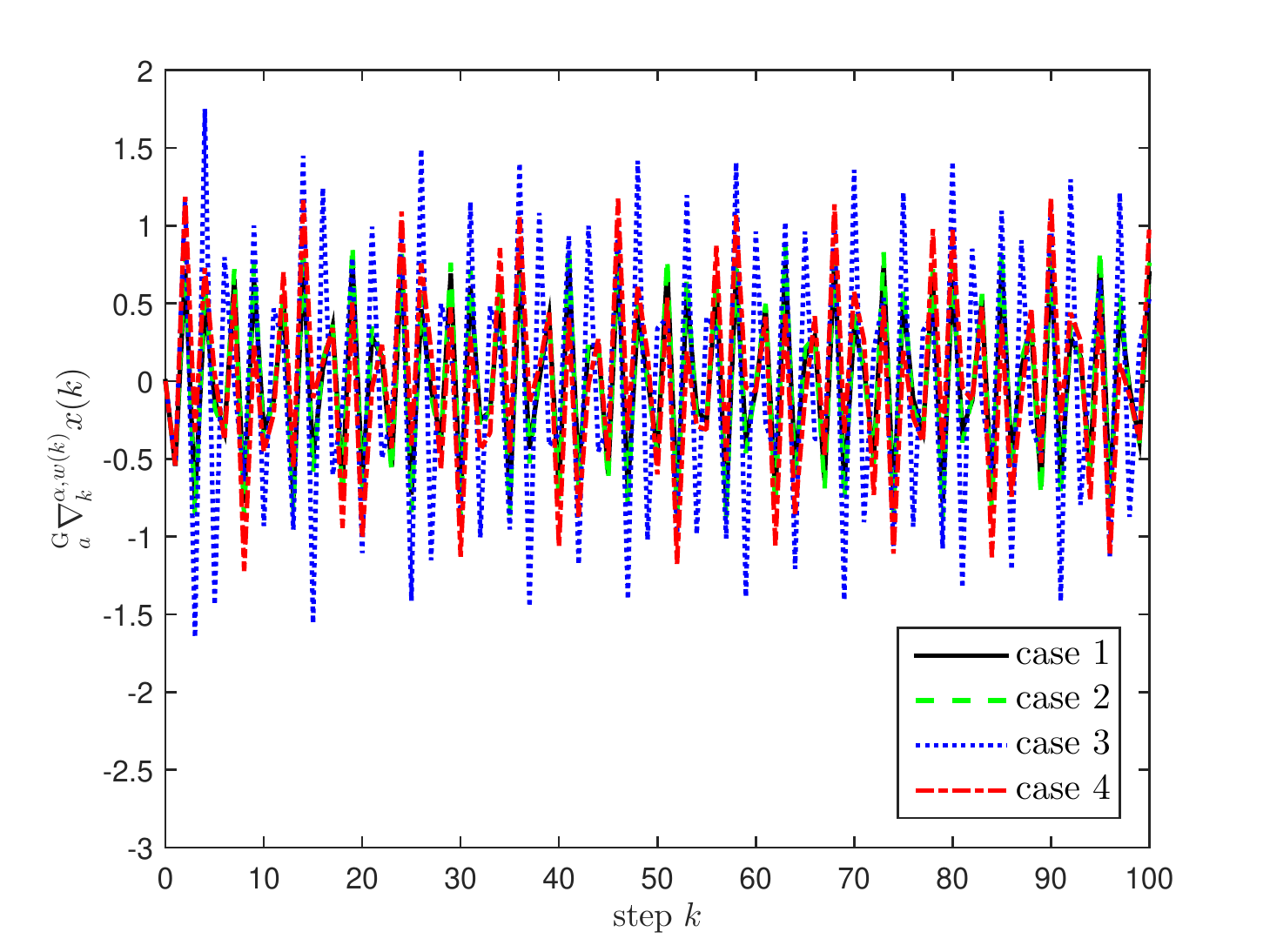}
\label{Fig2b}
\end{minipage}%
}
\centering
\caption{Time evolution of ${}_a^{\rm G}\nabla_k^{\alpha,w(k)} x(k)$.}
\label{Fig2}
\end{figure}

\begin{figure}[!htbp]
\centering
\setlength{\abovecaptionskip}{-2pt}
	\vspace{-10pt}
	\subfigtopskip=-2pt
	\subfigbottomskip=2pt
	\subfigcapskip=-10pt
\subfigure[$\alpha=0.5$]{
\begin{minipage}[t]{0.48\linewidth}
\includegraphics[width=1.0\hsize]{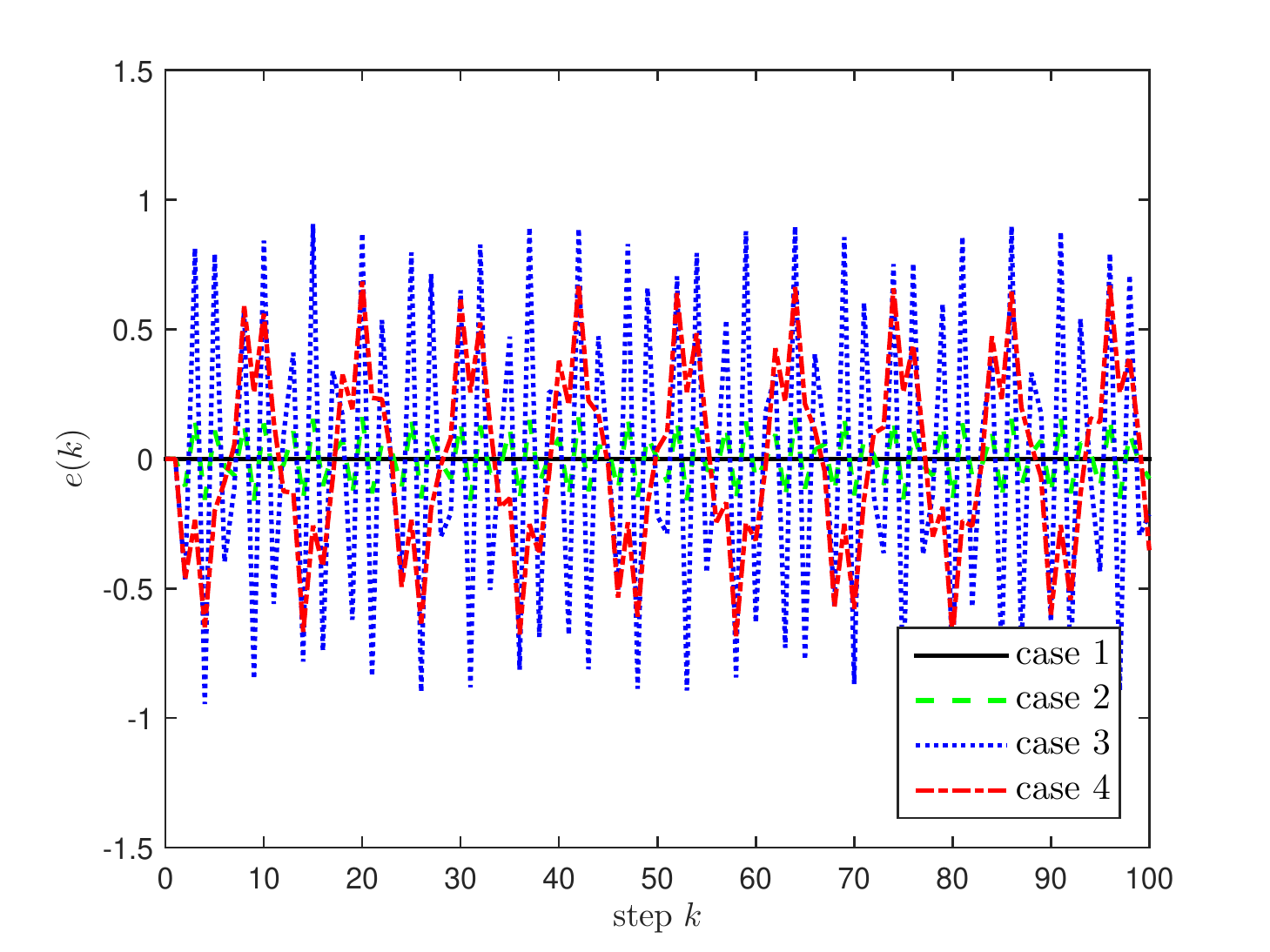}
\label{Fig3a}
\end{minipage}%
}
\subfigure[$\alpha=-0.5$]{
\begin{minipage}[t]{0.48\linewidth}
\includegraphics[width=1.0\hsize]{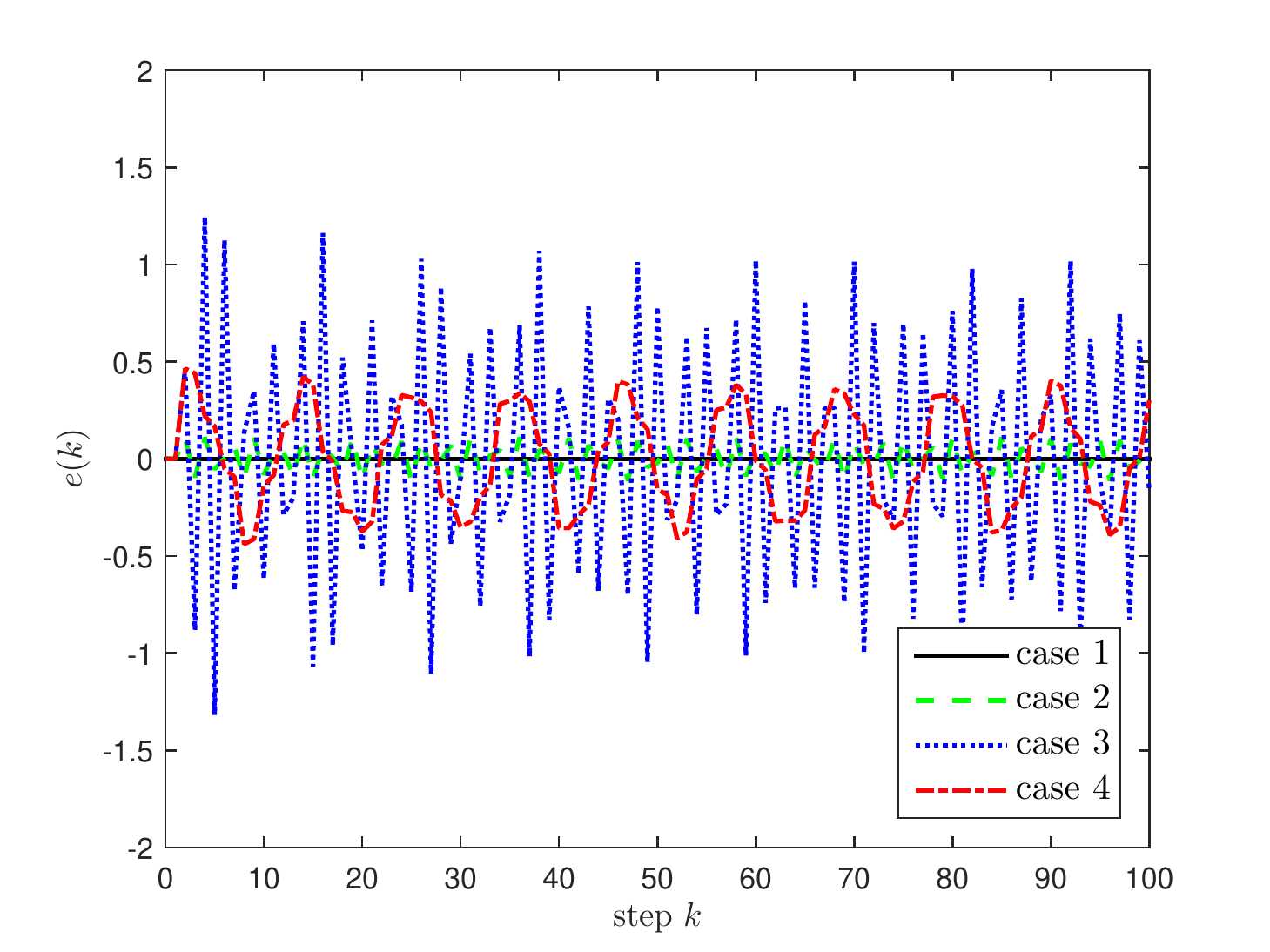}
\label{Fig3b}
\end{minipage}%
}
\centering
\caption{The error of ${}_a^{\rm G}\nabla_k^{\alpha,w(k)} x(k)$ in different case.}
\label{Fig3}
\end{figure}

To further check the value of tempered fractional difference, selecting $x(k)=\sin (10 k)$, $\alpha=0.5$ and the above mentioned four cases, the results are shown as Fig. \ref{Fig4}.
\begin{figure}[!htbp]
\centering
\setlength{\abovecaptionskip}{-2pt}
	\vspace{-10pt}
	\subfigtopskip=-2pt
	\subfigbottomskip=2pt
	\subfigcapskip=-10pt
\subfigure[$w(k) = {\sqrt 2 ^{k - a}}$]{
\begin{minipage}[t]{0.48\linewidth}
\includegraphics[width=1.0\hsize]{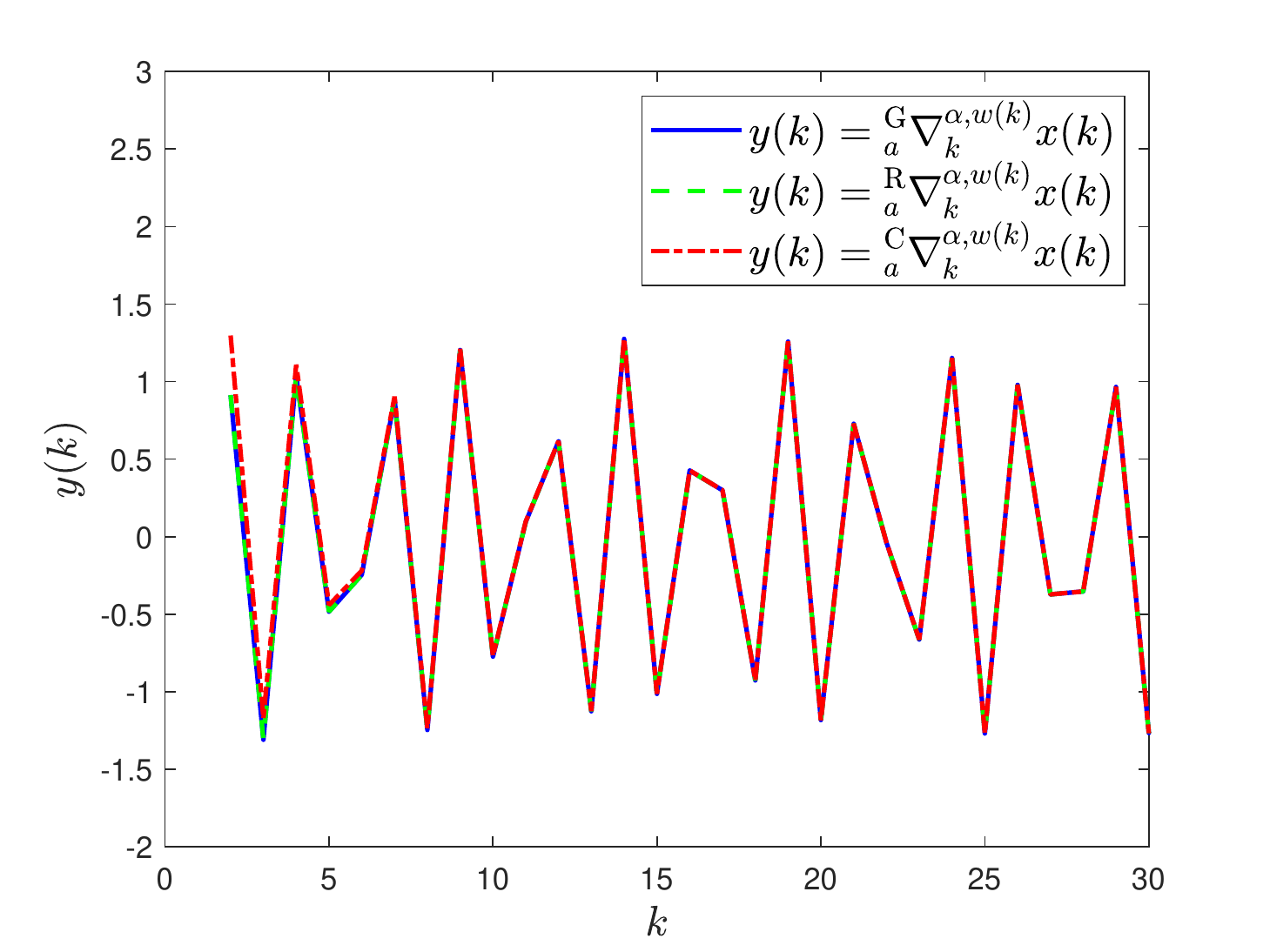}
\label{Fig4a}
\end{minipage}%
}
\subfigure[$w(k) =  - {\pi ^{k - a}}$]{
\begin{minipage}[t]{0.48\linewidth}
\includegraphics[width=1.0\hsize]{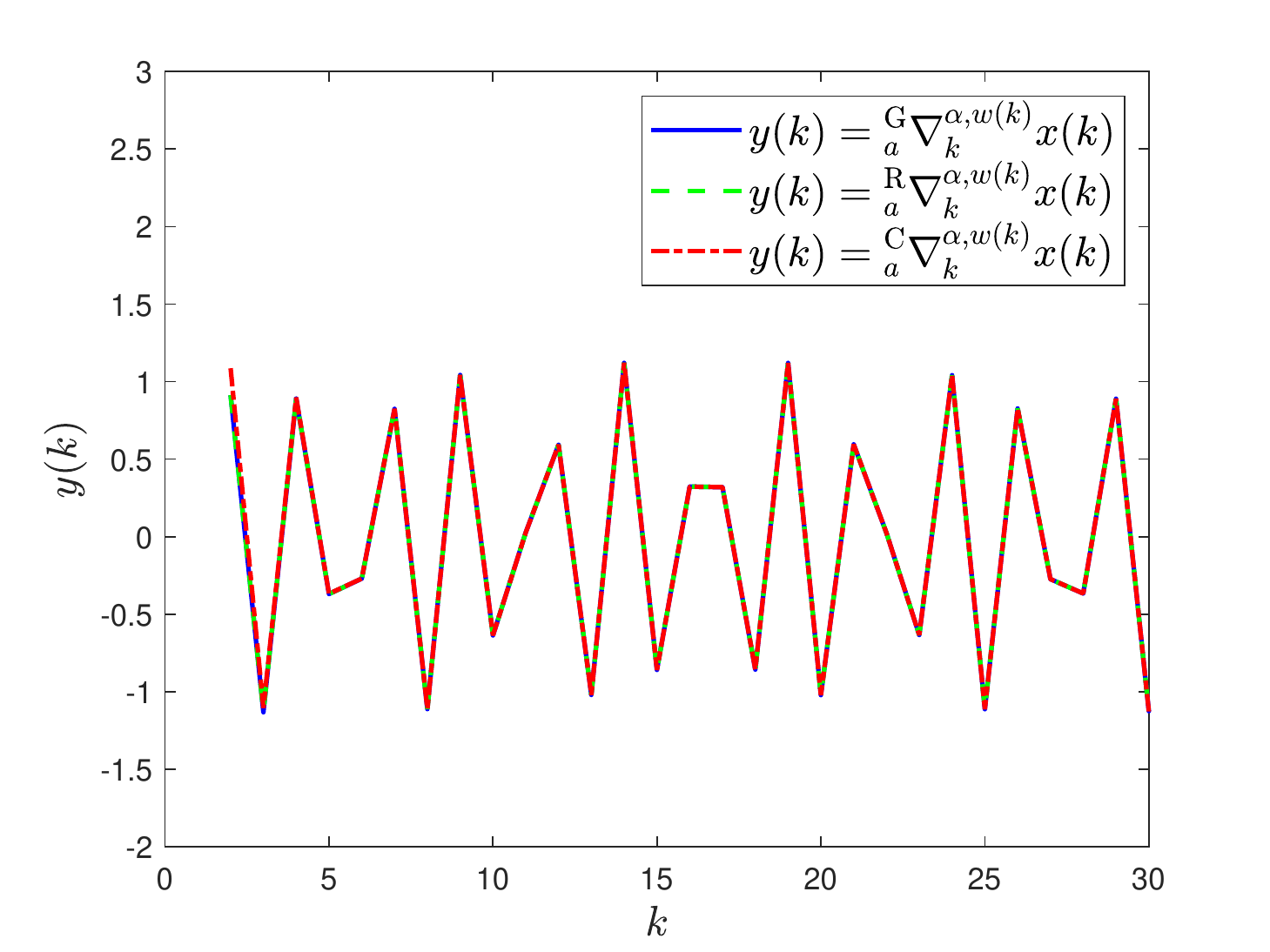}
\label{Fig4b}
\end{minipage}%
}\\
\subfigure[$w(k) = {( - 1)^{k - a}}$]{
\begin{minipage}[t]{0.48\linewidth}
\includegraphics[width=1.0\hsize]{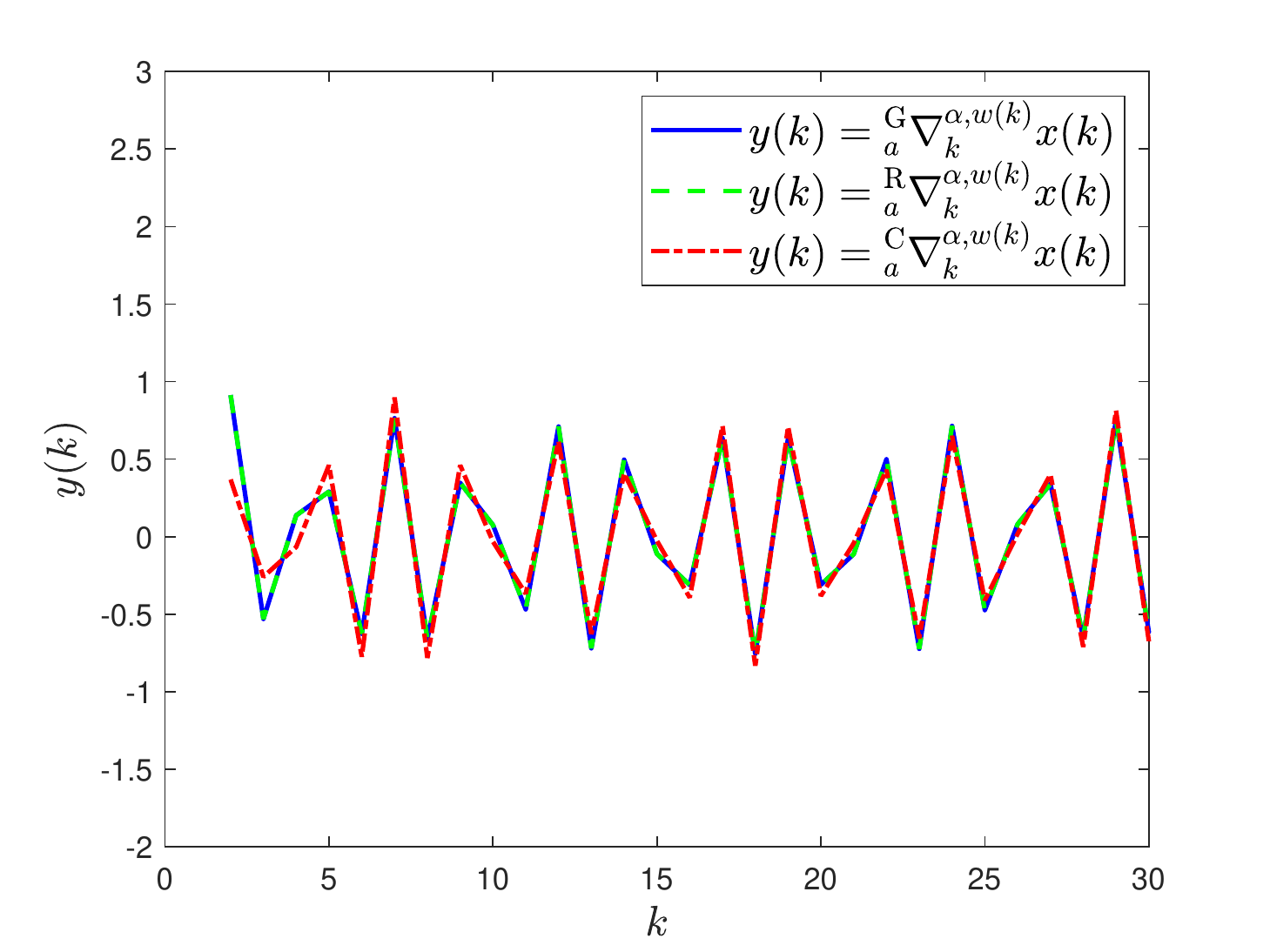}
\label{Fig4c}
\end{minipage}%
}
\subfigure[$w(k) = \sin (\frac{{k\pi }}{2} - \frac{{a\pi }}{2} + \frac{\pi }{4})$]{
\begin{minipage}[t]{0.48\linewidth}
\includegraphics[width=1.0\hsize]{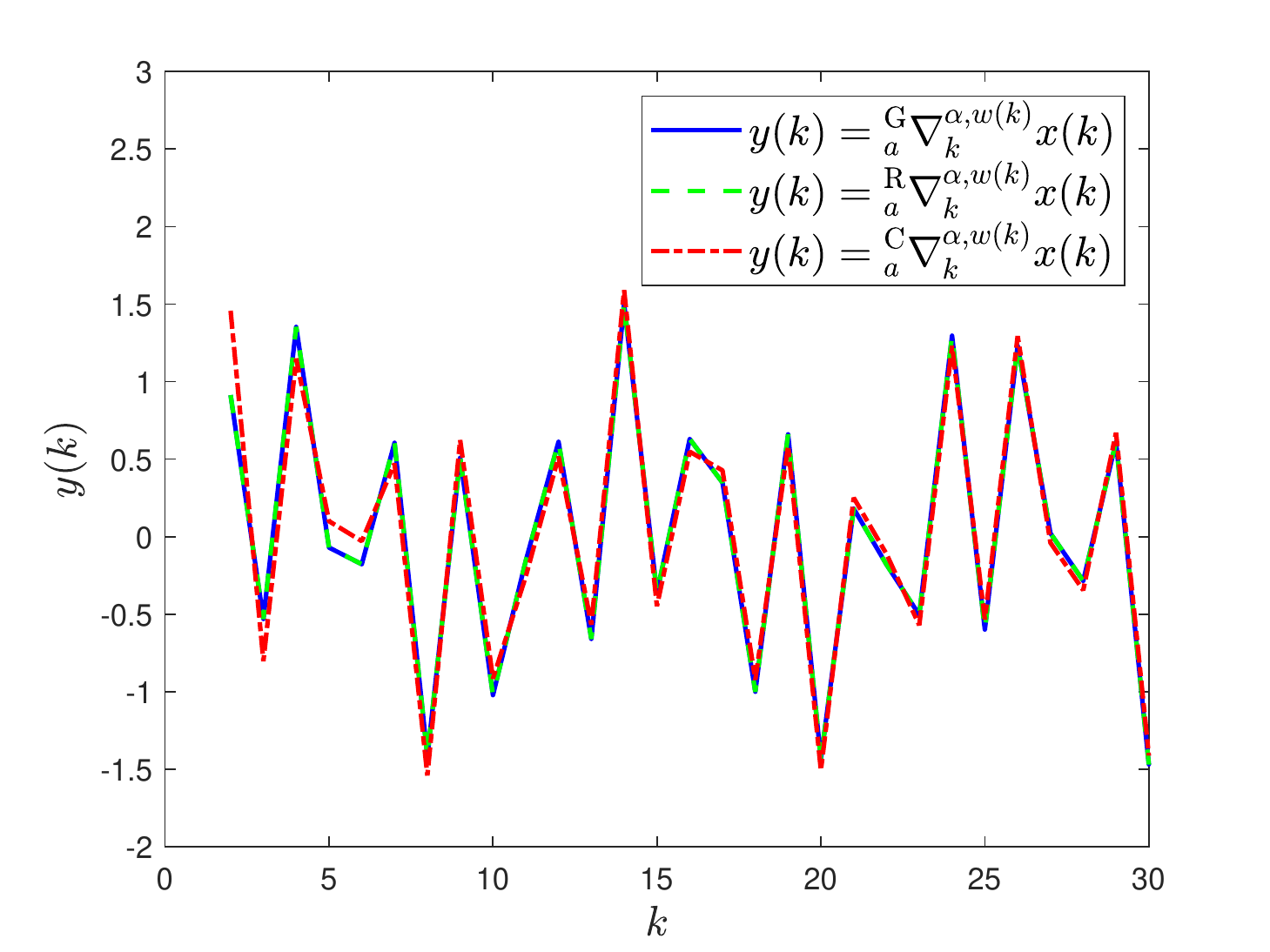}
\label{Fig4d}
\end{minipage}%
}
\centering
\caption{Time evolution of ${}_a^{\rm G}\nabla_k^{\alpha,w(k)} x(k)$, ${}_a^{\rm R}\nabla_k^{\alpha,w(k)} x(k)$ and ${}_a^{\rm C}\nabla_k^{\alpha,w(k)} x(k)$.}
\label{Fig4}
\end{figure}

It can be observed that for each $w(k)$, the difference of the three fractional differences is small. From the quantitative analysis, one has
\[\left\{ \begin{array}{l}
{\rm{case}}\;1:\;-4.4409\times10^{-16}\le {}_a^{\rm{G}}\nabla _k^{\alpha,w(k)} x(k)-{}_a^{\rm{R}}\nabla _k^{\alpha,w(k)} x(k)\le 1.1102\times10^{-15};\\
{\rm{case}}\;2:\;-8.8818\times10^{-16}\le {}_a^{\rm{G}}\nabla _k^{\alpha,w(k)} x(k)-{}_a^{\rm{R}}\nabla _k^{\alpha,w(k)} x(k)\le 3.3307\times10^{-16};\\
{\rm{case}}\;3:\;-1.0131\times10^{-16}\le {}_a^{\rm{G}}\nabla _k^{\alpha,w(k)} x(k)-{}_a^{\rm{R}}\nabla _k^{\alpha,w(k)} x(k)\le 4.9960\times10^{-16};\\
{\rm{case}}\;4:\;-2.2204\times10^{-16}\le {}_a^{\rm{G}}\nabla _k^{\alpha,w(k)} x(k)-{}_a^{\rm{R}}\nabla _k^{\alpha,w(k)} x(k)\le 4.4409\times10^{-16},
\end{array} \right.\]
which demonstrates the correctness of (\ref{Eq3.1}).

\section{Conclusions}\label{Section5}
In this paper, the nabla tempered fractional difference and fractional sum have been investigated originally. A series of analytic properties are developed with rigorous mathematical proof including the equivalence relation, the nabla Taylor formula, and the nabla Laplace transform. It is believed that this work reveals the relationship between time domain and frequency domain of nabla tempered fractional calculus and enriches the knowledge of nabla tempered fractional calculus.

\section*{Acknowledgment}
The work described in this paper was fully supported by the National Natural Science Foundation of China (62273092), the National Key R$\&$D Project of China (2020YFA0714300) and ZhiShan Youth Scholar Program of Southeast University.

%

\phantomsection
\addcontentsline{toc}{section}{References}
\section*{References}
\bibliographystyle{model1-num-names}
\bibliography{database}

\end{document}